\documentclass[12pt]{article}
\usepackage{amssymb}
\usepackage{amsmath}
\usepackage{amsthm}
\usepackage{color}
\usepackage{comment}
\usepackage{enumitem}
\usepackage{graphicx}
\begin{document}

\def\cO{\mathcal{O}}
\def\cS{\mathcal{S}}
\def\cX{\mathcal{X}}
\def\cY{\mathcal{Y}}
\def\ee{\varepsilon}
\def\sJ{J}
\def\sL{L}
\def\sR{R}

\newcommand{\removableFootnote}[1]{}

\newtheorem{theorem}{Theorem}[section]
\newtheorem{lemma}[theorem]{Lemma}
\newtheorem{proposition}[theorem]{Proposition}

\title{
Scaling laws for large numbers of coexisting attracting periodic solutions
in the border-collision normal form.
%emanating from border-collision bifurcations.
}
\author{
D.J.W.~Simpson\\
Institute of Fundamental Sciences\\
Massey University\\
Palmerston North\\
New Zealand}
\maketitle

%	To do:

%			once accepted and final edits made:
%					spellcheck again
%					remove picture commands
%					save as ``final'' version
%					remove unnecessary code:
%							- removable footnotes
%							- last appendix
%					    - other text that has been commented out
%					resend files

% notes:
%			always say "attracting" rather than "asymptotically stable"
%			for parameters (e.g. ee) but not coefficients (e.g. alpha-beta) square roots should be written as fractional powers
%			for parameters, negative powers should be written as such (not as one over a positive power)
%  		sometimes codimension-three [four] points are called scenarios, phenomena & cases! (I think this is ok)
%			say "towards" rather than "toward"
%			fractional powers in brackets written with "\big(" (because \left( is too big!), e.g.
%					$\cO \big( \Delta^{\frac{1}{2}} \big)$
%					$\cO \big( \ee^{\frac{1}{2}} \big)$

% keywords:
% MSC codes:

\begin{abstract}
A wide variety of intricate dynamics may be created at border-collision bifurcations of piecewise-smooth maps,
where a fixed point collides with a surface at which the map is nonsmooth.
For the border-collision normal form in two dimensions, a codimension-three scenario was described in previous work
at which the map has a saddle-type periodic solution and an infinite sequence of stable
periodic solutions that limit to a homoclinic orbit of the saddle-type solution.
This paper introduces an alternate scenario of the same map
at which there is an infinite sequence of stable periodic solutions
due to the presence of a repeated unit eigenvalue in the linearization of some iterate of the map.
It is shown that this scenario is codimension-four
and that the sequence of periodic solutions is unbounded,
aligning with eigenvectors corresponding to the unit eigenvalue.

Arbitrarily many attracting periodic solutions coexist near either scenario.
It is shown that if $K$ denotes the number of attracting periodic solutions,
and $\ee$ denotes the distance in parameter space from one of the two scenarios,
then in the codimension-three case $\ee$ scales with $\lambda^{-K}$,
where $\lambda > 1$ denotes the unstable stability multiplier associated with the saddle-type periodic solution,
and in the codimension-four case $\ee$ scales with $K^{-2}$.
Since $K^{-2}$ decays significantly slower than $\lambda^{-K}$,
large numbers of attracting periodic solutions coexist
in open regions of parameter space extending
substantially further from the codimension-four scenarios than the codimension-three scenarios.
\end{abstract}

%=====================================================================
\section{Introduction}
\label{sec:intro}
\setcounter{equation}{0}

The coexistence of attractors is a critical feature of many nonlinear dynamical systems \cite{Fe08}.
For such systems the long-term dynamics can be altered by changing the initial conditions.
When boundaries of basins of attraction are highly intertwined,
long-term dynamics can be extremely sensitive to initial conditions
and physical experiments may be inherently unpredictable.
This form of unpredictability in a deterministic system
is different to that caused by the presence of a single chaotic attractor
for which the exact dynamics cannot be accurately forecast on a long time-scale,
but statistical properties of the long-term dynamics can be determined.
Multistable systems often exhibit relatively novel dynamics in the presence of noise.
A solution remains near an attractor until noise drives the solution elsewhere.
Consequently, with noise of an appropriately intermediate strength,
solutions may experience periods of relatively steady behavior
separated by rapid transitions between neighborhoods of the attractors.
%For a review of multistability the reader is referred to \cite{Fe08}.

Bistability is a familiar phenomenon in dynamical systems.
The coexistence of two attractors is common
near subcritical Hopf bifurcations and tipping points of climate models \cite{ThSi11}.
The coexistence of a large number of attractors is more exotic, yet important in many areas of applied science.
The observation that stable beating solutions and a variety stable bursting solutions can coexist in a model of a single neuron
has been used to argue that the neuron is able to exhibit sophisticated information processing \cite{CaBa93}. %also \cite{FoLo96}
Multistability has been described in circulation models of oceans
for which different attractors correspond to different stable convection patterns \cite{Ra95b}.
%The presence of feedback in chemical reactions may induce various complex dynamics.
In \cite{WaSu03} it was shown that the addition of a buffer step to a model of an autocatalator
creates infinitely many coexisting attractors.
Extreme multistability has also been described in various prototypical models
such as the Duffing oscillator \cite{ArBa85},
a single kicked rotor \cite{FeGr96},
and the H\'{e}non map \cite{GoBa02}.
%see also \cite{AnFe04,FeGr98,Pi01,KnWe89}

In \cite{Ne74} it was shown that smooth maps exhibit infinitely many attractors
on a dense set of parameter values (known as a Newhouse region)
near where the map has a homoclinic tangency.
Typically the related bifurcation structure is extremely
complex involving nested bifurcation sequences \cite{Ro83,GoBa02}. %also \cite{Go00}
Area-preserving maps may exhibit infinitely many elliptic periodic orbits \cite{GoSh00}.
If a small amount of dissipation is added,
the periodic orbits become attracting but finite in number.
In \cite{FeGr97} it was found that the number of coexisting attractors appears to
be inversely proportional to the magnitude of the dissipation.
For networks of weakly coupled oscillators for which clusters tend to synchronize,
many different clusters are often possible corresponding to the coexistence of many stable solutions \cite{Ka90b}.
%Various methods exist for the control of multistable systems \cite{GoPi08}.

This paper investigates large numbers of coexisting attracting periodic solutions in the two-dimensional border-collision normal form,
\begin{equation}
\left[ \begin{array}{c} x_{i+1} \\ y_{i+1} \end{array} \right] = %f(x_i,y_i) =
\left\{ \begin{array}{lc}
\left[ \begin{array}{cc} \tau_{\sL} & 1 \\ -\delta_{\sL} & 0 \end{array} \right]
\left[ \begin{array}{c} x_i \\ y_i \end{array} \right] +
\left[ \begin{array}{c} 1 \\ 0 \end{array} \right] \mu \;, & x_i \le 0 \\
\left[ \begin{array}{cc} \tau_{\sR} & 1 \\ -\delta_{\sR} & 0 \end{array} \right]
\left[ \begin{array}{c} x_i \\ y_i \end{array} \right] +
\left[ \begin{array}{c} 1 \\ 0 \end{array} \right] \mu \;, & x_i \ge 0
\end{array} \right. \;.
\label{eq:f}
\end{equation}
The map (\ref{eq:f}) is piecewise-linear and continuous and describes dynamics 
local to border-collision bifurcations which arise in piecewise-smooth models of diverse physical systems \cite{DiBu08,ZhMo03,PuSu06}.
The map (\ref{eq:f}) has been the topic of many investigations \cite{NuYo92,BaGr99,ZhMo06b,SuGa08,Si10},
and multistability is emphasized in \cite{KaMa98,DuNu99,ZhMo08d}.

Border-collision for (\ref{eq:f}) occurs at $\mu = 0$.
The dynamics for $\mu \ne 0$ is independent to the magnitude of $\mu$, up to a spatial scaling,
and represents dynamics created by the border-collision bifurcation at $\mu = 0$.
For this reason, throughout this paper $\mu$ is treated as fixed at a nonzero value,
with which the parameter space of (\ref{eq:f}) is $\mathbb{R}^4$,
because $\tau_{\sL}$, $\delta_{\sL}$, $\tau_{\sR}$ and $\delta_{\sR}$ are each permitted to take any value in $\mathbb{R}$.

%The dynamics for $\mu \ne 0$ is determined by the values of the remaining parameters,
%$\tau_{\sL}, \delta_{\sL}, \tau_{\sR}, \delta_{\sR} \in \mathbb{R}$,
%and has been the topic of many investigations \cite{NuYo92,BaGr99,ZhMo06b,SuGa08,Si10,SiMe12}. % refer to ... and references within.
%Multistability for (\ref{eq:f}) is emphasized in \cite{KaMa98,DuNu99,ZhMo08d}.
%The four additional parameters $\tau_{\sL}$, $\delta_{\sL}$, $\tau_{\sR}$ and $\delta_{\sR}$ may take any value in $\mathbb{R}$.

%Dynamical systems with an abrupt feature,
%such as a switch for a circuit system or a control system or an impact for a mechanical system,
%are often well-modelled by piecewise-smooth equations 
%Codimension-one subsets of phase space on which the equations are nonsmooth are termed switching manifolds.
%The collision of an invariant set with a switching manifold under
%parameter change is a discontinuity-induced bifurcation at which new invariant sets may be created.
%For a piecewise-smooth map, the collision of a fixed point with a switching manifold
%at which the map is continuous, non-differentiable,
%but with well-defined one-sided derivatives at the point of collision,
%is termed a border-collision bifurcation.
%As with classical local bifurcations %(such as saddle-node and Hopf bifurcations),
%dynamics in a neighborhood of a border-collision bifurcation
%is completely determined by a truncated series expansion of the map about the bifurcation, except in degenerate cases.
%Dynamics near the bifurcation are well-approximated by a piecewise-linear, continuous map.

To study (\ref{eq:f}) it is convenient to associate orbits %$\left\{ (x_i,y_i) \right\}$
with symbol sequences on an alphabet $\{ \sL,\sR \}$.
A sequence is defined by setting the $i^{\rm th}$ symbol of the sequence to $\sL$ if $x_i < 0$,
and setting the $i^{\rm th}$ symbol to $\sR$ if $x_i > 0$.
(Either symbol may be chosen if $x_i = 0$.)
Any periodic solution to (\ref{eq:f}) has an associated symbol sequence, $\cS$, that is periodic,
and is referred to as an $\cS$-cycle.

In previous work \cite{Si14} it was shown that there exist codimension-three points of parameter space
at which (\ref{eq:f}) has a saddle-type periodic solution, say an $\cX$-cycle, with a coincident homoclinic connection,
and infinitely many stable $\cX^k \cY$-cycles, for some $\cY$.
As $k \to \infty$, the $\cX^k \cY$-cycles limit to a homoclinic orbit of the $\cX$-cycle.
The stability multipliers of the $\cX$-cycle, $\lambda_1$ and $\lambda_2$, must satisfy $\lambda_1 \lambda_2 = 1$.
Three particular examples were given for which the $\cX^k \cY$-cycles are not only stable, but attracting.
Since this phenomena is codimension-three and parameter space is four-dimensional,
there exist curves along which this phenomena occurs for a given combination of $\cX$ and $\cY$.

A summary of the content and organization of this paper is as follows.
Further conventions for (\ref{eq:f}) are given in \S\ref{sec:backg},
and the codimension-three scenario summarized above is further outlined in \S\ref{sub:codim3}.
Next it is assumed that $\tau_{\sL}$, $\delta_{\sL}$, $\tau_{\sR}$ and $\delta_{\sR}$
vary smoothly with a single parameter $\ee$, where $\ee = 0$ corresponds to a codimension-three point,
and the effect of increasing $\ee$ from zero is studied.
If the $\cX^k \cY$-cycles are attracting when $\ee = 0$,
then regardless of the direction in parameter space that we head from $\ee = 0$,
each $\cX^k \cY$-cycle is admissible and attracting over an interval of $\ee$-values.
In \S\ref{sub:asymptotic3}, upper and lower bounds are obtained for the upper end-point of this interval.
Both bounds are proportional to $\lambda_2(0)^k$,
where $\lambda_2(\ee)$ denotes the unstable stability multiplier of the $\cX$-cycle.
(The bifurcation values of single-round periodic solutions near homoclinic tangencies of smooth maps
satisfy the same limiting behavior \cite{GaSi72,GaSi73,GaWa87}.)
It follows that if $\ee_K$ denotes the supremum value of $\ee$ for which
the number of coexisting attracting $\cX^k \cY$-cycles is equal to $K$, then $\ee_K \sim \varphi(K) \lambda_2(0)^K$,
where $\varphi(K)$ is bounded between positive constants.
In \S\ref{sub:scaling3} this scaling law is illustrated for three examples.

In \S\ref{sub:codim4} novel codimension-four points are introduced
at which (\ref{eq:f}) has infinitely many stable $\cX^k \cY$-cycles.
At these points $M_{\cX}$ has a repeated unit eigenvalue,
where $M_{\cX}$ is in general a matrix whose eigenvalues are the stability multipliers of the $\cX$-cycle,
although at the codimension-four points the $\cX$-cycle does not exist.
It is tempting to infer that the codimension-four points are end-points of
curves of codimension-three points at which $\lambda_1 = \lambda_2 = 1$,
however this is not the case.
It is found that the codimension-four points must be distant from a codimension-three
scenario involving the same $\cX$ and $\cY$.
%The main result concerning the nature of the codimension-four points is Theorem \ref{th:codim4}.
In contrast to the codimension-three points, the $\cX^k \cY$-cycles cannot be attracting,
and as $k \to \infty$ the $\cX^k \cY$-cycles grow in size without bound.
The structure of the $\cX^k \cY$-cycles is discussed in \S\ref{sub:structure4},
and in \S\ref{sub:examples4} examples are given for three different choices of $\cX$ and $\cY$.
These are derived by performing calculations based on the requirement that $M_{\cX}$ has a repeated unit eigenvalue.
The validity of these examples is formally verified by explicitly computing
each point of the $\cX^k \cY$-cycles for an arbitrary value of $k$.

Section \ref{sec:perturb4} concerns perturbations from the codimension-four points.
If the direction in parameter space that we head from $\ee = 0$ is chosen appropriately,
then for small $\ee > 0$ there are a large number of attracting $\cX^k \cY$-cycles.
In \S\ref{sub:stability4}, Theorem \ref{th:alphaBeta} tells us exactly which directions are appropriate
and is proved by looking at the stability of the $\cX^k \cY$-cycles.
Admissibility of $\cX^k \cY$-cycles is studied in \S\ref{sub:admissibility4}
from which it is found that we have the alternate scaling law, $\ee_K \sim \varphi(K) K^{-2}$.
In \S\ref{sub:scaling4} this scaling law is illustrated for the three examples of \S\ref{sub:examples4}.

Finally \S\ref{sec:conc} presents a summary and discussion.
Appendices \ref{app:eigenvector}-\ref{app:mDependent} contain detailed aspects of the proofs of the results.

%includes a comparison of the results to Newhouse regions of smooth maps.

%=====================================================================
\section{Symbolic dynamics and periodic solutions}
\label{sec:backg}
\setcounter{equation}{0}

We denote the left and right half-maps of (\ref{eq:f}) by
\begin{equation}
f^{\sL}(x,y) = A_{\sL}
\left[ \begin{array}{c} x \\ y \end{array} \right] +
\left[ \begin{array}{c} 1 \\ 0 \end{array} \right] \mu \;, \qquad
f^{\sR}(x,y) = A_{\sR}
\left[ \begin{array}{c} x \\ y \end{array} \right] +
\left[ \begin{array}{c} 1 \\ 0 \end{array} \right] \mu \;,
\label{eq:fLfR}
\end{equation}
where
\begin{equation}
A_{\sL} = \left[ \begin{array}{cc} \tau_{\sL} & 1 \\ -\delta_{\sL} & 0 \end{array} \right] \;, \qquad
A_{\sR} = \left[ \begin{array}{cc} \tau_{\sR} & 1 \\ -\delta_{\sR} & 0 \end{array} \right] \;.
\label{eq:ALAR}
\end{equation}
For any symbol sequence $\cS : \mathbb{Z} \to \{ \sL,\sR \}$ and initial point $(x_0,y_0)$, the orbit defined by
\begin{equation}
\left( x_{i+1},y_{i+1} \right) = f^{\cS_i} \left( x_i,y_i \right) \;,
\end{equation}
for $i \ge 0$, constitutes a forward orbit that ``follows $\cS$''.
If this orbit has the property that $x_i \le 0$ whenever %for all values of $i$ for which
$\cS_i = \sL$, and $x_i \ge 0$ whenever $\cS_i = \sR$,
then it is also an orbit of (\ref{eq:f}) and we say it is {\em admissible}.

If $\cS$ is periodic, then it is specified by $n$ consecutive symbols, say $\cS_0 \cdots \cS_{n-1}$,
where $n$ is the minimal period of $\cS$.
To avoid later confusion, here let us be somewhat punctilious
and note that the list $\cS_0 \cdots \cS_{n-1}$ is finite and is therefore a {\em word}.
Moreover, $\cS_0 \cdots \cS_{n-1}$ is a {\em primitive} word
(that is, cannot be written as a power) because $n$ is the minimal period.
Conversely, given a primitive word $\cS_0 \cdots \cS_{n-1}$,
the infinite repetition of this word generates a periodic symbol sequence with minimal period $n$.
Consequently, periodic symbol sequences of minimal period $n$ are isomorphic\removableFootnote{
Clearly this association defines a bijection.
Isomorphisms preserve structure, so in order to refer to it as an isomorphism
we define a function on the collection of all periodic symbol sequence of minimal period $n$,
that has effectively the same definition for primitive word of length $n$.
We let
\begin{equation}
F(\cS) = \sum_{m=0}^{m=n-1} 2^m \chi(\cS_m) \;,
\end{equation}
where $\chi(\sL) = 0$ and $\chi(\sR) = 1$.
For example, $F(\sR \sL \sR) = 1 \times 1 + 2 \times 0 + 4 \times 1 = 5$.
Then each sequence or word has a different value of $F$,
and the value of $F$ is preserved under the bijection. 
}
to primitive words of length $n$.
For this reason we may use periodic symbol sequences and primitive words interchangeably,
which is particularly convenient in regards to the form $\cS[k] = \cX^k \cY$.

Following \cite{SiMe09,SiMe10,Si10}, for any periodic symbol sequence $\cS$ of minimal period $n$, we let
\begin{equation}
f^{\cS} = f^{\cS_{n-1}} \circ \cdots \circ f^{\cS_0} \;,
\label{eq:fS}
\end{equation}
denote the $n^{\rm th}$ iterate of (\ref{eq:f}) following $\cS$.
The map $f^{\cS}$ is affine, and its matrix part is
\begin{equation}
M_{\cS} = A_{\cS_{n-1}} \cdots A_{\cS_0} \;.
\label{eq:MS}
\end{equation}
Throughout this paper we use the notation $\left( x^{\cS}_i,y^{\cS}_i \right)$, for $i = 0,\ldots,n-1$,
to denote the points of an $\cS$-cycle.
The point $\left( x^{\cS}_0,y^{\cS}_0 \right)$ is a fixed point of $f^{\cS}$.
Consequently, the $\cS$-cycle is unique if and only if $I - M_{\cS}$ is non-singular.
Equivalently, the $\cS$-cycle is unique if and only if $M_{\cS}$ does not have a unit eigenvalue.

If the $\cS$-cycle is admissible (that is, $x^{\cS}_i \le 0$ whenever $\cS_i = \sL$, and $x^{\cS}_i \ge 0$ whenever $\cS_i = \sR$)
and has no points on the switching manifold (that is, $x^{\cS}_i \ne 0$, for each $i$),
then the image of a small neighborhood of $\left( x^{\cS}_0,y^{\cS}_0 \right)$ under $n$ iterations of (\ref{eq:f})
is given by $f^{\cS}$.
In this case the stability of the $\cS$-cycle is determined by the eigenvalues of $M_{\cS}$.
The $\cS$-cycle is stable if both eigenvalues of $M_{\cS}$ have modulus less than or equal to $1$,
and attracting if both eigenvalues have modulus less than $1$.
It follows that the $\cS$-cycle is stable if and only if
\begin{align}
\det(M_{\cS}) - {\rm trace}(M_{\cS}) + 1 &\ge 0 \;, \label{eq:stabConditionSN} \\
\det(M_{\cS}) + {\rm trace}(M_{\cS}) + 1 &\ge 0 \;, \label{eq:stabConditionPD} \\
\det(M_{\cS}) - 1 &\le 0 \;, \label{eq:stabConditionNS}
\end{align}
and is attracting if and only if the inequalities are satisfied strictly.
Note that if the $\cS$-cycle is unique,
equality is not possible in (\ref{eq:stabConditionSN}) because this corresponds to a unit eigenvalue.

%=====================================================================
\section{Perturbations from codimension-three points}
\label{sec:codimperturb3}
\setcounter{equation}{0}

In this section we study perturbations
from the codimension-three points of (\ref{eq:f}) that were introduced in \cite{Si14}.
We fix $\mu \ne 0$ and suppose that the remaining parameters of (\ref{eq:f})
vary smoothly with a real-valued parameter $\ee$,
where $\ee = 0$ corresponds to a codimension-three point for some $\cX$ and $\cY$.
In \S\ref{sub:codim3} we review the codimension-three points.
In \S\ref{sub:asymptotic3} we determine upper and lower bounds on the supremum value
of $\ee$ for which $\cX^k \cY$-cycles are admissible and attracting.
Lastly in \S\ref{sub:scaling3} we transform these bounds
into a scaling law that we illustrate for three examples.

%----------------------------------------------------------------------
\subsection{Multistability due to a coincident homoclinic connection.}
\label{sub:codim3}

We suppose that $M_{\cX}(\ee)$ has eigenvalues $\lambda_1(\ee)$ and $\lambda_2(\ee)$,
with $\lambda_1(0) \ne \lambda_2(0)$ and $\lambda_1(0),\lambda_2(0) \ne 1$.
The latter assumption on the eigenvalues implies that (\ref{eq:f}) has a unique $\cX$-cycle for small values of $\ee$.
We let $\zeta_1(\ee)$ and $\zeta_2(\ee)$ denote corresponding eigenvectors that vary smoothly with $\ee$,
and let $Q(\ee) = \left[ \zeta_1(\ee) ,\; \zeta_2(\ee) \right]$.
We then consider the change of coordinates
\begin{equation}
\left[ \begin{array}{c} u \\ v \end{array} \right] =
Q^{-1}(\ee) \left( \left[ \begin{array}{c} x \\ y \end{array} \right] -
\left[ \begin{array}{c} x^{\cX}_0(\ee) \\ y^{\cX}_0(\ee) \end{array} \right] \right) \;,
\label{eq:uvZ}
\end{equation}
and, for any $\cS$, let $g^{\cS}$ denote $f^{\cS}$ in $(u,v)$-coordinates.
The purpose of this coordinate change is so that $g^{\cX}$ is given simply by
\begin{equation}
g^{\cX}(w) = \left[ \begin{array}{cc} \lambda_1(\ee) & 0 \\ 0 & \lambda_2(\ee) \end{array} \right] w \;,
\label{eq:gXZ}
\end{equation}
where $w = (u,v)$.
At this stage have no knowledge of the map $g^{\cY}$, other than that it is affine, so we write it as
\begin{equation}
g^{\cY}(w) = \left[ \begin{array}{cc}
\gamma_{11}(\ee) & \gamma_{12}(\ee) \\
\gamma_{21}(\ee) & \gamma_{22}(\ee)
\end{array} \right] w +
\left[ \begin{array}{c} \sigma_1(\ee) \\ \sigma_2(\ee) \end{array} \right] \;,
\label{eq:gYZ}
\end{equation}
for some $\gamma_{ij}$, $\sigma_1$ and $\sigma_2$.
The following theorem is taken from \cite{Si14}.

%.....................................................................
%\begin{theorem}[\cite{Si14}]
%\begin{theorem}[Simpson, 2013 \cite{Si14}]		% <--- need to change year eventually; different for IJBC & arXiv due to different bib styles?
\begin{theorem}
Let $\cS[k] = \cX^k \cY$, where $\cX$ is primitive and $\cX_0 \ne \cY_0$.
Let $\tau_{\sL}, \delta_{\sL}, \tau_{\sR}, \delta_{\sR} \in \mathbb{R}$ and $\mu \ne 0$.
Suppose there exist infinitely many values of $k \ge 1$ for which (\ref{eq:f})
exhibits a unique, admissible, stable $\cS[k]$-cycle that has no points on the switching manifold.
Suppose that the eigenvalues of $M_{\cX}(0)$ are $0 \le \lambda_1(0) < 1 < \lambda_2(0)$,
and that $\zeta_2(0)$ (the eigenvector corresponding to $\lambda_2(0)$)
is not a scalar multiple of $[0,1]^{\sf T}$.
Then
\begin{enumerate}[label=\roman{*}),ref=\roman{*}]
\item
\label{it:lambda12Z}
$\lambda_1(0) \ne 0$ and $\lambda_2(0) = \frac{1}{\lambda_1(0)}$;
\item
\label{it:gamma22sigma2}
$\gamma_{22}(0) = 0$ and
$\sigma_2(0) = 0$;
\item
\label{it:gamma21sigma1}
$\gamma_{21}(0) \ne 0$ and
$\sigma_1(0) \ne 0$.
\end{enumerate}
\label{th:codim3}
\end{theorem}

Theorem \ref{th:codim3} is slightly weaker than the analogous result stated in \cite{Si14}
in that part (\ref{it:gamma21sigma1}) of the above theorem %Theorem \ref{th:codim3}
is proved as a step towards demonstrating that $\cS[k]$-cycles
limit to an orbit that is homoclinic to the $\cX$-cycle as $k \to \infty$.
The conditions $\gamma_{21}(0) \ne 0$ and $\sigma_1(0) \ne 0$ are crucial to the analysis below,
whereas the existence of a homoclinic connection is not needed
and indeed the existence and nature of homoclinic orbits for $\ee > 0$ is beyond the scope of this paper.

%----------------------------------------------------------------------
\subsection{Stability and admissibility of periodic solutions for perturbed parameter values}
\label{sub:asymptotic3}

To investigate $\cS[k]$-cycles, we compose $k$ instances of $g^{\cX}$
(\ref{eq:gXZ}) with $g^{\cY}$ (\ref{eq:gYZ}) to obtain
\begin{equation}
g^{\cS[k]}(w) = \left[ \begin{array}{cc}
\gamma_{11}(\ee) \lambda_1^k(\ee) & \gamma_{12}(\ee) \lambda_2^k(\ee) \\
\gamma_{21}(\ee) \lambda_1^k(\ee) & \gamma_{22}(\ee) \lambda_2^k(\ee)
\end{array} \right] w +
\left[ \begin{array}{c}
\sigma_1(\ee) \lambda_1^k(\ee) \\
\sigma_2(\ee) \lambda_2^k(\ee)
\end{array} \right] \;.
\label{eq:gSkZ}
\end{equation}
The matrix part of (\ref{eq:gSkZ}) has the same spectrum as $M_{\cS[k]}$
(the matrix part of $f^{\cS[k]}$), therefore
\begin{align}
\det \left( M_{\cS[k]}(\ee) \right) &=
\left( \gamma_{11}(\ee) \gamma_{22}(\ee) - \gamma_{12}(\ee) \gamma_{21}(\ee) \right)
\lambda_1^k(\ee) \lambda_2^k(\ee) \;, \label{eq:detMSkZ} \\
{\rm trace} \left( M_{\cS[k]}(\ee) \right) &=
\gamma_{11}(\ee) \lambda_1^k(\ee) + \gamma_{22}(\ee) \lambda_2^k(\ee) \;. \label{eq:traceMSkZ}
\end{align}
Given any suitably large value of $k$,
we first establish an upper bound on the largest value of $\ee$ 
for which the eigenvalues of $M_{\cS[k]}(\ee)$ have modulus less than or equal to $1$.
For values of $\ee$ greater than this bound,
if the $\cS[k]$-cycle is admissible with no points on the switching manifold then it cannot be stable.
%Take care to note that throughout this paper $\ee$ is assumed to be small.

%.....................................................................
\begin{lemma}
Suppose $\mu \ne 0$ and that the remaining parameters of (\ref{eq:f}) vary smoothly with $\ee$.
Suppose that when $\ee = 0$ (\ref{eq:f}) satisfies the assumptions of Theorem \ref{th:codim3}.
Suppose $\gamma_{22}'(0) \ne 0$.
Then there exists $\ee^* > 0$ and $k_{\rm min} \in \mathbb{Z}$,
such that for all $k \ge k_{\rm min}$
and $\frac{3}{\left| \gamma_{22}'(0) \right|} \lambda_2^{-k}(0) \le \ee \le \ee^*$,
if the $\cS[k]$-cycle is admissible with no points on the switching manifold,
then it is unstable.
\label{le:upperBoundZ}
\end{lemma}

We expect that the point in parameter space corresponding to $\ee = 0$
lies on a curve of codimension-three points associated with Theorem \ref{th:codim3}.
The condition $\gamma_{22}'(0) \ne 0$ ensures that as we increase the value of $\ee$
from zero, we move away from this curve in a transverse direction.

\begin{proof}
We have $\lambda_1(0) < 1$, $\lambda_2(0) > 1$ and $\gamma_{22}(0) = 0$,
thus by (\ref{eq:traceMSkZ}),
\begin{equation}
{\rm trace} \left( M_{\cS[k]}(\ee) \right) \sim \gamma_{22}'(0) \lambda_2^k(0) \ee \;.
\end{equation}
Also $\lambda_2^k(0) \ge \frac{3}{\left| \gamma_{22}'(0) \right| \ee}$,
hence for sufficiently small $\ee > 0$,
$\left| {\rm trace} \left( M_{\cS[k]}(\ee) \right) \right|$ is approximately greater than $3$.
In particular $\left| {\rm trace} \left( M_{\cS[k]}(\ee) \right) \right| > 2$,
thus the stability conditions
(\ref{eq:stabConditionSN})-(\ref{eq:stabConditionNS}) cannot all be satisfied.
Therefore if the $\cS[k]$-cycle is admissible with no points on the switching manifold,
it cannot be stable.
\end{proof}

The next lemma gives an analogous lower bound on the largest value of $\ee$
for which the $\cS[k]$-cycle is admissible and attracting.
This result is considerably more difficult to obtain
because it is necessary to demonstrate that every point of an $\cS[k]$-cycle
lies on the correct side of the switching manifold.
%For this reason it is beyond the scope of this paper 

If the $\cS[k]$-cycles are attracting when $\ee = 0$,
then there is a large number of admissible, attracting $\cS[k]$-cycles
for a perturbation in any direction from $\ee = 0$.
Alternatively if $\cS[k]$-cycles are stable but not attracting when $\ee = 0$, only certain directions will work.
In this case we have equality in one of (\ref{eq:stabConditionSN})-(\ref{eq:stabConditionNS})
for the $\cS[k]$-cycles when $\ee = 0$.
Equality in (\ref{eq:stabConditionSN}) is not possible
in view of the uniqueness of $\cS[k]$-cycles.
Equality in (\ref{eq:stabConditionPD}) is also not possible because, by (\ref{eq:traceMSkZ}) and
part (\ref{it:gamma22sigma2}) of Theorem \ref{th:codim3},
${\rm trace} \left( M_{\cS[k]}(0) \right) = \gamma_{11}(0) \lambda_1^k(0) \to 0$ as $k \to \infty$.
However, equality in (\ref{eq:stabConditionNS}) is possible, and occurs if $\gamma_{12}(0) \gamma_{21}(0) = -1$.
In this case $\cS[k]$-cycles are attracting for $\ee > 0$
only if $\det \left( M_{\cX}(\ee) \right) < 1$.
The inequality $\lambda_1'(0) \lambda_2(0) + \lambda_1(0) \lambda_2'(0) < 0$
is equivalent to $\frac{d}{d \ee} \det \left( M_{\cX}(\ee) \right) \big|_{\ee = 0} < 0$,
and therefore $\cS[k]$-cycles are attracting for small $\ee > 0$ if this inequality holds.

%.....................................................................
\begin{lemma}
Suppose $\mu \ne 0$ and that the remaining parameters of (\ref{eq:f}) vary smoothly with $\ee$.
Suppose that when $\ee = 0$ (\ref{eq:f}) satisfies the assumptions of Theorem \ref{th:codim3} and
\begin{equation}
\min_i \left| x^{\cS[k]}_i(0) \right| \not\to 0 {\rm ~as~} k \to \infty \;.
\label{eq:distanceToSwManZ}
\end{equation}
If $\gamma_{12}(0) \gamma_{21}(0) = -1$,
suppose $\lambda_1'(0) \lambda_2(0) + \lambda_1(0) \lambda_2'(0) < 0$.
Then there exists $k_{\rm min} \in \mathbb{Z}$ and $\Delta > 0$,
such that for all $k \ge k_{\rm min}$ and
$0 < \ee \le \Delta \lambda_2^{-k}(0)$,
the $\cS[k]$-cycle is admissible and attracting.
\label{le:lowerBoundZ}
\end{lemma}

As $k \to \infty$, $\cS[k]$-cycles approach an orbit that is homoclinic to the $\cX$-cycle.
Equation (\ref{eq:distanceToSwManZ}) is equivalent to the assumption that this homoclinic orbit
has no points on the switching manifold.
If (\ref{eq:distanceToSwManZ}) does not hold then $\cS[k]$-cycles are closer to the switching manifold 
than they are in generic scenarios, and consequently it may be possible
for perturbations smaller than $\ee \propto \lambda_2^{-k}(0)$
to cause the $\cS[k]$-cycles to become virtual.

\begin{proof}
Given $k_{\rm min} \in \mathbb{Z}$ and $\Delta > 0$,
if $k \ge k_{\rm min}$ and $0 < \ee \le \Delta \lambda_2^{-k}(0)$,
then by (\ref{eq:traceMSkZ}) we can write
${\rm trace} \left( M_{\cS[k]}(\ee) \right) =
\cO \left( \lambda_1(0)^{k_{\rm min}} \right) + \cO(\Delta)$.
Also by (\ref{eq:detMSkZ}) we have,
\begin{equation}
\det \left( M_{\cS[k]}(\ee) \right) =
-\gamma_{12}(0) \gamma_{21}(0) \big\{ 1 +
\left( \lambda_1'(0) \lambda_2(0) + \lambda_1(0) \lambda_2'(0) \right) k \ee \big\}
+ \cO(\ee) + \cO \left( k^2 \ee^2 \right) \;.
\nonumber
\end{equation}
From these equations we can see that there exists $k_{\rm min} \in \mathbb{Z}$
and $\Delta > 0$ such that
the stability conditions (\ref{eq:stabConditionSN})-(\ref{eq:stabConditionNS})
hold strictly for all $k \ge k_{\rm min}$ and
$0 < \ee \le \Delta \lambda_2^{-k}(0)$.
In this case if the $\cS[k]$-cycle is admissible, it is also attracting.

In order to verify admissibility it suffices to write\removableFootnote{
Really we have 
$\det \left( I - M_{\cS[k]}(\ee) \right) = \det \left( I - M_{\cS[k]}(0) \right)
\cO(k\ee) + \cO \left( \lambda_1(0)^{k_{\rm min}} \right) + \cO(\Delta)$,
but the idea is that we may assume that $\ee^*$ is sufficiently small
and $k_{\rm min}$ is sufficiently large that
the $\cO(\Delta)$ error term dominates.
Also we must have $\Delta \gg \ee$, otherwise the interval of $k$-values
that we are considering is the empty set.
The purpose of this contraction is to simplify the error terms in the statements below.
}
$\det \left( I - M_{\cS[k]}(\ee) \right) = \det \left( I - M_{\cS[k]}(0) \right) + \cO(\Delta)$,
where $\det \left( I - M_{\cS[k]}(0) \right) \ne 0$.
The unique fixed point of $g^{\cS[k]}$ is $w^{\cS[k]}_0(\ee)$.
From (\ref{eq:gXZ}) and (\ref{eq:gSkZ}) we find that 
the $j^{\rm th}$ iterate of this point under $g^{\cX}$ is given by\removableFootnote{
By substituting $j = k$ and solving $x^{\cS[k]}_0(\ee) = 0$, we obtain
\begin{equation}
\ee \sim \frac{q_{12} \gamma_{21} \sigma_1 + \left( 1 - \gamma_{12} \gamma_{21} \right) x^{\cX}_0}
{q_{12} \sigma_2'} \lambda_2^{-k} \;,
\end{equation}
where all coefficients are evaluated at $\ee = 0$.
To prove admissibility until $\ee \sim \varphi \lambda_2^{-k}$,
we have to justify the above expression
by writing $w^{\cS[k]}_0(\ee)$ as a series expansion or using the implicit function theorem
(I'm not sure how to do this).
Repeat to find where $x^{\cS[k]}_i(\ee) = 0$, for $i = 0,\ldots,n_{\cX}-1$
and $i = (k-1) n_{\cX},\ldots,k n_{\cX} + n_{\cY}$
(and if in any case the denominator is zero then it may be discounted).
Take the minimum of the constants.
Argue admissibility up until this minimum for all $i$
(would have to generalize my current argument somehow).
}
\begin{equation}
w^{\cS[k]}_{j n_{\cX}}(\ee) =
\frac{1}{\det \left( I - M_{\cS[k]}(\ee) \right)}
\left[ \begin{array}{c}
\lambda_1^j(\ee) \left( 1 - \gamma_{22}(\ee) \lambda_2^k(\ee) \right) \sigma_1(\ee) +
\gamma_{12}(\ee) \lambda_1^j(\ee) \lambda_2^k(\ee) \sigma_2(\ee) \\
\gamma_{21}(\ee) \lambda_1^k(\ee) \lambda_2^j(\ee) \sigma_1(\ee) +
\lambda_2^j(\ee) \left( 1 - \gamma_{11}(\ee) \lambda_1^k(\ee) \right) \sigma_2(\ee)
\end{array} \right] \;,
\label{eq:wSkjZ}
\end{equation}
where $n_{\cX}$ denotes the length of the word $\cX$.
By using the bounds
\begin{equation}
\left( \gamma_{12}(\ee) \sigma_2(\ee) - \gamma_{22}(\ee) \sigma_1(\ee) \right)
\lambda_2^k(\ee) = \cO(\Delta) \;, \qquad
\sigma_2(\ee) \lambda_1^{-k}(\ee) = \cO(\Delta) \;,
\nonumber
\end{equation}
we can conclude from (\ref{eq:wSkjZ}) with $j=0$ that
$w^{\cS[k]}_0(\ee) = w^{\cS[k]}_0(0) + \cO(\Delta)$.
By appropriately iterating $w^{\cS[k]}_0(\ee)$ under $g^{\sL}$ and $g^{\sR}$ we obtain 
\begin{equation}
w^{\cS[k]}_i(\ee) = w^{\cS[k]}_i(0) + \cO(\Delta) \;,
\label{eq:wSkiZ}
\end{equation}
for $i = 0,1,\ldots$ up to any $k$-dependent value.
Importantly, this value must be independent of $k$, because $k$ may be arbitrarily large.
For the purposes of this proof we use (\ref{eq:wSkiZ}) for $i = 0,\ldots,n_{\cX}-1$.
Similarly from (\ref{eq:wSkjZ}) with $j=k-1$
and $n_{\cX} + n_{\cY}$ subsequent iterations of $g^{\sL}$ and $g^{\sR}$ following $\cX \cY$,
we may say that (\ref{eq:wSkiZ}) also holds for all
$i = (k-1) n_{\cX},\ldots, k n_{\cX} + n_{\cY} - 1$.
We can now choose $\Delta > 0$ such that
as $\ee$ ranges from $0$ to $\Delta \lambda_2^{-k}(0)$,
each $w^{\cS[k]}_i(\ee)$ remains on the same side of the switching manifold
for all $i = 0,\ldots,n_{\cX}-1$ and $i = (k-1) n_{\cX},\ldots, k n_{\cX} + n_{\cY} - 1$
(and $\Delta > 0$ is chosen sufficiently small
that the $\cS[k]$-cycles are attracting for large $k$).
Then these points are admissible and
it remains to verify the admissibility of $w^{\cS[k]}_{j n_{\cX} + m}(\ee)$
for all $j = 1,\ldots,k-2$ and $m = 0,\ldots,n_{\cX}-1$.

All points in the triangle with vertices 
$w^{\cX}_m(\ee)$, $w^{\cS[k]}_m(\ee)$ and $w^{\cS[k]}_{(k-1) n_{\cX} + m}(\ee)$
lie on the same side of the switching manifold
for all $0 < \ee \le \Delta \lambda_2^{-k}(0)$,
because, as we have just shown, for each $m$, these three points
are admissible for all values of $\ee$ in this range.
In the case $m = 0$, we note that each
$w^{\cS[k]}_{j n_{\cX}}(\ee)$ is the image of $w^{\cS[k]}_{(j-1) n_{\cX}}(\ee)$ under $g^{\cX}$.
By (\ref{eq:gXZ}), each $w^{\cS[k]}_{j n_{\cX}}(\ee)$
lies inside the given triangle and is therefore admissible\removableFootnote{
This doesn't require as much work as for Lemma \ref{le:lowerBound}
because here $g^{\cX}$ is given exactly.
}.
The same is true for each $m \ne 0$ as
can be seen by repeating this argument for coordinates centered at $w^{\cX}_m(\ee)$
with axes that coincide locally with the stable and unstable manifolds of this point.
\end{proof}

%----------------------------------------------------------------------
\subsection{A scaling law for the number of attracting periodic solutions}
\label{sub:scaling3}

For small $\ee > 0$, let $\kappa(\ee)$ denote the number of $\cS[k]$-cycles
that are admissible and attracting\removableFootnote{
I may need to mention or explain why that
here we are lazy and not worrying about $\cS[k]$-cycles with points on the switching manifold.
}.
For a suitably small value $\ee^* > 0$, and any positive integer $K$, let
\begin{equation}
\ee_K = \sup_{0 \le \ee \le \ee^*} \left[ \kappa(\ee) = K \right] \;,
\label{eq:eeK}
\end{equation}
denote the supremum value of $\ee$ for which the number of admissible, attracting $\cS[k]$-cycles is equal to $K$.
The following result is a simple consequence of Lemmas \ref{le:upperBoundZ} and \ref{le:lowerBoundZ}.

%.....................................................................
\begin{theorem}
Suppose $\mu \ne 0$ and that the remaining parameters of (\ref{eq:f}) vary smoothly with $\ee$.
Suppose that when $\ee = 0$ (\ref{eq:f}) satisfies the assumptions of Theorem \ref{th:codim3},
$\gamma_{22}'(0) \ne 0$ and (\ref{eq:distanceToSwManZ}) holds.
If $\gamma_{12}(0) \gamma_{21}(0) = -1$, suppose
$\lambda_1'(0) \lambda_2(0) + \lambda_1(0) \lambda_2'(0) < 0$.
Then as $K \to \infty$,
\begin{equation}
\ee_K \sim \varphi(K) \lambda_2^{-K}(0) \;,
\label{eq:scalingLawZ}
\end{equation}
for a function $\varphi(K)$ that is bounded between positive constants.
\label{th:scaling3}
\end{theorem}

\begin{proof}
Let $\hat{\varphi} = \frac{3}{\left| \gamma_{22}'(0) \right|}$, and assume $0 < \ee \le \ee^*$.
By Lemma \ref{le:upperBoundZ}, if
$\ee \ge \hat{\varphi} \lambda_2^{-k}(0)$
then the $\cS[k]$-cycle is virtual or unstable.
By Lemma \ref{le:lowerBoundZ}, there exists $k_{\rm min} \in \mathbb{Z}$ and $\Delta > 0$
such that for all $k \ge k_{\rm min}$,
if $\ee < \Delta \lambda_2^{-k}(0)$ then the $\cS[k]$-cycle is admissible and attracting.
Therefore, the supremum value of $\ee$ for which the $\cS[k]$-cycle is admissible and attracting
lies between $\Delta \lambda_2^{-k}(0)$ and $\hat{\varphi} \lambda_2^{-k}(0)$.
Hence $\ee_K$ is no greater than $\hat{\varphi} \lambda_2^{-K}(0)$,
and $\ee_K$ is no smaller than $\Delta \lambda_2^{-(K+k_{\rm min})}(0)$.
We therefore have (\ref{eq:scalingLawZ}), where
$\Delta \lambda_2^{-k_{\rm min}}(0) \le \varphi(K) \le \hat{\varphi}$.
\end{proof}

The values $\ee_K$ are bifurcations at which an $\cS[k]$-cycle loses either stability or admissibility.
Intuitively we expect that for large $K$ each $\ee_K$ to corresponds to the same type of bifurcation.
Indeed this is case for the three examples given below (and for the three examples of \S\ref{sub:scaling4}).
The bifurcations are border-collision bifurcations at which one point of an $\cS[k]$-cycle
collides with the switching manifold and admissibility is lost.
A search for parameter values where these particular bifurcations occur leads to 
$\ee_K \sim \varphi \lambda_2(0)^{-K}$, for some constant $\varphi$.
It remains to determine whether or not $\varphi(K)$ in (\ref{eq:scalingLawZ}) is constant in general,
and study the nature of the sequence of border-collision bifurcations.
This is discussed further in \S\ref{sec:conc}.

Here we illustrate (\ref{eq:scalingLawZ})
by numerically computing $\cS[k]$-cycles for parameter values near three points in parameter space
satisfying the assumptions of Theorem \ref{th:codim3} that were given in \cite{Si14}.
For each example we consider the arbitrary linear perturbation
\begin{equation}
\tau_{\sL}(\ee) = \tau_{\sL}(0) + a \ee \;, \qquad
\delta_{\sL}(\ee) = \delta_{\sL}(0) + b \ee \;, \qquad
\tau_{\sR}(\ee) = \tau_{\sR}(0) + c \ee \;, \qquad
\delta_{\sR}(\ee) = \delta_{\sR}(0) + d \ee \;,
\label{eq:paramEe}
\end{equation}
where $a,b,c,d \in \mathbb{R}$ are constants.
Furthermore, for each example $\cS[k]$-cycles are attracting for $\ee = 0$
and so there exist a large number of attracting $\cS[k]$-cycles
regardless of the choice of $a$, $b$, $c$ and $d$.
In each case there is a curve passing through the point at $\ee = 0$
along which there are infinitely many attracting $\cS[k]$-cycles.
For each example we use
\begin{equation}
a = 0 \;, \qquad
b = 1 \;, \qquad
c = 0 \;, \qquad
d = 0 \;,
\label{eq:abcdFIC}
\end{equation}
with which $\gamma_{22}'(0) \ne 0$ so that
as we increase $\ee$ from zero we move away from this curve transversely, as is the case for a generic perturbation\removableFootnote{
One can see this easily.
At codimension-three points $\det \left( M_{\cX} \right) = 1$.
With $b \ne 0$ and $d = 0$, the determinant of $M_{\cX}$
changes at a rate asymptotically linearly with respect to $\ee$.
}.
Different values of $a$, $b$, $c$ and $d$ give similar results.

%%%%%%%%%%%%%%%%%%%%%%%%%%%%%%%%%%%%%%%%%%%%%%%%%%%%%%%%%%%%%
\begin{figure}[b!]
\begin{center}
\setlength{\unitlength}{1cm}
\begin{picture}(14.9,11.3)
\put(.5,5.9){\includegraphics[height=5.4cm]{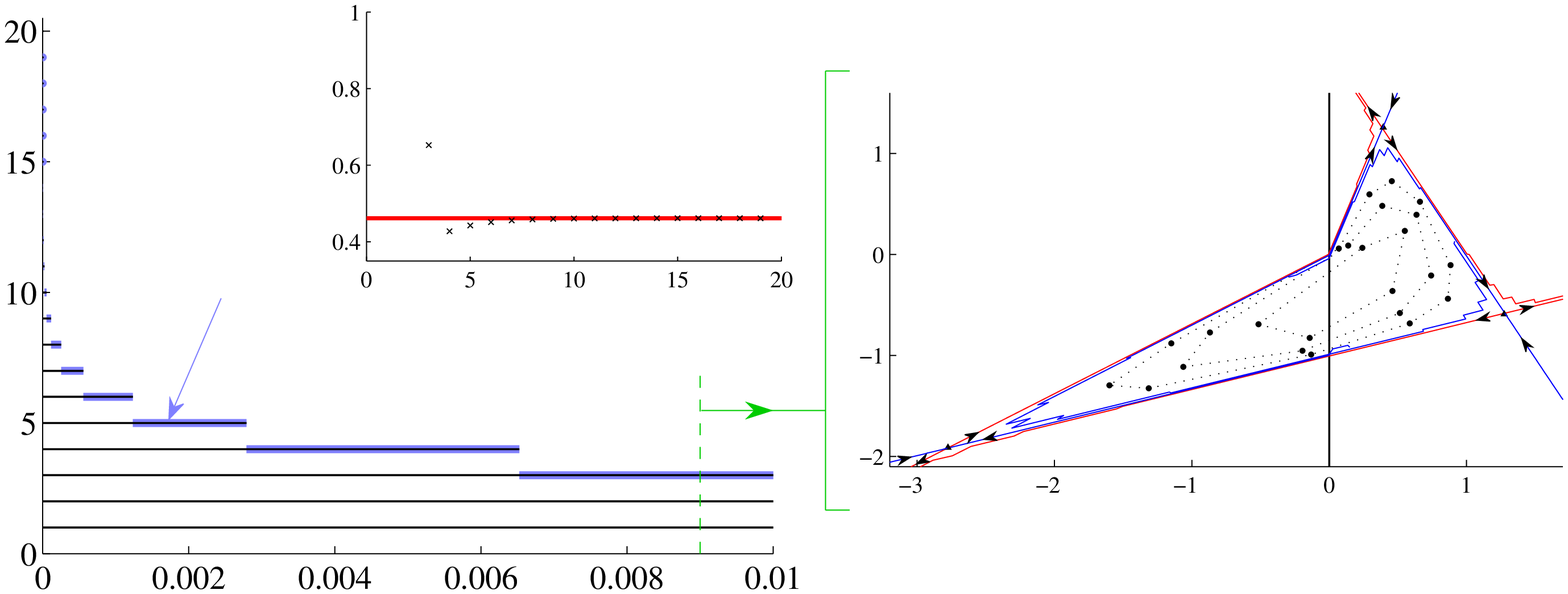}} % width = 13.9
\put(0,0){\includegraphics[height=5.4cm]{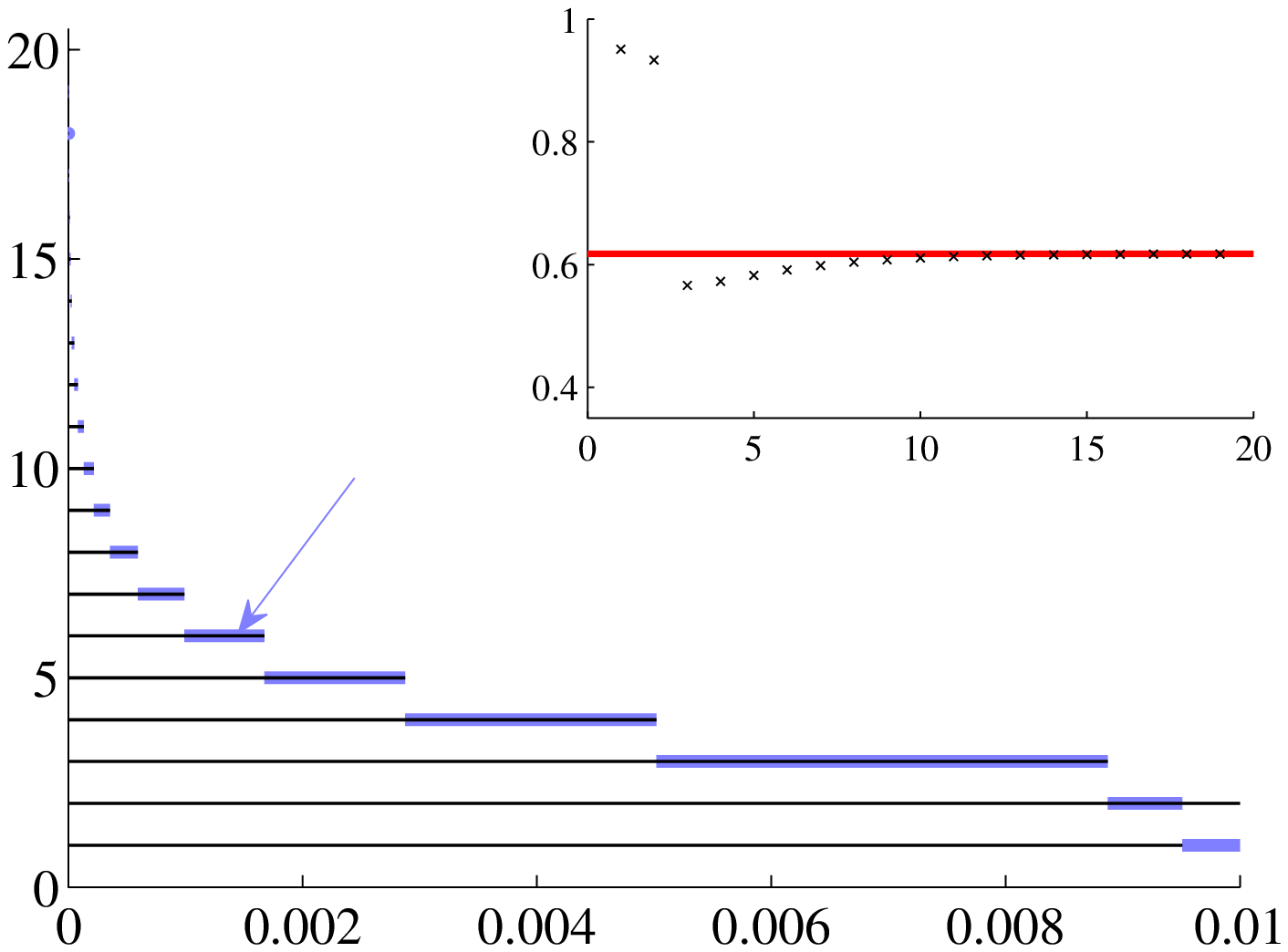}}	
\put(7.7,0){\includegraphics[height=5.4cm]{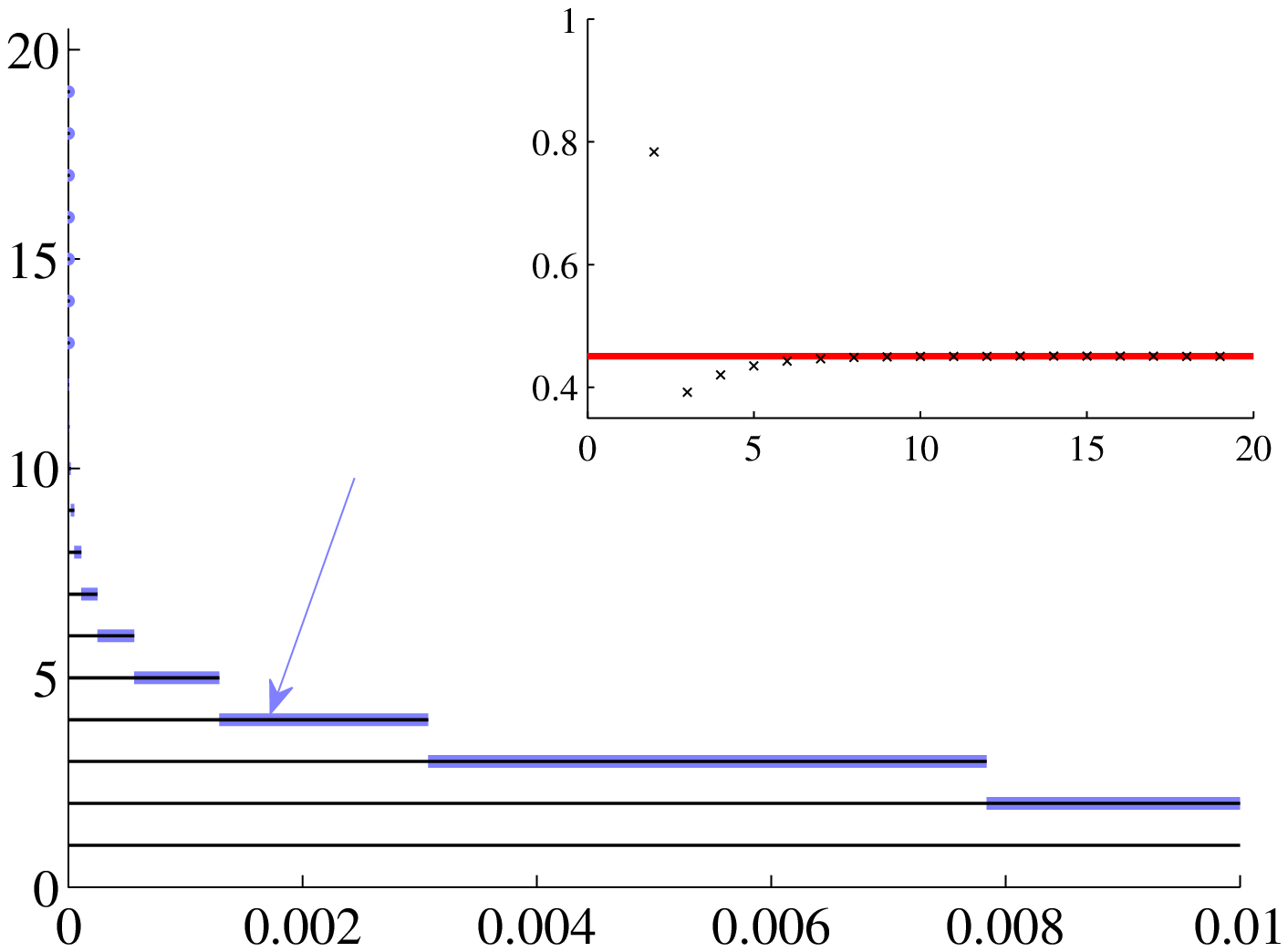}}
\put(11.47,6.8){\scriptsize $x$}
\put(8,8.7){\scriptsize $y$}
\put(10.9,10.6){\scriptsize $\ee = 0.009$}
\put(4.2,5.9){\small $\ee$}
\put(.5,9.2){\small $k$}
\put(3.7,0){\small $\ee$}
\put(0,3.3){\small $	k$}
\put(11.4,0){\small $\ee$}
\put(7.7,3.3){\small $k$}
\put(2.45,8.8){\footnotesize \color{blue} $\kappa(\ee)$}
\put(1.95,2.9){\footnotesize \color{blue} $\kappa(\ee)$}
\put(9.65,2.9){\footnotesize \color{blue} $\kappa(\ee)$}
\put(5.55,8.5){\scriptsize $K$}
\put(2.85,10.1){\scriptsize $\frac{\ee_{K+1}}{\ee_{K}}$}
\put(5.05,2.6){\scriptsize $K$}
\put(2.35,4.2){\scriptsize $\frac{\ee_{K+1}}{\ee_{K}}$}
\put(12.75,2.6){\scriptsize $K$}
\put(10.05,4.2){\scriptsize $\frac{\ee_{K+1}}{\ee_{K}}$}
\put(1.8,11.1){\large \sf \bfseries A}
\put(1.3,5.2){\large \sf \bfseries B}
\put(9,5.2){\large \sf \bfseries C}
\end{picture}
\caption{
Values of $\ee$ for which $\cS[k] = \cX^k \cY$-cycles are admissible, attracting periodic solutions of (\ref{eq:f}) with $\mu = 1$
and the remaining parameter values are of the form (\ref{eq:paramEe}).
Panels A, B and C respectively correspond to examples with
$\cX = \sR \sL \sR$, $\cX = \sR \sL^2 \sR$ and $\cX = \sR \sL \sR \sL \sR$,
%$n_{\cX} = 3$, $n_{\cX} = 4$ and $n_{\cX} = 5$ (where $n_{\cX}$ denotes the length of the word $\cX$)
as discussed in the text.
For each value of $k$, the interval of values of $\ee$
over which $\cS[k]$-cycles are admissible and attracting is indicated by a horizontal line segment.
The integer-valued function, $\kappa(\ee)$, denotes the number of $\cS[k] = \cX^k \cY$-cycles that are admissible and attracting.
The horizontal line in each inset indicates the corresponding value of $\frac{1}{\lambda_2(0)} = \lambda_1(0)$
(the stable stability multiplier of the $\cX$-cycle when $\ee = 0$).
Also included is a phase portrait corresponding to $\ee = 0.009$ in panel A.
Each point of the three $\cS[k]$-cycles that are admissible and attracting ($k=1,2,3$)
is connected by a dotted line segment to its third iterate under (\ref{eq:f}).
Parts of the stable and unstable manifolds of the $\cX$-cycle are also shown.
\label{fig:manyEeAllZ}
}
\end{center}
\end{figure}
%%%%%%%%%%%%%%%%%%%%%%%%%%%%%%%%%%%%%%%%%%%%%%%%%%%%%%%%%%%%%

We first consider
\begin{equation}
\tau_{\sL}(0) = -\frac{55}{117} \;, \qquad
\delta_{\sL}(0) = \frac{4}{9} \;, \qquad
\tau_{\sR}(0) = -\frac{5}{2} \;, \qquad
\delta_{\sR}(0) = \frac{3}{2} \;.
\label{eq:paramF}
\end{equation}
At these parameter values (\ref{eq:f}) with $\mu = 1$
has an $\cS[k]$-cycle for all $k \ge 1$, where
\begin{equation}
\cX = \sR \sL \sR \;, \qquad
\cY = \sL \sR \;.
\label{eq:XYF}
\end{equation}
In panel A of Fig.~\ref{fig:manyEeAllZ},
horizontal line segments indicate the range of values of $\ee$ for which $\cS[k]$-cycles
are admissible and attracting.
These line segments emanate from $\ee = 0$ for all $k \ge 1$ because when $\ee = 0$
the $\cS[k]$-cycles are admissible and stable for each of these values of $k$.
In general the number of coexisting attracting $\cS[k]$-cycles, $\kappa(\ee)$,
is equal to the number of line segments that intersect the given value of $\ee$.
Here $\kappa(\ee)$ simply
coincides with the largest value of $k$ for which the $\cS[k]$-cycle is admissible and stable.
%This is because for this particular example if the $\cS[K]$-cycle is admissible and stable,
%then for all $1 \le k < K$ the $\cS[k]$-cycle is also admissible and stable.

The values, $\ee_K$, are the right-hand end-points of the horizontal line segments.
These are border-collision bifurcations where
the $(3K)^{\rm th}$ point of the $\cS[K]$-cycle collides with the switching manifold.
If $\varphi(K)$ is a constant, the scaling law (\ref{eq:scalingLawZ}) predicts that
$\frac{\ee_{K+1}}{\ee_K} \to \frac{1}{\lambda_2(0)} = \lambda_1(0)$, as $K \to \infty$.
For this example $\frac{1}{\lambda_2(0)} = \frac{6}{13} \approx 0.4615$,
and, as shown in the inset of Fig.~\ref{fig:manyEeAllZ}-A,
this prediction is consistent with the data.

With $\ee = 0.009$, for example, we have $\kappa = 3$.
The corresponding three $\cS[k]$-cycles are shown to the right of panel A.
The $\cX$-cycle is also shown here.
The stable manifold of the $\cX$-cycle has a complicated structure
(for clarity only part of the manifold near the $\cX$-cycle is shown)
but does not appear to intersect the unstable manifold of the $\cX$-cycle at these parameter values.

The values
\begin{equation}
\tau_{\sL}(0) = 0.5 \;, \qquad
\delta_{\sL}(0) = \frac{1}{\delta_{\sR}} \;, \qquad
\tau_{\sR}(0) = -1.139755486 \;, \qquad
\delta_{\sR}(0) = 1.378851759 \;,
\label{eq:paramI}
\end{equation}
approximate (to ten significant figures)
a point in parameter space at which (\ref{eq:f}) with $\mu = 1$
has infinitely many attracting $\cS[k]$-cycles, where
\begin{equation}
\cX = \sR \sL^2 \sR \;, \qquad
\cY = \sL^2 \sR \;.
\label{eq:XYI}
\end{equation}
Here $\frac{1}{\lambda_2(0)} \approx 0.6175$ and,
as shown in Fig.~\ref{fig:manyEeAllZ}-B,
the scaling law (\ref{eq:scalingLawZ}) appears valid.

Finally, panel C of Fig.~\ref{fig:manyEeAllZ} corresponds to
\begin{equation}
\tau_{\sL}(0) = -0.7 \;, \qquad
\delta_{\sL}(0) = \delta_{\sR}^{-\frac{3}{2}} \;, \qquad
\tau_{\sR}(0) = -3.308423793 \;, \qquad
\delta_{\sR}(0) = 1.659870677 \;,
\label{eq:paramC}
\end{equation}
and
\begin{equation}
\cX = \sR \sL \sR \sL \sR \;, \qquad
\cY = \sL \sR \;.
\label{eq:XYC}
\end{equation}
Here $\frac{1}{\lambda_2(0)} \approx 0.4507$,
and $\frac{\ee_{K+1}}{\ee_K}$ appears to limit to this value.

%=====================================================================
\section{Multistability due to a repeated unit eigenvalue}
\label{sec:codim4}
\setcounter{equation}{0}

This section investigates the coexistence of a large number of stable $\cX^k \cY$-cycles
due to $M_{\cX}$ having a repeated unit eigenvalue.
In \S\ref{sub:codim4} it is shown that infinite coexistence via this mechanism is a codimension-four phenomenon.
Additional properties of the $\cX^k \cY$-cycles are discussed in \S\ref{sub:structure4}
and three examples of codimension-four points are given in \S\ref{sub:examples4}.

We first show that if the geometric multiplicity of the repeated unit eigenvalue of
$M_{\cX}$ is $1$ (that is, the corresponding eigenspace is one-dimensional, as is generically the case)\removableFootnote{
If the geometric multiplicity is $2$,
then $M_{\cX}$ must be the identity matrix and it is not clear that this is possible.
Certainly if $\tau_{\sL} = \delta_{\sL} = \tau_{\sR} = \delta_{\sR} = 1$,
then $M_{\cX} = I$ with $n_{\cX} = 6$.
But for this example $A_{\sL} = A_{\sR}$.
I have the following conjecture:
{\em If $A_L \ne A_R$, then for any primitive $\cX$, $M_{\cX} \ne I$.}
},
then the corresponding eigenspace cannot be tangent to the switching manifold.

%.....................................................................
\begin{lemma}
Let $\cS[k] = \cX^k \cY$, where $\cX$ is primitive and $\cX_0 \ne \cY_0$.
Let $\tau_{\sL}, \delta_{\sL}, \tau_{\sR}, \delta_{\sR} \in \mathbb{R}$ and $\mu \ne 0$.
Suppose there exist infinitely many values of $k \ge 1$ for which (\ref{eq:f})
exhibits a unique, admissible, stable $\cS[k]$-cycle that has no points on the switching manifold.
Suppose that $M_{\cX}$ has a repeated unit eigenvalue of geometric multiplicity $1$.
Then $[0,1]^{\sf T}$ is not an eigenvector of $M_{\cX}$.
\label{le:eigenvector}
\end{lemma}

\begin{proof}
Let $\left( x^{\cS[k]}_i, y^{\cS[k]}_i \right)$ denote the points of an $\cS[k]$-cycle.
Since the symbols $\cS[k]_{(k-1) n_{\cX}} = \cX_0$ and $\cS[k]_{k n_{\cX}} = \cY_0$
are different, by assumption,
if the $\cS[k]$-cycle is admissible with no points on the switching manifold,
then the signs of $x^{\cS[k]}_{(k-1) n_{\cX}}$ and $x^{\cS[k]}_{k n_{\cX}}$ must be different.
Since admissibility is assumed to hold for infinitely many values of $k$,
the signs of $x^{\cS[k]}_{(k-1) n_{\cX}}$ and $x^{\cS[k]}_{k n_{\cX}}$ must be different
for arbitrarily large values of $k$.

In Appendix \ref{app:eigenvector} we show that if
$[0,1]^{\sf T}$ is an eigenvector of $M_{\cX}$\removableFootnote{
It follows immediately that $x \mapsto x + \rho_1$ under $\cX$,
but this does not provide an immediate contradiction.
We actually have to solve for the points $\left( x^{\cS[k]}_i, y^{\cS[k]}_i \right)$.
},
then for all $j = 0,\ldots,k$,
\begin{equation}
x^{\cS[k]}_{j n_{\cX}} = \rho_1 j
+ \frac{\rho_1 \gamma_{11}}{1-\gamma_{11}} k
+ \frac{\sigma_1}{1-\gamma_{11}} \;,
\label{eq:xSkj}
\end{equation}
for some constants $\rho_1, \gamma_{11}, \sigma_1 \in \mathbb{R}$,
with $\gamma_{11} \ne 1$.
(Take care to note that these constants differ
from those used in the remainder of this section.)
By (\ref{eq:xSkj}) we have, in particular,
\begin{equation}
x^{\cS[k]}_{(k-1) n_{\cX}} =
\frac{\rho_1}{1-\gamma_{11}} k
+ \frac{\sigma_1}{1-\gamma_{11}} - \rho_1 \;, \qquad
x^{\cS[k]}_{k n_{\cX}} =
\frac{\rho_1}{1-\gamma_{11}} k
+ \frac{\sigma_1}{1-\gamma_{11}} \;.
\nonumber
\end{equation}
Therefore, regardless of the values of
$\rho_1$, $\gamma_{11}$ and $\sigma_1$,
the signs of $x^{\cS[k]}_{(k-1) n_{\cX}}$ and $x^{\cS[k]}_{k n_{\cX}}$
are the same for large values of $k$.
This contradicts the admissibility assumption,
hence $[0,1]^{\sf T}$ cannot be an eigenvector of $M_{\cX}$.
\end{proof}

%----------------------------------------------------------------------
\subsection{A codimension-four scenario for infinite coexistence}
\label{sub:codim4}

Here we first introduce an alternate coordinate system
in which calculations are simplified because $M_{\cX}$ is transformed to a triangular matrix,
and then state and describe consequences of infinite coexistence due a repeated unit eigenvalue.

Suppose $M_{\cX}$ has a repeated unit eigenvalue with a one-dimensional eigenspace
equal to all non-zero scalar multiples of $[1,\nu]^{\sf T}$, for some $\nu \in \mathbb{R}$.
In order to transform $M_{\cX}$ to a triangular matrix, we let
\begin{equation}
Q = \left[ \begin{array}{cc} 1 & 0 \\ \nu & 1 \end{array} \right] \;,
\label{eq:Q}
\end{equation}
and consider the coordinate change
\begin{equation}
\left[ \begin{array}{c} u \\ v \end{array} \right] =
Q^{-1} \left[ \begin{array}{c} x \\ y \end{array} \right] \;.
\label{eq:uv}
\end{equation}
We let $w = (u,v)$,
and for any periodic symbol sequence $\cS$, let $g^{\cS}$ denote $f^{\cS}$ in $(u,v)$-coordinates.
Then
\begin{equation}
g^{\cX}(w) = \left[ \begin{array}{cc} 1 & \omega_{12} \\ 0 & 1 \end{array} \right] w
+ \left[ \begin{array}{c} \rho_1 \\ \rho_2 \end{array} \right] \;,
\label{eq:gX}
\end{equation}
where $\omega_{12} \ne 0$ (because the eigenspace of $M_{\cX}$ is one-dimensional)
and $\rho_1, \rho_2 \in \mathbb{R}$.
The matrix part of (\ref{eq:gX}) represents the transformation of $M_{\cX}$ (the matrix part of $f^{\cS}$)
to $(u,v)$-coordinates.
Also we write
\begin{equation}
g^{\cY}(w) = \left[ \begin{array}{cc} \gamma_{11} & \gamma_{12} \\ \gamma_{21} & \gamma_{22} \end{array} \right] w
+ \left[ \begin{array}{c} \sigma_1 \\ \sigma_2 \end{array} \right] \;,
\label{eq:gY}
\end{equation}
for some real-valued constants $\gamma_{ij}$, $\sigma_1$ and $\sigma_2$.
Note that unlike for the $\ee$-dependent $(u,v)$-coordinate system used in \S\ref{sec:codimperturb3},
here the switching manifold is simply $u=0$.

%.....................................................................
\begin{theorem}
Let $\cS[k] = \cX^k \cY$, where $\cX$ is primitive and $\cX_0 \ne \cY_0$.
Let $\tau_{\sL}, \delta_{\sL}, \tau_{\sR}, \delta_{\sR} \in \mathbb{R}$ and $\mu \ne 0$.
Suppose there exist infinitely many values of $k \ge 1$ for which (\ref{eq:f})
exhibits a unique, admissible, stable $\cS[k]$-cycle that has no points on the switching manifold.
Suppose that $M_{\cX}$ has a repeated unit eigenvalue of geometric multiplicity $1$.
Then the coefficients of $g^{\cX}$ (\ref{eq:gX}) and $g^{\cY}$ (\ref{eq:gY}) satisfy
\begin{enumerate}[label=\roman{*}),ref=\roman{*}]
\item
\label{it:gamma21}
$\gamma_{21} = 0$;
\item
\label{it:rho2}
$\rho_2 \ne 0$;
\item
\label{it:gamma22}
$\gamma_{22} = -1$.
\end{enumerate}
Furthermore, the stability multipliers of the $\cS[k]$-cycles are $\gamma_{11}$ and $-1$.
The $\cS[k]$-cycles are not attracting and if $\delta_{\sL} = \delta_{\sR} = 1$, then $\gamma_{11} = -1$.
\label{th:codim4}
\end{theorem}

The assumption that $M_{\cX}$ has a repeated unit eigenvalue is a codimension-two restriction
on the parameter values of (\ref{eq:f}).
Assuming that there is no degeneracy in the conditions $\gamma_{21} = 0$ and $\gamma_{22} = -1$,
it follows that Theorem \ref{th:codim4} describes a scenario that is at least codimension-four.
From the results of the next section we conclude that this scenario is codimension-four.
Indeed, higher codimension scenarios are not possible for the two-dimensional border-collision form with $\mu \ne 0$
because there are only four parameters that we may vary.

\begin{proof}[Proof of Theorem \ref{th:codim4}]
We prove parts (\ref{it:gamma21})-(\ref{it:gamma22}) in order.
\begin{enumerate}[label=\roman{*})]
\item
We first combine (\ref{eq:gX}) and (\ref{eq:gY}) in order to obtain an expression for the map $g^{\cS[k]}$.
Straight-forward calculations reveal that powers of (\ref{eq:gX}) are given by\removableFootnote{
With $\Omega = \left[ \begin{array}{cc} 1 & \omega_{12} \\ 0 & 1 \end{array} \right]$,
$g^{\cX^k}(w) = \Omega^k w + \sum_{m=0}^{k-1} \Omega^m
\left[ \begin{array}{c} \rho_1 \\ \rho_2 \end{array} \right]$.
}
\begin{equation}
g^{\cX^k}(w) = \left[ \begin{array}{cc} 1 & \omega_{12} k \\ 0 & 1 \end{array} \right] w
+ \left[ \begin{array}{c} \rho_1 k + \frac{\rho_2 \omega_{12}}{2} k (k-1) \\ \rho_2 k \end{array} \right] \;.
\label{eq:gXk}
\end{equation}
The composition of (\ref{eq:gY}) and (\ref{eq:gXk}) is
\begin{equation}
g^{\cS[k]}(w) = \left[ \begin{array}{cc}
\gamma_{11} & \gamma_{11} \omega_{12} k + \gamma_{12} \\
\gamma_{21} & \gamma_{21} \omega_{12} k + \gamma_{22}
\end{array} \right] w
+ \left[ \begin{array}{c}
\gamma_{11} \rho_1 k + \frac{\gamma_{11} \rho_2 \omega_{12}}{2} k (k-1) + \gamma_{12} \rho_2 k + \sigma_1 \\
\gamma_{21} \rho_1 k + \frac{\gamma_{21} \rho_2 \omega_{12}}{2} k (k-1) + \gamma_{22} \rho_2 k + \sigma_2
\end{array} \right] \;.
\label{eq:gSk}
\end{equation}
Since (\ref{eq:uv}) is a non-singular coordinate transformation,
the matrix parts of $f^{\cS[k]}$ and $g^{\cS[k]}$ have identical eigenvalues.
Therefore, ${\rm trace} \left( M_{\cS[k]} \right) = \gamma_{21} \omega_{12} k + \gamma_{11} + \gamma_{22}$.
But $\cS[k]$-cycles are assumed to be stable for arbitrarily large values of $k$.
This is only possible if ${\rm trace} \left( M_{\cS[k]} \right) \not\to \infty$,
therefore we must have $\gamma_{21} = 0$ (since $\omega_{12} \ne 0$).
\item
With $\gamma_{21} = 0$, the matrix part of $g^{\cS[k]}$ is upper triangular
with diagonal entries, $\gamma_{11}$ and $\gamma_{22}$.
Hence these are the eigenvalues of $M_{\cS[k]}$.
By the results of \S\ref{sec:backg},
we require $\gamma_{11}, \gamma_{22} \ne 1$ in order for $\cS[k]$-cycles to be unique.

In $(u,v)$-coordinates, we denote $\cS[k]$-cycles by
$w^{\cS[k]}_i = \left( u^{\cS[k]}_i, v^{\cS[k]}_i \right)$,
for $i = 0, \ldots, k n_{\cX} + n_{\cY} - 1$.
%where $n_{\cX}$ and $n_{\cY}$ are the lengths of the words $\cX$ and $\cY$.
Then $w^{\cS[k]}_0$ is the unique fixed point of (\ref{eq:gSk}).
By computing this point and iterating it $j$ times under $g^{\cX}$ (\ref{eq:gX}), we obtain the formula\removableFootnote{
See {\sc calcSymSkCycle.m}
}\removableFootnote{
Here my error terms are of a different nature than everywhere else in this paper,
but I don't see how to improve it.
}
\begin{equation}
w^{\cS[k]}_{j n_{\cX}} = \left[ \begin{array}{c}
\frac{\omega_{12} \rho_2}{2} \left( j^2 + \frac{2 \gamma_{22}}{1-\gamma_{22}} jk
+ \frac{\gamma_{11} \left( 1 + \gamma_{22} \right)}{\left( 1 - \gamma_{11} \right) \left( 1 - \gamma_{22} \right)} k^2
\right) + \cO \left( j,k \right) \\
\rho_2 j + \frac{\gamma_{22} \rho_2}{1 - \gamma_{22}} k + \frac{\sigma_2}{1-\gamma_{22}}
\end{array} \right] \;,
\label{eq:wSkj}
\end{equation}
valid for $j = 0,\ldots,k$.

Next we use the assumption that $\cS[k]$-cycles are admissible to show that $\rho_2 \ne 0$.
Note, the symbols $\cS[k]_{(k-1) n_{\cX}} = \cX_0$ and $\cS[k]_{k n_{\cX}} = \cY_0$ are different, by assumption,
therefore $w^{\cS[k]}_{(k-1) n_{\cX}}$ and $w^{\cS[k]}_{k n_{\cX}}$ lie on different sides of the switching manifold ($u=0$)
for infinitely many values of $k$.
If $\rho_2 = 0$, then by (\ref{eq:wSkj}), 
$u^{\cS[k]}_{(k-1) n_{\cX}} = \eta k + \upsilon_1$ and
$u^{\cS[k]}_{k n_{\cX}} = \eta k + \upsilon_2$,
for some constants $\eta, \upsilon_1, \upsilon_2 \in \mathbb{R}$.
We must have $\eta \ne 0$, so that $w^{\cS[k]}_{k n_{\cX}}$ actually differs with $k$.
But in this case $u^{\cS[k]}_{(k-1) n_{\cX}}$ and $u^{\cS[k]}_{k n_{\cX}}$ have the same sign for large $k$, which is a contradiction.
\item
By (\ref{eq:wSkj}),
\begin{align}
u^{\cS[k]}_{(k-1) n_{\cX}} &=
\frac{\omega_{12} \rho_2 \left( 1 + \gamma_{22} \right)}
{2 \left( 1 - \gamma_{11} \right) \left( 1 - \gamma_{22} \right)} k^2 + \cO(k) \;, \\
u^{\cS[k]}_{k n_{\cX}} &=
\frac{\omega_{12} \rho_2 \left( 1 + \gamma_{22} \right)}
{2 \left( 1 - \gamma_{11} \right) \left( 1 - \gamma_{22} \right)} k^2 + \cO(k) \;,
\end{align}
where $\omega_{12}, \rho_2 \ne 0$.
Hence in order for $u^{\cS[k]}_{(k-1) n_{\cX}}$ and $u^{\cS[k]}_{k n_{\cX}}$ to have different signs for large values of $k$
we must have $\gamma_{22} = -1$.
\end{enumerate}
Finally, the stability multipliers of the $\cS[k]$-cycles are the eigenvalues of $M_{\cS[k]}$,
which are $\gamma_{11}$ and $\gamma_{22} = -1$, as mentioned above.
Since $\gamma_{22} = -1$ does not have modulus less $1$,
the $\cS[k]$-cycles are not attracting (specifically, there is equality in (\ref{eq:stabConditionPD})).
If $\delta_{\sL} = \delta_{\sR} = 1$,
then (\ref{eq:f}) is area-preserving and so the eigenvalues of $M_{\cS[k]}$ must multiply to $1$.
Thus in this case $\gamma_{11} = -1$.
\end{proof}

%----------------------------------------------------------------------
\subsection{The structure of the $\cX^k \cY$-cycles.}
\label{sub:structure4}

In $(u,v)$-coordinates, we denote points of $\cS[k]$-cycles by
$w^{\cS[k]}_i$, for $i = 0,\ldots,k n_{\cX} + n_{\cY} - 1$,
where $n_{\cX}$ and $n_{\cY}$ are the lengths of the words $\cX$ and $\cY$.
In view of part (\ref{it:gamma22}) of Theorem \ref{th:codim4},
we substitute $\gamma_{22} = -1$ into (\ref{eq:wSkj}) which gives
\begin{equation}
w^{\cS[k]}_{j n_{\cX}} = \left[ \begin{array}{c}
-\frac{\omega_{12} \rho_2}{2} j(k-j) + \cO \left( j,k \right) \\
\rho_2 \left( j - \frac{1}{2} k \right) + \frac{\sigma_2}{2}
\end{array} \right] \;,
\label{eq:wSkj2}
\end{equation}
valid for all $j = 0,\ldots,k$.
Since $\cS[k]_{j n_{\cX}} = \cX_0$, for all $j = 0,\ldots,k-1$,
if $\cS[k]$-cycles are admissible for large $k$, and, say, $\cX_0 = \sR$, then by (\ref{eq:wSkj2})
we must have $\omega_{12} \rho_2 < 0$, because $\omega_{12} \ne 0$ and $\rho_2 \ne 0$.
The converse is true if $\cX_0 = \sL$, that is,
\begin{equation}
{\rm if~} \cX_0 = \sR, {\rm ~then~} \omega_{12} \rho_2 < 0 \;, \qquad
{\rm if~} \cX_0 = \sL, {\rm ~then~} \omega_{12} \rho_2 > 0 \;.
\label{eq:signomega12rho2}
\end{equation}

By (\ref{eq:wSkj2}), for fixed $k$ the points $w^{\cS[k]}_{j n_{\cX}}$,
for $j=0,\ldots,k$, lie on a parabola.
The point furthest from the origin has $j \approx \frac{k}{2}$,
and for large $k$ the direction of this point from the origin is roughly the
same as the direction of the vector $[1,0]^{\sf T}$.
We say that the points $w^{\cS[k]}_{j n_{\cX}}$ align asymptotically with $[1,0]^{\sf T}$,
because the size of the convex hull of the points is proportional to $k^2$,
yet each point is only an $\cO(k)$ distance from a scalar multiple of $[1,0]^{\sf T}$.

To investigate the behavior of the other points of the $\cS[k]$-cycles, let us write
\begin{equation}
\left( g^{\cX_{m-1}} \circ \cdots \circ g^{\cX_0} \right)(w) =
\left[ \begin{array}{cc} \psi_{11m} & \psi_{12m} \\ \psi_{21m} & \psi_{22m} \end{array} \right] w + 
\left[ \begin{array}{c} \chi_{1m} \\ \chi_{2m} \end{array} \right] \;,
\label{eq:gXtruncated}
\end{equation}
for six $m$-dependent coefficients $\psi_{ijm}$, $\chi_{1m}$ and $\chi_{2m}$.
The map (\ref{eq:gXtruncated}) gives the image of a point $w$ under $m$ iterations of (\ref{eq:f}) following $\cX$.
For $m=0$, (\ref{eq:gXtruncated}) is the identity map.
For all $j = 0,\ldots,k-1$ and $m = 0,\ldots,n_{\cX}-1$,
\begin{equation}
w^{\cS[k]}_{j n_{\cX} + m} =
\left[ \begin{array}{cc} \psi_{11m} & \psi_{12m} \\ \psi_{21m} & \psi_{22m} \end{array} \right]
w^{\cS[k]}_{j n_{\cX}} + 
\left[ \begin{array}{c} \chi_{1m} \\ \chi_{2m} \end{array} \right] \;.
\label{eq:wSkji}
\end{equation}
In view of (\ref{eq:wSkj2}), the points $w^{\cS[k]}_{j n_{\cX} + m}$
align asymptotically with $\left[ \psi_{11m}, \psi_{21m} \right]^{\sf T}$.
By multiplying this vector by $Q$,
we see that in $(x,y)$-coordinates these points align asymptotically with
$\left[ \psi_{11m}, \psi_{11m} \nu + \psi_{21m} \right]^{\sf T}$.
It is a straight-forward exercise to show that this vector is the unique eigenvector of $M_{\cX^{(m)}}$,
where $\cX^{(m)}$ denotes the $m^{\rm th}$ left shift permutation of $\cX$.
We now show that none of these vectors are tangent to the switching manifold.

%.....................................................................
\begin{lemma}
Suppose that (\ref{eq:f}) satisfies the conditions of Theorem \ref{th:codim4}.
Then for all $m = 0,\ldots,n_{\cX}-1$, we have $\psi_{11m} \ne 0$,
where $\psi_{11m}$ is defined by (\ref{eq:gXtruncated}).
\label{le:psi11}
\end{lemma}

\begin{proof}
Suppose for a contradiction $\psi_{11m} = 0$.
Then by (\ref{eq:wSkji}), $u^{\cS[k]}_{j n_{\cX} + m} = \psi_{12m} v^{\cS[k]}_{j n_{\cX}} + \chi_{1m}$,
where $\psi_{12m} \ne 0$ (for otherwise the matrix part of (\ref{eq:gXtruncated})
would be singular, which is not possible because $M_{\cX}$ is nonsingular).
It follows that the signs of $u^{\cS[k]}_m$ and $u^{\cS[k]}_{(k-1) n_{\cX} + m}$ are different for large $k$,
because, by (\ref{eq:wSkj2}), the signs of $v^{\cS[k]}_0$ and $v^{\cS[k]}_{(k-1) n_{\cX}}$
are different for large $k$.
This contradicts the assumption that $\cS[k]$-cycles are admissible for arbitrarily large values of $k$,
because $\cS[k]_m = \cS[k]_{(k-1) n_{\cX} + m} = \cX_m$.
\end{proof}

Note that from (\ref{eq:wSkj2}) and (\ref{eq:wSkji}) we have
\begin{equation}
u^{\cS[k]}_{j n_{\cX} + m} = -\frac{\psi_{11m} \omega_{12} \rho_2}{2} j(k-j) + \cO(j,k) \;.
\label{eq:uSkjm}
\end{equation}
Therefore (\ref{eq:signomega12rho2}) may be generalized to:
\begin{equation}
{\rm if~} \cX_m = \sR, {\rm ~then~} \psi_{11m} \omega_{12} \rho_2 < 0 \;, \qquad
{\rm if~} \cX_m = \sL, {\rm ~then~} \psi_{11m} \omega_{12} \rho_2 > 0 \;.
\label{eq:signkappa11omega12rho2}
\end{equation}
%Note that $\psi_{110} = 1$.

%----------------------------------------------------------------------
\subsection{Examples of codimension-four points}
\label{sub:examples4}

Here we identify values of $\tau_{\sL}, \delta_{\sL}, \tau_{\sR}$ and $\delta_{\sR}$,
for which (\ref{eq:f}) has infinitely many stable $\cS[k]$-cycles
for three different combinations of $\cX$ and $\cY$.
To find such values, given $\cX$ and $\cY$,
we first use the assumption that $M_{\cX}$ has a repeated unit eigenvalue
to obtain two conditions on the parameter values.
We then construct $\Gamma = Q^{-1} M_{\cY} Q$ (the matrix part of $g^{\cY}$ (\ref{eq:gY})).
By Theorem \ref{th:codim4}, $\gamma_{21} = 0$ and $\gamma_{22} = -1$,
where $\gamma_{ij}$ denotes the $(i,j)$-element of $\Gamma$.
This gives us two further conditions on the parameter values.
Solving all four conditions simultaneously
produces parameter values for which (\ref{eq:f}) potentially has infinitely many stable $\cS[k]$-cycles.
For each example we can calculate the $\cS[k]$-cycles explicitly
to verify that they are stable and admissible for infinitely many values of $k$\removableFootnote{
Calculations for all three examples are done symbolically in {\sc goSymGenNewZero.m}.
}.

%-----------------------------------------------------------------------------
\subsubsection*{An example with $n_{\cX} = 3$}

In \cite{Si10}, by numerically computing Arnold tongues and simply looking at where they overlap,
it was found that with
\begin{equation}
\textstyle
\tau_{\sL} = \frac{34}{25} \cos \left( \frac{19 \pi}{25} \right) \approx -0.9914 \;, \quad
\delta_{\sL} = 0.4624 \;, \quad %\left( \frac{17}{25} \right)^2
\tau_{\sR} = \frac{5}{2} \cos \left( \frac{27 \pi}{50} \right) \approx -0.3133 \;, \quad
\delta_{\sR} = 1.5625 \;, %\left( \frac{5}{4} \right)^2
\label{eq:paramSi10}
\end{equation}
and $\mu = 1$, (\ref{eq:f}) has six attracting $\cS[k]$-cycles, where\removableFootnote{
In \cite{Si10} I have written $\sL \sR \sL \sL (\sR \sL \sL)^k$, which is wrong!
}
\begin{equation}
\cX = \sR^2 \sL \;, \qquad
\cY = \sL \sR \sL^2 \;.
\label{eq:XYA}
\end{equation}
This suggests that the combination (\ref{eq:XYA}) may give
infinite coexistence for a set of parameter values near (\ref{eq:paramSi10})\removableFootnote{
With $\cX = \sR \sL$ I haven't been able to find an example using intuitive choices for $\cY$.
}.
Straight-forward calculations for the matrices $A_{\sL}$ and $A_{\sR}$
reveal that with $\cX = \sR^2 \sL$,
$M_{\cX}$ has a repeated unit eigenvalue when\removableFootnote{
With (\ref{eq:tauDeltaEig11A}) we have
$\det \left( M_{\cS[k]} \right) = \delta_{\sR}^{-5}$,
and thus we must have $\delta_{\sR} \ge 1$ in order for $\cS[k]$-cycles to be stable.
Note, this observation is not required here.
}
\begin{equation}
\tau_{\sL} = -\frac{\tau_{\sR} \delta_{\sR}^3 +
2 \delta_R^2 + \tau_{\sR}}{\delta_{\sR}^2 \left( \delta_{\sR} - \tau_{\sR}^2 \right)} \;, \qquad
\delta_{\sL} = \frac{1}{\delta_{\sR}^2} \;.
\label{eq:tauDeltaEig11A}
\end{equation}
With (\ref{eq:tauDeltaEig11A}), $[1,\nu]^{\sf T}$ is an eigenvector of $M_{\cX}$
with $\nu = \frac{\delta_{\sR}-\tau_{\sR}^2}{\tau_{\sR}+\delta_{\sR}^2}$.
By evaluating $\Gamma = Q^{-1} M_{\cY} Q$, where $Q$ is given by (\ref{eq:Q}) and $\cY = \sL \sR \sL^2$, we obtain
\begin{align}
\gamma_{21} &= \frac{ \left( \delta_{\sR} - \tau_{\sR} + 1 \right)
\left( \delta_{\sR}^2 + \delta_{\sR} \tau_{\sR} - \delta_{\sR} + \tau_{\sR}^2 + \tau_{\sR} + 1 \right)}
{\delta_{\sR}^3 \left( \delta_{\sR}^2 + \tau_{\sR} \right)^2 \left( \delta_{\sR} - \tau_{\sR}^2 \right)^2} \nonumber \\
&\quad\times \left( \delta_{\sR}^6 \tau_{\sR}^2 + 2 \delta_{\sR}^5 \tau_{\sR} + 2 \delta_{\sR}^4 \tau_{\sR}^3 +
\delta_{\sR}^4 + 3 \delta_{\sR}^3 \tau_{\sR}^2 + 3 \delta_{\sR}^2 \tau_{\sR}^4 +
2 \delta_{\sR}^2 \tau_{\sR} - \delta_{\sR} \tau_{\sR}^6 + 2 \delta_{\sR} \tau_{\sR}^3 + \tau_{\sR}^2 \right) \;, \label{eq:gamma21A} \\
\gamma_{22} &= \frac{\delta_{\sR}^5 \tau_{\sR} + \delta_{\sR}^4 + 2 \delta_{\sR}^3 \tau_{\sR}^2 +
\delta_{\sR}^2 \tau_{\sR} + 3 \delta_{\sR} \tau_{\sR}^3 - \tau_{\sR}^5 + \tau_{\sR}^2}
{\delta_{\sR}^2 \left( \delta_{\sR}^2 + \tau_{\sR} \right)
\left( \tau_{\sR}^2 - \delta_{\sR} \right)} \;. \label{eq:gamma22A}
\end{align}
By plotting (\ref{eq:gamma21A}) and (\ref{eq:gamma22A}),
it is quickly seen that there are only two points, $\left( \tau_{\sR},\delta_{\sR} \right)$,
at which $\gamma_{21} = 0$ and $\gamma_{22} = -1$.
One of these points is $\left( \tau_{\sR},\delta_{\sR} \right) = \left( 1 + \sqrt{2},1 \right)$,
but at this point all $\cS[k]$-cycles are virtual (for both $\mu = 1$ and $\mu = -1$), so we do not consider it further.
The other point is $\left( \tau_{\sR},\delta_{\sR} \right) = \left( 1 - \sqrt{2},1 \right)$,
which with (\ref{eq:tauDeltaEig11A}) gives
\begin{equation}
\tau_{\sL} = -\sqrt{2} \;, \qquad
\delta_{\sL} = 1 \;, \qquad
\tau_{\sR} = 1-\sqrt{2} \;, \qquad
\delta_{\sR} = 1 \;.
\label{eq:paramA}
\end{equation}
Fig.~\ref{fig:coexistNeutralA} shows $\cS[k]$-cycles of (\ref{eq:f}) with (\ref{eq:paramA}) and $\mu = 1$.
In this figure the eigenspace of each $M_{\cX^{(m)}}$ is also shown.
As discussed in \S\ref{sub:structure4}, the $\cS[k]$-cycles
align with these eigenspaces with increasing values of $k$.

In Appendix \ref{app:verify3} we prove the following result stating that the
$\cS[k]$-cycles of Fig.~\ref{fig:coexistNeutralA} are admissible for all $k \ge 8$, and are stable
(but not attracting) because each $\cS[k]$-cycle has a repeated stability multiplier of $-1$.
The stability multipliers are straight-forward to compute.
Admissibility is verified by deriving an explicit expression for each point of an $\cS[k]$-cycle
and demonstrating that all points lie on the correct side of the switching manifold only when $k \ge 8$.

%.....................................................................
\begin{proposition}
Let $\cX = \sR^2 \sL$, $\cY = \sL \sR \sL^2$ and $\cS[k] = \cX^k \cY$.
Then the map (\ref{eq:f}) with (\ref{eq:paramA}) and $\mu = 1$
has a unique $\cS[k]$-cycle for all $k \ge 1$,
that is admissible and stable for all $k \ge 8$,
and is virtual otherwise.
\label{pr:verify3}
\end{proposition}

%%%%%%%%%%%%%%%%%%%%%%%%%%%%%%%%%%%%%%%%%%%%%%%%%%%%%%%%%%%%%
\begin{figure}[t!]
\begin{center}
\setlength{\unitlength}{1cm}
\begin{picture}(15,7.5)
\put(0,0){\includegraphics[height=7.5cm]{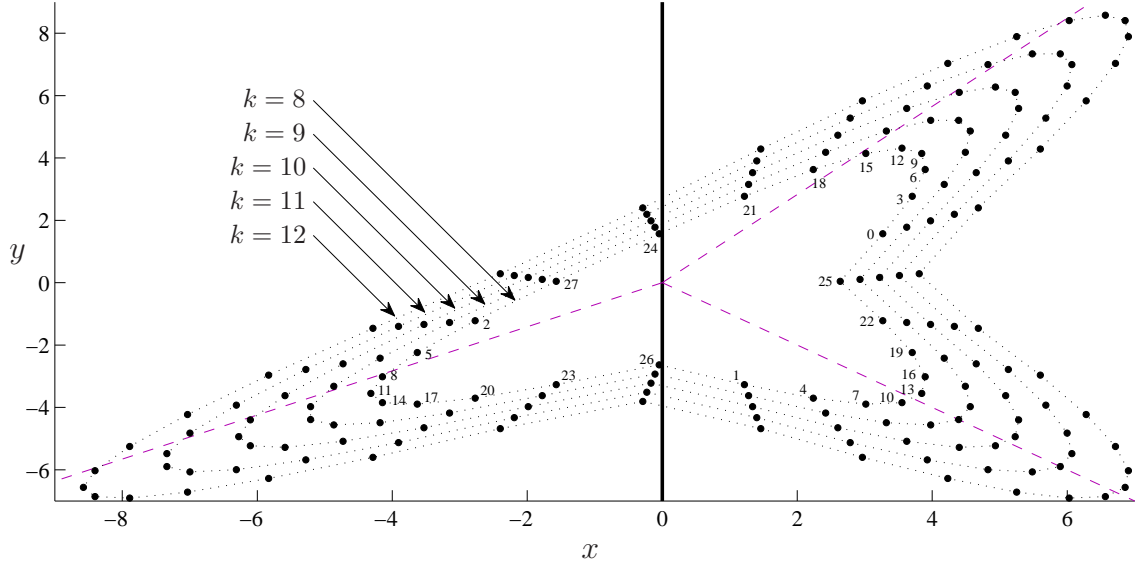}}
\put(7.6,.1){$x$}
\put(0,4.1){$y$}
\put(3.1,6.1){\footnotesize $k = 8$}
\put(3.1,5.65){\footnotesize $k = 9$}
\put(2.93,5.2){\footnotesize $k = 10$}
\put(2.93,4.75){\footnotesize $k = 11$}
\put(2.93,4.3){\footnotesize $k = 12$}
\end{picture}
\caption{
A phase portrait of the two-dimensional border-collision normal form (\ref{eq:f})
with $\mu = 1$ and remaining parameter values given by (\ref{eq:paramA}).
For all $k \ge 8$, there is a unique, stable, admissible $\cS[k]$-cycle,
where $\cS[k] = (\sR^2 \sL)^k \sL \sR \sL^2$.
The $\cS[k]$-cycles are nested and grow in size without bound.
For clarity, only the $\cS[k]$-cycles for $k = 8, \ldots, 12$ are shown.
In order to distinguish these five periodic solutions,
each point is connected by a dotted line segment to its third iterate under (\ref{eq:f}).
Each point of the $\cS[8]$-cycle (which has period $28$)
is numbered by its index $i$, which corresponds to the $i^{\rm th}$ symbol of $\cS[8]$
as well as the number of iterations of (\ref{eq:f}) from the $0^{\rm th}$-point.
The one-dimensional eigenspace of each $M_{\cX^{(m)}}$, for $m = 0,1,2$,
is shown as a dashed line on the side of the switching manifold corresponding to the symbol $\cX_m$.
(Here $\cX = \sR^2 \sL$, thus
$M_{\cX^{(0)}} \equiv M_{\cX} = A_{\sL} A_{\sR}^2$,
$M_{\cX^{(1)}} = A_{\sR} A_{\sL} A_{\sR}$, and
$M_{\cX^{(2)}} = A_{\sR}^2 A_{\sL}$.)
\label{fig:coexistNeutralA}
}
\end{center}
\end{figure}
%%%%%%%%%%%%%%%%%%%%%%%%%%%%%%%%%%%%%%%%%%%%%%%%%%%%%%%%%%%%%

%-----------------------------------------------------------------------------
\subsubsection*{An example with $n_{\cX} = 4$}

Replacing $\sR$ with $\sR^2$ in (\ref{eq:XYA}) produces
\begin{equation}
\cX = \sR^3 \sL \;, \qquad
\cY = \sL \sR^2 \sL^2 \;.
\label{eq:XYB}
\end{equation}
Here $M_{\cX}$ has a repeated unit eigenvalue when
\begin{equation}
\tau_{\sL} = \frac{2 + (\delta_{\sL} + \delta_{\sR})(\tau_{\sR}^2 - \delta_{\sR})}
{\tau_{\sR} (\tau_{\sR}^2 - 2 \delta_{\sR})} \;, \qquad
\delta_{\sL} = \frac{1}{\delta_{\sR}^3} \;.
\label{eq:tauDeltaEig11B}
\end{equation}
With (\ref{eq:tauDeltaEig11B}), $[1,\nu]^{\sf T}$ is an eigenvector of $M_{\cX}$
with $\nu = \frac{-\tau_{\sR}^3+2 \delta_{\sR} \tau_{\sR}}{\delta_{\sR}^3-\delta_{\sR}+\tau_{\sR}^2}$.
By evaluating $\Gamma = Q^{-1} M_{\cY} Q$
we obtain expressions for $\gamma_{21}$ and $\gamma_{22}$ in terms of $\tau_{\sR}$ and $\delta_{\sR}$
that are too complicated to include here\removableFootnote{
\begin{align}
\gamma_{21} &= \frac{(\delta_{\sR} + \tau_{\sR} + 1)
(\delta_{\sR} - \tau_{\sR} + 1)
(\delta_{\sR}^2 - 2 \delta_{\sR} + \tau_{\sR}^2 + 1)}
{\delta_{\sR}^4 \tau_{\sR}^2 (2 \delta_{\sR} - \tau_{\sR}^2)^2 (\delta_{\sR}^3 - \delta_{\sR} + \tau_{\sR}^2)^2} \nonumber \\
&\quad\times \Big( \delta_{\sR}^{10} - 2 \delta_{\sR}^9 \tau_{\sR}^2 + \delta_{\sR}^8 \tau_{\sR}^4 - 4 \delta_{\sR}^8 +
8 \delta_{\sR}^7 \tau_{\sR}^2 - 6 \delta_{\sR}^6 \tau_{\sR}^4 + 6 \delta_{\sR}^6 + 2 \delta_{\sR}^5 \tau_{\sR}^6 - 
16 \delta_{\sR}^5 \tau_{\sR}^2 \nonumber \\
&\quad+ 22 \delta_{\sR}^4 \tau_{\sR}^4 - 4 \delta_{\sR}^4 - 17 \delta_{\sR}^3 \tau_{\sR}^6 +
8 \delta_{\sR}^3 \tau_{\sR}^2 + 7 \delta_{\sR}^2 \tau_{\sR}^8 - 6 \delta_{\sR}^2 \tau_{\sR}^4 + \delta_{\sR}^2 -
\delta_{\sR} \tau_{\sR}^10 + 2 \delta_{\sR} \tau_{\sR}^6 - 2 \delta_{\sR} \tau_{\sR}^2 + \tau_{\sR}^4 \Big) \;, \\
\gamma_{22} &= \frac{\delta_{\sR}^8 - \delta_{\sR}^7 \tau_{\sR}^2 - 3 \delta_{\sR}^6 + 4 \delta_{\sR}^5 \tau_{\sR}^2
- 2 \delta_{\sR}^4 \tau_{\sR}^4 + 3 \delta_{\sR}^4 - 9 \delta_{\sR}^3 \tau_{\sR}^2 + 11 \delta_{\sR}^2 \tau_{\sR}^4
- \delta_{\sR}^2 - 6 \delta_{\sR} \tau_{\sR}^6 + 2 \delta_{\sR} \tau_{\sR}^2 + \tau_{\sR}^8 - \tau_{\sR}^4}
{\delta_{\sR}^3 \tau_{\sR} (- \tau_{\sR}^2 + 2 \delta_{\sR}) (\delta_{\sR}^3 - \delta_{\sR} + \tau_{\sR}^2)} \;.
\end{align}
}.
Numerical computations of these expressions indicate that
there is a unique choice of parameter values for which $\gamma_{21} = 0$, $\gamma_{22} = -1$,
and $\cS[k]$-cycles are admissible for large $k$.
To state these parameter values succinctly, we note that with (\ref{eq:tauDeltaEig11B}) and $\delta_{\sR} = 1$ we have
\begin{equation}
\gamma_{21} = \frac{(\tau_{\sR} - 2)(\tau_{\sR} + 2)
(\tau_{\sR}^4 - \tau_{\sR}^3 - 3 \tau_{\sR}^2 + 2)
(\tau_{\sR}^4 + \tau_{\sR}^3 - 3 \tau_{\sR}^2 + 2)}
{\tau_{\sR}^2 (\tau_{\sR}^2 - 2)^2} \;.
\label{eq:gamma21B}
\end{equation}
The parameter values are
\begin{equation}
\tau_{\sL} \approx -1.1629 \;, \qquad
\delta_{\sL} = 1 \;, \qquad
\tau_{\sR} \approx 0.7952 \;, \qquad
\delta_{\sR} = 1 \;,
\label{eq:paramB}
\end{equation}
where $\tau_{\sR}$ is a root of the quartic polynomial, $\tau_R^4-\tau_R^3-3 \tau_R^2 + 2$,
that appears in (\ref{eq:gamma21B}), and $\tau_{\sL} = \frac{2 \tau_{\sR}}{\tau_{\sR}^2 - 2}$, by (\ref{eq:tauDeltaEig11B}).
With (\ref{eq:paramB}) and $\mu = 1$, the $\cS[k]$-cycles are stable (each with a repeated stability multiplier of $-1$)
and admissible for all $k \ge 4$.
The following proposition formalizes this statement.
We omit a proof of this proposition as it may be achieved
by repeating the steps of the proof for the previous example.
The $\cS[k]$-cycles are shown Fig.~\ref{fig:coexistNeutralB}
where we can observe that the $\cS[k]$-cycles grow along the eigenspaces of the matrices $M_{\cX^{(m)}}$.

%.....................................................................
\begin{proposition}
Let $\cX = \sR^3 \sL$, $\cY = \sL \sR^2 \sL^2$ and $\cS[k] = \cX^k \cY$.
Then the map (\ref{eq:f}) with (\ref{eq:paramB}) and $\mu = 1$
has a unique $\cS[k]$-cycle for all $k \ge 1$, that is admissible and stable for all $k \ge 4$, and is virtual otherwise.
\label{pr:verify4}
\end{proposition}

%%%%%%%%%%%%%%%%%%%%%%%%%%%%%%%%%%%%%%%%%%%%%%%%%%%%%%%%%%%%%
\begin{figure}[t!]
\begin{center}
\setlength{\unitlength}{1cm}
\begin{picture}(15,7.5)
\put(0,0){\includegraphics[height=7.5cm]{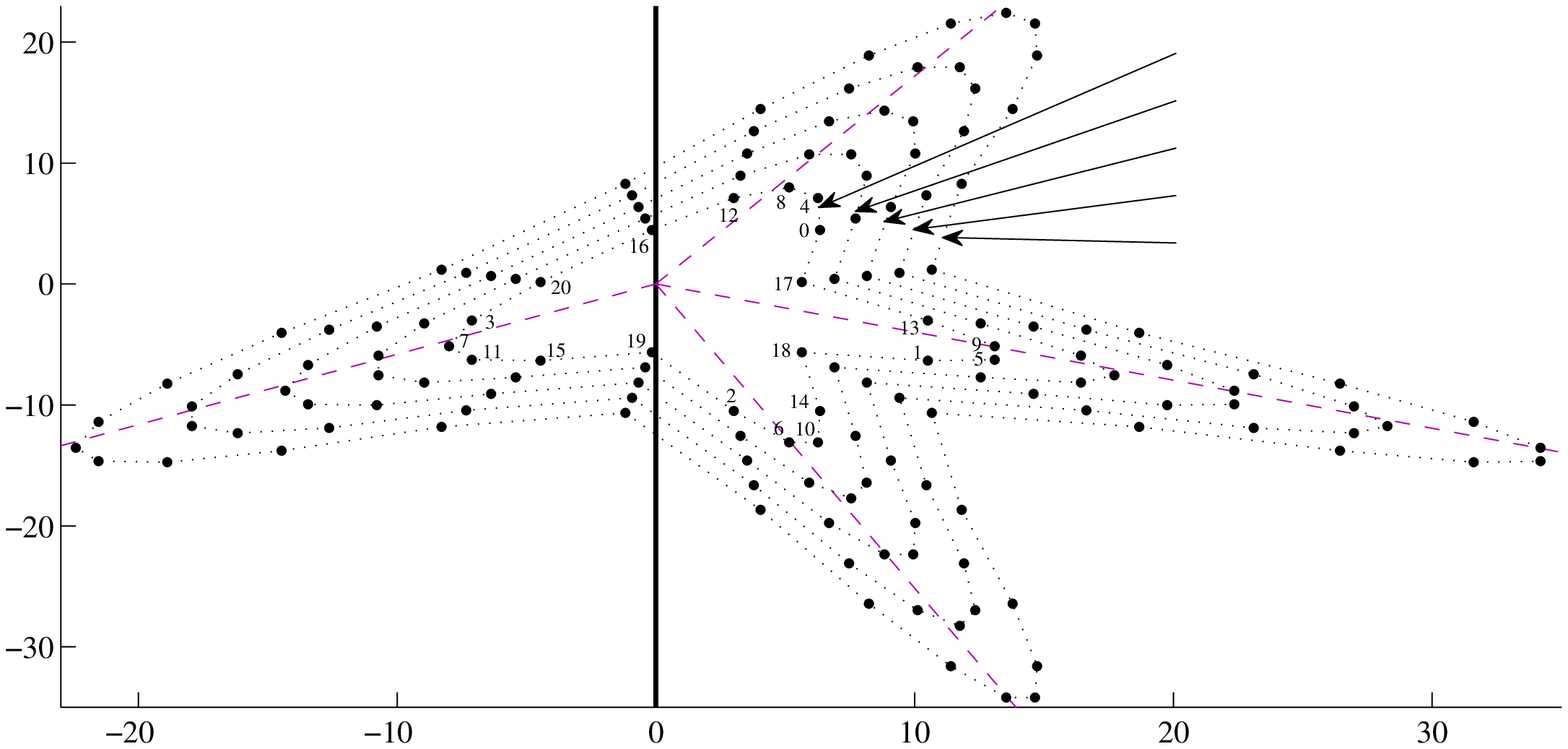}}
\put(7.6,.1){$x$}
\put(0,4.1){$y$}
\put(11.4,6.95){\footnotesize $k = 4$}
\put(11.4,6.5){\footnotesize $k = 5$}
\put(11.4,6.05){\footnotesize $k = 6$}
\put(11.4,5.6){\footnotesize $k = 7$}
\put(11.4,5.15){\footnotesize $k = 8$}
\end{picture}
\caption{
A phase portrait of (\ref{eq:f}) with (\ref{eq:paramB}) and $\mu = 1$.
For all $k \ge 4$, the $\cS[k]$-cycles are stable and admissible,
where $\cS[k] = (\sR^3 \sL)^k \sL^2 \sR^2 \sL$.
These are shown for $k = 4,\ldots,8$,
and for clarity each point is connected to its fourth iterate under (\ref{eq:f}) by a dotted line segment.
The $i^{\rm th}$ point of the $\cS[4]$-cycle is labeled by the index $i$.
The eigenspace of each $M_{\cX^{(m)}}$, for $m = 0,1,2,3$, where $\cX = \sR^3 \sL$,
is shown as a dashed line on the side of the switching manifold corresponding to the symbol $\cX_m$.
\label{fig:coexistNeutralB}
}
\end{center}
\end{figure}
%%%%%%%%%%%%%%%%%%%%%%%%%%%%%%%%%%%%%%%%%%%%%%%%%%%%%%%%%%%%%

%-----------------------------------------------------------------------------
\subsubsection*{An example with $n_{\cX} = 1$}

Lastly suppose\removableFootnote{
With $\sL$ and $\sR$ switched (so that $\cX_0 = \sR$ as usual),
I can only find infinite many admissible periodic orbits when $\mu < 0$.
}
\begin{equation}
\cX = \sL \;, \qquad
\cY = \sR^5 \;.
\label{eq:XYK}
\end{equation}
(Similar examples may be obtained by defining $\cY$ to be a different power of $\sR$.)
With (\ref{eq:XYK}), $M_{\cX} = A_{\sL}$,
therefore $M_{\cX}$ has a repeated unit eigenvalue when
\begin{equation}
\tau_{\sL} = 2 \;, \qquad
\delta_{\sL} = 1 \;.
\label{eq:tauDeltaEig11K}
\end{equation}
With (\ref{eq:tauDeltaEig11K}), $\nu = -1$ %$[1,-1]^{\sf T}$ is an eigenvector of $M_{\cX}$
and $\Gamma = Q^{-1} M_{\cY} Q$ yields
\begin{align}
\gamma_{21} &= -\left( \delta_{\sR}^2 - 3 \delta_{\sR} \tau_{\sR}^2 + \tau_{\sR}^4 \right)
\left( \delta_{\sR} - \tau_{\sR} + 1 \right) \;, \nonumber \\
\gamma_{22} &= 2 \delta_{\sR}^2 \tau_{\sR} + \delta_{\sR}^2 - \delta_{\sR} \tau_{\sR}^3 -
3 \delta_{\sR} \tau_{\sR}^2 + \tau_{\sR}^4 \;. \nonumber
\end{align}
Through numerical computations we find that we have $\gamma_{21} = 0$ and $\gamma_{22} = -1$
only when $\left( \tau_{\sR},\delta_{\sR} \right) = \left( \frac{1 \pm \sqrt{5}}{2},1 \right)$.
With $\tau_{\sR} = \frac{1-\sqrt{5}}{2}$, $\cS[k]$-cycles are virtual for all $\mu \ne 0$.
With the second value of $\tau_{\sR}$, all together we have
\begin{equation}
\tau_{\sL} = 2 \;, \qquad
\delta_{\sL} = 1 \;, \qquad
\tau_{\sR} = \frac{1+\sqrt{5}}{2} \;, \qquad
\delta_{\sR} = 1 \;.
\label{eq:paramK}
\end{equation}
With also $\mu = 1$, $\cS[k]$-cycles are stable (with a repeated stability multiplier of $-1$)
and admissible for all $k \ge 11$, see Proposition \ref{pr:verify1}.
As with the previous example we omit a proof of the result
as it may be achieved in the same fashion as for Proposition \ref{pr:verify3}.
As shown in Fig.~\ref{fig:coexistNeutralK},
the $\cS[k]$-cycles align asymptotically with $[1,-1]^{\sf T}$.
This is the eigenvector of 
$M_{\cX} = A_{\sL} = \left[ \begin{array}{cc} 2 & 1 \\ -1 & 0 \end{array} \right]$.

%.....................................................................
\begin{proposition}
Let $\cX = \sL$, $\cY = \sR^5$ and $\cS[k] = \cX^k \cY$.
Then the map (\ref{eq:f}) with (\ref{eq:paramK}) and $\mu = 1$
has a unique $\cS[k]$-cycle for all $k \ge 1$,
that is admissible and stable for all $k \ge 11$,
and is virtual otherwise.
\label{pr:verify1}
\end{proposition}

%%%%%%%%%%%%%%%%%%%%%%%%%%%%%%%%%%%%%%%%%%%%%%%%%%%%%%%%%%%%%
\begin{figure}[t!]
\begin{center}
\setlength{\unitlength}{1cm}
\begin{picture}(15,7.5)
\put(0,0){\includegraphics[height=7.5cm]{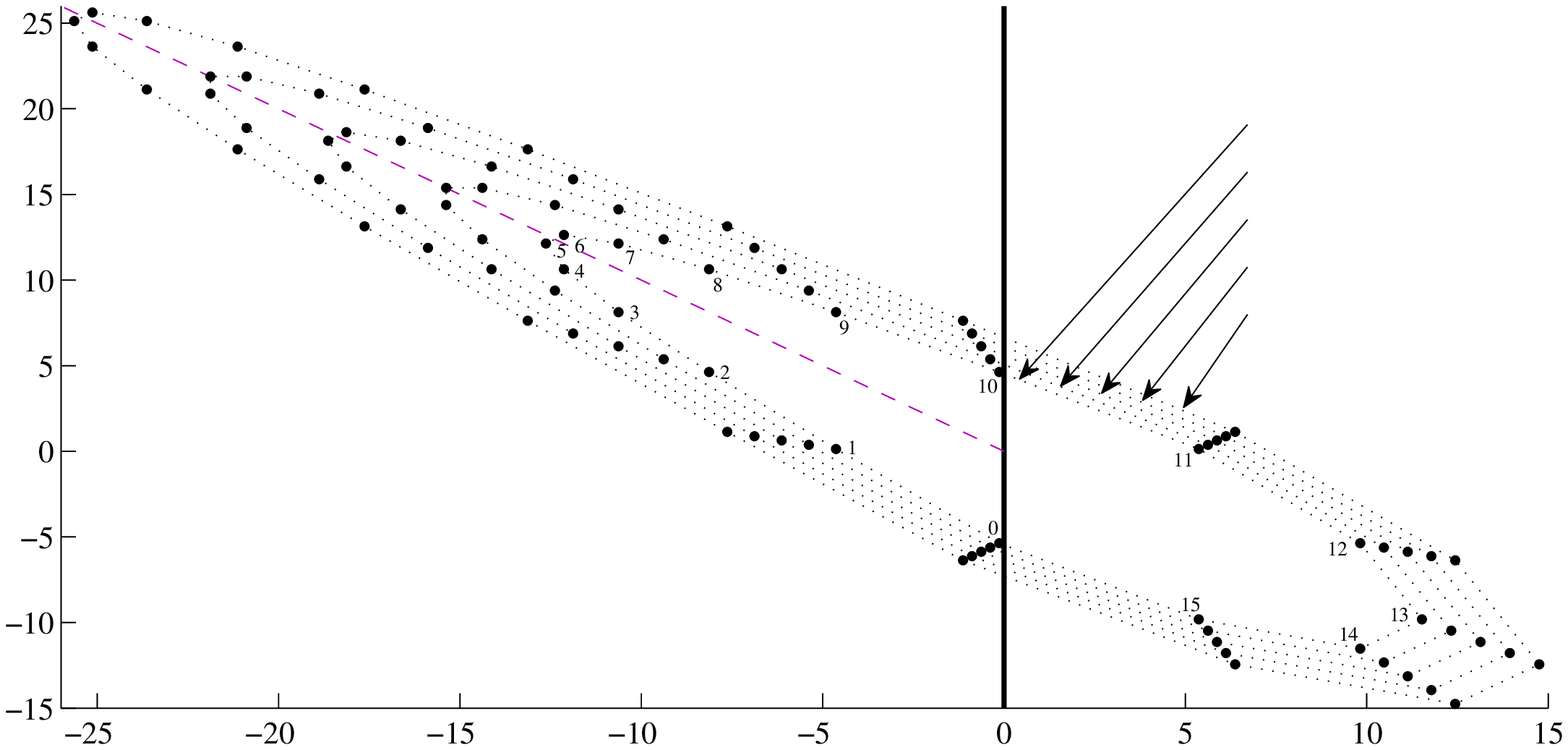}}
\put(7.6,.1){$x$}
\put(0,4.1){$y$}
\put(12.05,6.32){\footnotesize $k = 11$}
\put(12.05,5.87){\footnotesize $k = 12$}
\put(12.05,5.42){\footnotesize $k = 13$}
\put(12.05,4.97){\footnotesize $k = 14$}
\put(12.05,4.52){\footnotesize $k = 15$}
\end{picture}
\caption{
A phase portrait of (\ref{eq:f}) with (\ref{eq:paramK}) and $\mu = 1$.
With $\cS[k] = \sL^k \sR^5$, $\cS[k]$-cycles are stable and admissible for all $k \ge 11$, and are shown up to $k = 15$.
Each point is connected to its image under (\ref{eq:f}) by a dotted line segment.
The $i^{\rm th}$ point of the $\cS[11]$-cycle is labeled by the index $i$.
The dashed line is the eigenspace of $M_{\cX}$,
which for this example is simply the matrix $A_{\sL}$,
and is shown only on the left side of the switching manifold.
\label{fig:coexistNeutralK}
}
\end{center}
\end{figure}
%%%%%%%%%%%%%%%%%%%%%%%%%%%%%%%%%%%%%%%%%%%%%%%%%%%%%%%%%%%%%

%=====================================================================
\section{Perturbations from codimension-four points}
\label{sec:perturb4}

In \S\ref{sub:examples4} we identified three sets of parameter values
at which (\ref{eq:f}) has infinitely many stable $\cS[k]$-cycles in accordance with Theorem \ref{th:codim4}.
For each of these examples we have $\delta_{\sL} = \delta_{\sR} = 1$,
although this restriction is not an immediate consequence of Theorem \ref{th:codim4}\removableFootnote{
With silly choices for $\cY$
(where the words $\cX \cY$ and $\cY \cX$ differ in more than two symbols)
I can find examples
(with $\cX = \sL$)
for which $\gamma_{21} = 0$, $\gamma_{22} = -1$ and $\gamma_{11} \ne -1$,
but some points of $\cS[k]$-cycles are not admissible.
}.
With $\delta_{\sL} = \delta_{\sR} = 1$, (\ref{eq:f}) preserves area and orientation
and each $\cS[k]$-cycle has a repeated stability multiplier of $-1$. 

In this section we study perturbations of (\ref{eq:f})
from a general codimension-four point at which Theorem \ref{th:codim4} is satisfied
and at which $\delta_{\sL} = \delta_{\sR} = 1$ (with the exception that this assumption is not required in Lemma \ref{le:Xcycle}).
By assuming $\delta_{\sL} = \delta_{\sR} = 1$, we able to produce strong results.
It remains to determine if the assumptions of Theorem \ref{th:codim4}
can be satisfied at a point for which $\delta_{\sL}, \delta_{\sR} \ne 1$.

%----------------------------------------------------------------------
\subsection{Stability of periodic solutions for perturbed parameter values}
\label{sub:stability4}

Here we investigate the stability of $\cS[k]$-cycles.
The matrices $A_{\sL}(\ee)$ and $A_{\sR}(\ee)$ are smooth functions of $\ee$,
as is any product of these matrices.
Thus if $M_{\cX}(\ee)$ has a repeated unit eigenvalue when $\ee = 0$, we can write
\begin{align}
\det \left( M_{\cX}(\ee) \right) &= 1 + \alpha \ee + \cO \left( \ee^2 \right) \;, \label{eq:detMX} \\
{\rm trace} \left( M_{\cX}(\ee) \right) &= 2 + \beta \ee + \cO \left( \ee^2 \right) \;, \label{eq:traceMX}
\end{align}
for some constants $\alpha, \beta \in \mathbb{R}$.
These constants are determined by the direction in four-dimensional parameter space that
we head by varying $\ee$ from zero.
We then have the following theorem.

%.....................................................................
\begin{theorem}
Suppose $\mu \ne 0$ and that the remaining parameters of (\ref{eq:f}) vary smoothly with $\ee$.
Suppose that when $\ee = 0$ (\ref{eq:f}) satisfies the assumptions of Theorem \ref{th:codim4}
and $\delta_{\sL}(0) = \delta_{\sR}(0) = 1$.
Suppose $\alpha \ne 0$ and $\alpha \ne \beta$,
where $\alpha$ and $\beta$ are defined by (\ref{eq:detMX}) and (\ref{eq:traceMX}).
Then the following statements are equivalent:
\begin{enumerate}[label=\roman{*}),ref=\roman{*}]
\item
\label{it:alphaBeta}
$\beta < \alpha < 0$;
\item
\label{it:XAsyStable}
there exists $\ee^* > 0$, such that for all $0 < \ee \le \ee^*$,
the eigenvalues of $M_{\cX}(\ee)$ are complex-valued and have modulus less than $1$;
\item
\label{it:SkAsyStable}
there exists $k_{\rm min} \in \mathbb{Z}$ and $\Delta > 0$,
such that for all $k \ge k_{\rm min}$ and $0 < \ee \le \Delta k^{-2}$,
the eigenvalues of $M_{\cS[k]}(\ee)$ have modulus less than $1$\removableFootnote{
We have to be careful here as we cannot talk about the stability of a virtual periodic solution.
If the $\cS[k]$-cycle is admissible, then this implies it is attracting.
If the $\cS[k]$-cycle is virtual, then it
corresponds to an attracting periodic solution of $f^{\cS[k]}$.
}\removableFootnote{
I suspect that the $k$-dependent upper bound on $\ee$ is actually not necessary,
but I don't know how to prove it.
Certainly as $k \to \infty$,
$\det \left( M_{\cS[k]}(\ee) \right) + {\rm trace} \left( M_{\cS[k]}(\ee) \right) + 1 \to 1$.
}.
\end{enumerate}
Furthermore, with $\beta < \alpha < 0$ the eigenvalues of $M_{\cX}$ are
$\lambda(\ee) = r(\ee) {\rm e}^{\pm {\rm i} \theta(\ee)}$,
where
\begin{equation}
r(\ee) = 1 + \frac{\alpha}{2} \ee + \cO \left( \ee^2 \right) \;, \qquad
\theta(\ee) = \sqrt{\alpha - \beta} \,\ee^{\frac{1}{2}} + \cO(\ee) \;.
\label{eq:rtheta}
\end{equation}
\label{th:alphaBeta}
\end{theorem}

For a small perturbation of parameter values in any direction from a codimension-four point,
(\ref{eq:f}) has a large number of admissible $\cS[k]$-cycles.
Theorem \ref{th:alphaBeta} tells us that these $\cS[k]$-cycles will be attracting
exactly when we choose the direction of perturbation such that $\beta < \alpha < 0$.
%(because admissible $\cS[k]$-cycles are attracting if and only if the eigenvalues of 
%$M_{\cS[k]}(\ee)$ have modulus less than $1$).
Moreover, the same condition
indicates whether or not $M_{\cX}(\ee)$ has complex eigenvalues with modulus less than $1$.

\begin{proof}
The eigenvalues of $M_{\cX}(\ee)$ are
$\lambda(\ee) = \frac{1}{2} {\rm trace} \left( M_{\cX}(\ee) \right)
\pm {\rm i} \sqrt{\det \left( M_{\cX}(\ee) \right) - \frac{1}{4} {\rm trace} \left( M_{\cX}(\ee) \right)^2}$.
By substituting (\ref{eq:detMX}) and (\ref{eq:traceMX}) into this expression
and converting to polar form we obtain (\ref{eq:rtheta}).
Therefore, if $\beta < \alpha < 0$, then $r(\ee) < 1$ and $\theta(\ee) \in \mathbb{R}$
for all sufficiently small $\ee > 0$.
%and so the eigenvalues of $M_{\cX}(\ee)$ are complex and have modulus less than $1$.
Conversely, if the eigenvalues of $M_{\cX}(\ee)$ are complex and have modulus less than $1$ for arbitrarily small $\ee > 0$,
then by (\ref{eq:rtheta}), since $\alpha \ne 0$ and $\alpha \ne \beta$, we must have $\beta < \alpha < 0$.
This verifies the equivalence of (\ref{it:alphaBeta}) and (\ref{it:XAsyStable}).

%To complete the proof of Theorem \ref{th:alphaBeta},
%it remains to show that part (\ref{it:alphaBeta}) implies part (\ref{it:SkAsyStable}).
%This requires stronger statements than (\ref{eq:detMSk}) and (\ref{eq:traceMSk}),
%because no upper bound is given to the value of $k$ in part (\ref{it:SkAsyStable}).
%However, unlike for (\ref{eq:detMSk}) and (\ref{eq:traceMSk}),
%here we may assume $\beta < \alpha < 0$ which is extremely helpful.

To verify that (\ref{it:alphaBeta}) implies (\ref{it:SkAsyStable}),
we use (\ref{eq:rtheta}) to show that, if $\beta < \alpha < 0$,
then for all $k \in \mathbb{Z}$\removableFootnote{
This form is implicitly implies that the error terms are independent of $k$.
},
\begin{align}
\det \left( M_{\cS[k]}(\ee) \right) &= r^{2k}(\ee) \left( 1 + \cO(\ee) \right) \;, \label{eq:detMSk} \\
{\rm trace} \left( M_{\cS[k]}(\ee) \right) &= -r^k(\ee) \left( 2 \cos(k\theta(\ee))
+ \phi \sin(k\theta(\ee)) \ee^{\frac{1}{2}} + \cO(\ee) \right) \;, \label{eq:traceMSk}
\end{align}
where $\phi$ is a real-valued constant whose value is not important here (see Appendix \ref{sub:detTraceMSk}).
Equations (\ref{eq:detMSk}) and (\ref{eq:traceMSk}) are derived in Appendix \ref{sub:detTraceMSk}
by using an $\ee$-dependent similarity transform
that enables us to express arbitrary powers of $M_{\cS[k]}$ explicitly.
Here we use (\ref{eq:detMSk}) and (\ref{eq:traceMSk}) to show that the inequalities
(\ref{eq:stabConditionSN})-(\ref{eq:stabConditionNS}) are satisfied strictly.
Since $\alpha < 0$, $r(\ee)$ decreases asymptotically linearly with $\ee$
and hence $\det \left( M_{\cS[k]}(\ee) \right) - 1 < 0$, for small $\ee$ and large $k$, i.e.~(\ref{eq:stabConditionNS}) holds strictly.
Also, given small $\Delta > 0$, for $0 < \ee \le \Delta k^{-2}$
we have $k \theta(\ee) = \cO \big( \Delta^{\frac{1}{2}} \big)$.
Thus ${\rm trace} \left( M_{\cS[k]}(\ee) \right) \approx -2$ and (\ref{eq:stabConditionSN}) holds strictly.
Lastly, we have
\begin{align}
& \det \left( M_{\cS[k]}(\ee) \right) + {\rm trace} \left( M_{\cS[k]}(\ee) \right) + 1 \nonumber \\
&= \left( r^k(\ee) - 1 \right)^2 +
r^k(\ee) \left\{ 2 - 2 \cos(k\theta(\ee)) - \phi \sin(k\theta(\ee)) \ee^{\frac{1}{2}} + \cO(\ee) \right\} \nonumber \\
&= \left( r^k(\ee) - 1 \right)^2 +
r^k(\ee) \left\{ k^2 \left( \alpha - \beta + \cO \left( k^{-1} \right)
+ \cO \big( \ee^{\frac{1}{2}} \big) \right) \ee
+ \cO \left( \Delta^2 \right) \right\} \;,
\label{eq:stabConditionPD2}
\end{align}
where we have substituted
$\cos(k\theta(\ee)) = 1 - k^2 \left( \frac{\alpha-\beta}{2} \ee
+ \cO \big( \ee^{\frac{3}{2}} \big) \right) + \cO \left( \Delta^2 \right)$, and
$\sin(k\theta(\ee)) = k \left( \sqrt{\alpha-\beta} \,\ee^{\frac{1}{2}} + \cO(\ee) \right) + \cO \big( \Delta^{\frac{3}{2}} \big)$.
Since $\beta < \alpha$, there exists $\ee^* > 0$ and $k_{\rm min} > 0$
such that for all $0 < \ee < \ee^*$ and $k \ge k_{\rm min}$
the $\cO(\ee)$ term inside the braces in (\ref{eq:stabConditionPD2}) is greater than,
say, $\frac{(\alpha-\beta)}{2} k^2 \ee$, which is positive.
This term is also $\cO(\Delta)$ and thus dominates the $\cO \left( \Delta^2 \right)$ error term.
Hence there exists $\Delta > 0$ such that (\ref{eq:stabConditionPD2}) is strictly positive
(i.e.~(\ref{eq:stabConditionPD}) holds strictly),
$k \ge k_{\rm min}$ and $0 < \ee \le \Delta k^{-2}$.
Therefore (\ref{eq:stabConditionSN})-(\ref{eq:stabConditionNS}) are all satisfied strictly
and the eigenvalues of $M_{\cS[k]}(\ee)$ have modulus less than 1.

It remains to verify that (\ref{it:SkAsyStable}) implies (\ref{it:alphaBeta}).
For any fixed value of $k$,
by expanding (\ref{eq:detMSk}) and (\ref{eq:traceMSk}) in $\ee$ through the use of (\ref{eq:rtheta}), we obtain
\begin{align}
\det \left( M_{\cS[k]}(\ee) \right) &= 1 + k \left( \alpha + \cO \left( k^{-1} \right) \right) \ee
+ \cO \left( \ee^2 \right) \;, \label{eq:detMSk2} \\
{\rm trace} \left( M_{\cS[k]}(\ee) \right) &= -2 + k^2 \left( \alpha - \beta + \cO \left( k^{-1} \right) \right) \ee
+ \cO \left( \ee^2 \right) \;. \label{eq:traceMSk2}
\end{align}
Even though (\ref{eq:detMSk}) and (\ref{eq:traceMSk}) were derived assuming $\beta < \alpha < 0$,
(\ref{eq:detMSk2}) and (\ref{eq:traceMSk2}) are valid for all $\alpha,\beta \in \mathbb{R}$
because, for any fixed $k$, $\det \left( M_{\cS[k]}(\ee) \right)$ and ${\rm trace} \left( M_{\cS[k]}(\ee) \right)$
are smooth functions of $\ee$, and therefore have unique Taylor expansions in $\ee$.
(Formulas (\ref{eq:detMSk2}) and (\ref{eq:traceMSk2})
can also derived directly by computing the Taylor expansion of $Q^{-1} M_{\cX}^k(\ee) Q$ to first order in $\ee$
without applying the $\ee$-dependent similarity transform used in Appendix \ref{sub:detTraceMSk}\removableFootnote{
See Appendix \ref{sub:aux1}.
}.)
We then have
\begin{align}
\det \left( M_{\cS[k]}(\ee) \right) - {\rm trace} \left( M_{\cS[k]}(\ee) \right) + 1 &=
4 + \cO(\ee) \;, \nonumber \\
\det \left( M_{\cS[k]}(\ee) \right) + {\rm trace} \left( M_{\cS[k]}(\ee) \right) + 1 &=
k^2 \left( \alpha - \beta + \cO \left( k^{-1} \right) \right) \ee + \cO \left( \ee^2 \right) \;,
\nonumber \\
\det \left( M_{\cS[k]}(\ee) \right) - 1 &=
k \left( \alpha + \cO \left( k^{-1} \right) \right) \ee + \cO \left( \ee^2 \right) \;, \nonumber
\end{align}
%Therefore, if $\beta < \alpha < 0$, then for all sufficiently large $k$, for all sufficiently small $\ee > 0$
%(\ref{eq:stabConditionSN})-(\ref{eq:stabConditionNS}) are satisfied strictly.
and therefore the inequalities (\ref{eq:stabConditionSN})-(\ref{eq:stabConditionNS})
hold strictly for arbitrarily large $k$ and small $\ee > 0$ only if $\beta < \alpha < 0$ (since $\alpha \ne 0$ and $\alpha \ne \beta$).
\end{proof}

%----------------------------------------------------------------------
\subsection{Admissibility of periodic solutions for perturbed parameter values}
\label{sub:admissibility4}

Here we study admissibility of the $\cS[k]$-cycles.
Lemma \ref{le:Xcycle} describes $\cX$-cycles for small $\ee > 0$.
Lemma \ref{le:upperBound} tells us that if $\ee \ge \frac{4 \pi^2}{\alpha-\beta} k^{-2}$,
then under reasonable assumptions the $\cS[k]$-cycle must be virtual.
Conversely Lemma \ref{le:lowerBound} tells us that there exists $\Delta > 0$,
such that if $\ee \le \Delta k^{-2}$, then under similar assumptions the $\cS[k]$-cycle is admissible.

At a codimension-four point associated with Theorem \ref{th:codim4},
the $\cX$-cycle does not exist, or is possibly non-unique, because $M_{\cX}$ has a unit eigenvalue.
The following result tells us that if $\alpha > \beta$
(as is necessary for the asymptotic stability of $\cS[k]$-cycles for $\ee > 0$),
then for small $\ee > 0$ the $\cX$-cycle is unique but {\em completely virtual},
meaning that every point of the $\cX$-cycle lies on the wrong side of the switching manifold for admissibility.
%Alternatively if $\alpha < \beta$ the $\cX$-cycle is admissible.

%.....................................................................
\begin{lemma}
Suppose $\mu \ne 0$ and that the remaining parameters of (\ref{eq:f}) vary smoothly with $\ee$.
Suppose that when $\ee = 0$ (\ref{eq:f}) satisfies the assumptions of Theorem \ref{th:codim4}.
Then if $\alpha > \beta$ [resp.~$\alpha < \beta$],
where $\alpha$ and $\beta$ are defined by (\ref{eq:detMX}) and (\ref{eq:traceMX}),
there exists $\ee^* > 0$ such that for all $0 < \ee \le \ee^*$,
the $\cX$-cycle is unique and completely virtual [resp.~unique and admissible].
\label{le:Xcycle}
\end{lemma}

\begin{proof}
We let $w^{\cX}_i(\ee) = \left( u^{\cX}_i(\ee),v^{\cX}_i(\ee) \right)$ denote the points of the $\cX$-cycle,
and let $w^{\cS[k]}_i(\ee) = \left( u^{\cS[k]}_i(\ee),v^{\cS[k]}_i(\ee) \right)$ denote the points of an $\cS[k]$-cycle,
in $(u,v)$-coordinates for $\ee > 0$.
By (\ref{eq:uSkjm}) we have
\begin{equation}
u^{\cS[k]}_{j n_{\cX} + m}(\ee) = -\frac{\psi_{11m} \omega_{12} \rho_2}{2} j(k-j) + \cO(j,k) + \cO(\ee) \;.
\label{eq:uSkjm2}
\end{equation}
In Appendix \ref{sub:Xcycle} we derive the formula
\begin{equation}
u^{\cX}_m(\ee) = \frac{1}{\ee}
\left( \frac{\psi_{11m} \omega_{12} \rho_2}{\alpha - \beta} + \cO(\ee) \right) \;.
\label{eq:uXm}
\end{equation}
Therefore if $\alpha > \beta$ the signs of $u^{\cS[k]}_{j n_{\cX} + m}$ and $u^{\cX}_m(\ee)$ are different\removableFootnote{
for (i) sufficiently small $\ee > 0$,
(ii) sufficiently large $k \in \mathbb{Z}$,
(iii) any $j = 0,\ldots,k-1$,
(iv) any $m = 0,\ldots,n_{\cX}-1$.
},
hence the point $w^{\cX}_m(\ee)$ is virtual.
Similarly if $\alpha < \beta$, each $w^{\cX}_m(\ee)$ is admissible.
\end{proof}

Given $k \in \mathbb{Z}$,
we now derive an upper bound on the largest value of $\ee$ for which the $\cS[k]$-cycle is admissible.
To do this we note from Theorem \ref{th:alphaBeta} that if $\beta < \alpha < 0$, as required for asymptotic stability,
then the eigenvalues of $M_{\cX}$ are complex-valued and so the map $f^{\cX}$ can be thought of as representing a rotation about its fixed point.
We combine this observation with the fact that the fixed point of $f^{\cX}$ is virtual,
by Lemma \ref{le:Xcycle}, to obtain an upper bound on the number of times
an orbit of (\ref{eq:f}) may consecutively follow the sequence $\cX$,
and hence an upper bound on value of $\ee$.

%.....................................................................
\begin{lemma}
Suppose $\mu \ne 0$ and that the remaining parameters of (\ref{eq:f}) vary smoothly with $\ee$.
Suppose that when $\ee = 0$ (\ref{eq:f}) satisfies the assumptions of Theorem \ref{th:codim4}
and $\delta_{\sL}(0) = \delta_{\sR}(0) = 1$.
Suppose $\beta < \alpha < 0$.
Then there exists $\ee^* > 0$ and $k_{\rm min} \in \mathbb{Z}$,
such that for all $k \ge k_{\rm min}$ and $\frac{4 \pi^2}{\alpha-\beta} k^{-2} \le \ee \le \ee^*$,
the $\cS[k]$-cycle is virtual.
\label{le:upperBound}
\end{lemma}

\begin{proof}
For ease of explanation, suppose $\cX_0 = \sR$ (without loss of generality).
Then if the $\cS[k]$-cycle is admissible with no points on the switching manifold,
each point
$\left( x^{\cS[k]}_{j n_{\cX}} ,\; y^{\cS[k]}_{j n_{\cX}} \right)$,
for $j=0,\ldots,k-1$, lies in the right half-plane.
Also $\left( x^{\cS[k]}_{(j+1) n_{\cX}} ,\; y^{\cS[k]}_{(j+1) n_{\cX}} \right)
= f^{\cX} \left( x^{\cS[k]}_{j n_{\cX}} ,\; y^{\cS[k]}_{j n_{\cX}} \right)$, for $j=0,\ldots,k-1$.

By Lemma \ref{le:Xcycle}, the fixed point of $f^{\cX}$, $\left( x^{\cX}_0 ,\; y^{\cX}_0 \right)$,
lies in the left half-plane.
By Theorem \ref{th:alphaBeta}, the eigenvalues of $M_{\cX}(\ee)$
are $\lambda(\ee) = r(\ee) {\rm e}^{\pm {\rm i} \theta(\ee)}$
where $r(\ee)$ and $\theta(\ee)$ are given by (\ref{eq:rtheta}).
Therefore the image of a point under the map $f^{\cX}$ may be thought of as
a rotation about $\left( x^{\cX}_0 ,\; y^{\cX}_0 \right)$.
Regardless of the precise geometry of this rotation,
if we iterate any point in the right half-plane
a total of $\lceil \frac{\pi}{\theta(\ee)} \rceil$ times under $f^{\cX}$,
the resulting sequence of points will rotate at least $180^\circ$ about
$\left( x^{\cX}_0 ,\; y^{\cX}_0 \right)$,
and therefore must cross $x=0$.
Hence for all $k \ge \lceil \frac{\pi}{\theta(\ee)} \rceil + 1$, the $\cS[k]$-cycle cannot be admissible.
By crudely doubling the leading order approximation to $\frac{\pi}{\theta(\ee)}$,
as determined from (\ref{eq:rtheta}), we obtain the given upper bound on $k$\removableFootnote{
For sufficiently small $\ee > 0$, $\frac{2 \pi}{\theta(\ee)} > \lceil \frac{\pi}{\theta(\ee)} \rceil + 1$.
}.
\end{proof}

The following lemma gives a lower bound on the largest value of $\ee$
for which the $\cS[k]$-cycle is admissible\removableFootnote{
As a Corollary to this Lemma, when $\ee = 0$, $\cS[k]$-cycles are admissible for all $k \ge k_{\rm min}$.
Or do I know this already?
It is interesting that this is not necessary for the proof.
}.

%.....................................................................
\begin{lemma}
Suppose $\mu \ne 0$ and that the remaining parameters of (\ref{eq:f}) vary smoothly with $\ee$.
Suppose that when $\ee = 0$ (\ref{eq:f}) satisfies the assumptions of Theorem \ref{th:codim4},
$\delta_{\sL}(0) = \delta_{\sR}(0) = 1$, and
\begin{equation}
\min_i \left| x^{\cS[k]}_i(0) \right| \to \infty
{\rm ~as~} k \to \infty \;.
\label{eq:distanceToSwMan}
\end{equation}
Suppose $\beta < \alpha < 0$.
Then there exists $k_{\rm min} \in \mathbb{Z}$ and $\Delta > 0$, such that for all
$k \ge k_{\rm min}$ and $0 < \ee \le \Delta k^{-2}$,
the $\cS[k]$-cycle is admissible and attracting.
\label{le:lowerBound}
\end{lemma}

As $k \to \infty$, $\cS[k]$-cycles grow in size without bound.
Equation (\ref{eq:distanceToSwMan}) states that the distance of
the $\cS[k]$-cycles from the switching manifold
also grows in size without bound, which is the case for each of the examples in \S\ref{sub:examples4} as evident from the figures.

\begin{proof}
Let $w^{\cS[k]}_i(\ee) = \left( u^{\cS[k]}_i(\ee),v^{\cS[k]}_i(\ee) \right)$
denote the points of an $\cS[k]$-cycle in $(u,v)$-coordinates.
Here we show that there exists $k_{\rm min} \in \mathbb{Z}$ and $\Delta > 0$
such that, for each $i = 0,\ldots,k n_{\cX} + n_{\cY} - 1$,
the sign of $u^{\cS[k]}_i(\ee)$ is constant for all
$k \ge k_{\rm min}$ and $0 < \ee \le \Delta k^{-2}$.
This verifies admissibility because in $(u,v)$-coordinates the switching manifold is $u=0$,
and $\cS[k]$-cycles are assumed to be admissible when $\ee = 0$.
The asymptotic stability of the $\cS[k]$-cycles then follows immediately from Theorem \ref{th:alphaBeta}.

To calculate $w^{\cS[k]}_i(\ee)$ for small $\ee > 0$ and large $k \in \mathbb{Z}$,
we compute powers of the map $f^{\cX}$.
For this reason it is convenient to perform an $\ee$-dependent coordinate change such that the matrix $M_{\cX}$
is transformed to
$r(\ee) \left[ \begin{array}{cc}
\cos(\theta(\ee)) & \sin(\theta(\ee)) \\
-\sin(\theta(\ee)) & \cos(\theta(\ee))
\end{array} \right]$.
This is achieved in Appendix \ref{sub:Omegahat}, where we obtain the formulas
\begin{equation}
\left[ \begin{array}{c} u^{\cS[k]}_0(\ee) \\ v^{\cS[k]}_0(\ee) \end{array} \right]
= \left[ \begin{array}{c}
k \left( -\frac{\rho_1}{2} + \frac{ \left( \gamma_{12} + \omega_{12} \right) \rho_2}{4}
- \frac{\omega_{12} \sigma_2}{4} + \cO(\Delta)
+ \cO \left( k^{-1} \right) \right) \\
k \left( -\frac{\rho_2}{2}
+ \cO(\Delta) + \cO \left( k^{-1} \right) \right)
\end{array} \right] \;,
\label{eq:uv0}
\end{equation}
and
\begin{equation}
\left[ \begin{array}{c}
u^{\cS[k]}_{(k-1) n_{\cX}}(\ee) \\
v^{\cS[k]}_{(k-1) n_{\cX}}(\ee)
\end{array} \right] =
\left[ \begin{array}{c}
k \left( \frac{\rho_1}{2} + \frac{\left( \gamma_{12} - 3 \omega_{12} \right) \rho_2}{4}
+ \frac{\omega_{12} \sigma_2}{4}
+ \cO(\Delta) + \cO \left( k^{-1} \right) \right) \\
k \left( \frac{\rho_2}{2}
+ \cO(\Delta) + \cO \left( k^{-1} \right) \right)
\end{array} \right] \;.
\label{eq:uvk1nX}
\end{equation}
The key facet of the derivation of (\ref{eq:uv0}) and (\ref{eq:uvk1nX})
is that the error terms involve powers of $k \theta(\ee)$.
Therefore with $\ee \propto \Delta k^{-2}$,
since $\theta(\ee) = \cO \big( \ee^{\frac{1}{2}} \big)$,
the error terms can be expressed as powers of $\Delta^{\frac{1}{2}}$
(and the leading order error term is $\cO(\Delta)$).

We can obtain useful expressions for some additional points $w^{\cS[k]}_i(\ee)$
by iterating (\ref{eq:uv0}) and (\ref{eq:uvk1nX}) under $g^{\sL}$ and $g^{\sR}$,
in the appropriate order.
This is because iterations produce formulas of the same general form
as long as the number of iterations is small, specifically independent of $k$.
This approach was used above in the proof of Lemma \ref{le:lowerBoundZ}.

By iterating (\ref{eq:uv0}) $n_{\cX} - 1$ times, we obtain
\begin{equation}
u^{\cS[k]}_i(\ee) = k \left( \eta_i + \cO(\Delta)
+ \cO \left( k^{-1} \right) \right) \;,
\label{eq:uSki}
\end{equation}
for $i = 0,\ldots,n_{\cX}-1$, and some constants $\eta_i$.
Similarly by iterating (\ref{eq:uvk1nX}) $n_{\cX} + n_{\cY} - 1$ times,
equation (\ref{eq:uSki}) is valid for $i = (k-1) n_{\cX},\ldots,k n_{\cX} + n_{\cY} - 1$,
and some constants $\eta_i$.
The assumption (\ref{eq:distanceToSwMan}) implies that $\eta_i \ne 0$
for each of the $2 n_{\cX} + n_{\cY}$ given values of $i$,
because $u^{\cS[k]}_i(0)$ is an affine function of $k$
%(as evident from the proof of Theorem \ref{th:codim4})
and the coefficient of its linear part is $\eta_i$.
By (\ref{eq:uSki}), there exists $k_{\rm min} \in \mathbb{Z}$ and $\Delta > 0$
such that, for each $i = 0,\ldots,n_{\cX}-1$ and $i = (k-1) n_{\cX},\ldots,k n_{\cX} + n_{\cY} - 1$,
the sign of $u^{\cS[k]}_i(\ee)$ is constant for all $k \ge k_{\rm min}$ and $0 < \ee \le \Delta k^{-2}$.
To complete the proof, in Appendix \ref{sub:lowerBound} we show that this is also true
for all $i = 0,\ldots,k n_{\cX} + n_{\cY}-1$.
\end{proof}

%----------------------------------------------------------------------
\subsection{A scaling law for the number of attracting periodic solutions}
\label{sub:scaling4}

As in \S\ref{sub:scaling3},
we let $\kappa(\ee)$ denote the number of $\cS[k]$-cycles that are admissible and attracting,
and let $\ee_K$ denote the supremum value of $\ee > 0$
for which there are $K$ admissible, attracting $\cS[k]$-cycles.
Lemmas \ref{le:upperBound} and \ref{le:lowerBound} lead to the following result
which may be proved in the same fashion as for Theorem \ref{th:scaling3}.

%.....................................................................
\begin{theorem}
Suppose $\mu \ne 0$ and that the remaining parameters of (\ref{eq:f}) vary smoothly with $\ee$.
Suppose that when $\ee = 0$ (\ref{eq:f}) satisfies the assumptions of Theorem \ref{th:codim4},
$\delta_{\sL}(0) = \delta_{\sR}(0) = 1$, $\beta < \alpha < 0$, and (\ref{eq:distanceToSwMan}) holds.
Then as $K \to \infty$,
\begin{equation}
\ee_K \sim \varphi(K) K^{-2} \;,
\label{eq:scalingLaw}
\end{equation}
for a function $\varphi(K)$ that is bounded between positive constants\removableFootnote{
I think it is too hard to prove that $\varphi(K)$ is a constant
because I would have to carefully analyze the asymptotic behavior of $u^{\cS[k]}_i(\ee)$
without assuming that $k^2 \ee$ is small
which is given by a complicated nonlinear function.
}.
\label{th:scaling4}
\end{theorem}

Here we illustrate (\ref{eq:scalingLaw})
by perturbing the three codimension-four points identified in \S\ref{sub:examples4}.
As in \S\ref{sub:scaling3}
we consider perturbations of the form (\ref{eq:paramEe}),
where by Theorem \ref{th:alphaBeta}
we are constrained to choose $a$, $b$, $c$ and $d$ such that $\beta < \alpha < 0$.

For the first example the values of (\ref{eq:paramEe}) at $\ee = 0$ are given by (\ref{eq:paramA}),
and we have
\begin{align}
\alpha &= b + 2 d \;, \nonumber \\
\beta &= -2 \big(-1+\sqrt{2}\big) a + \big(-1+\sqrt{2}\big) b
- 2 \big(-1+\sqrt{2}\big) c + \big(-1+2\sqrt{2}\big) d \;. \nonumber
\end{align}
We therefore have $\beta < \alpha < 0$ when, for instance,
\begin{equation}
a = 1 \;, \qquad
b = -1 \;, \qquad
c = 0 \;, \qquad
d = 0 \;.
\label{eq:abcdA}
\end{equation}
Different values of $a$, $b$, $c$ and $d$ for which $\beta < \alpha < 0$ yield similar results.
In Fig.~\ref{fig:manyEeAll}-A we indicate the range of values of $\ee$ for which
the $\cS[k]$-cycles are admissible and attracting for all values of $k$ up to 100.
Since $\cS[k]$-cycles are admissible and stable only for $k \ge 8$ when $\ee = 0$, see Proposition \ref{pr:verify3},
these line segments emanate from $\ee = 0$ for $k \ge 8$.
For large $K$, $\ee_K$ is the right-hand end-point of the line segment for $k = K+7$,
and is a border-collision bifurcation at which the $(3k+2)^{\rm th}$ point of the $\cS[k]$-cycle
collides with the switching manifold.
With $\ee = 0.009$, there are five admissible, attracting $\cS[k]$-cycles;
these are shown to the right of panel A.
The $(3k+2)^{\rm th}$ points of these periodic solutions lie just to the left of the switching manifold
(around $y \approx -4$).

As discussed in \S\ref{sub:scaling3},
it is reasonable to assume that $\varphi(K)$ is constant,
which appears to be the case for this example.
Then the scaling law (\ref{eq:scalingLaw}) predicts that
$\frac{\ee_{2K}}{\ee_K}$ approaches the value $\frac{1}{4}$.
As shown in the inset of panel A,
this is consistent with the data in the figure.
(In order to obtain the points shown in the insets of Fig.~\ref{fig:manyEeAll},
$\cS[k]$-cycles were computed up to $k = 200$.)

%%%%%%%%%%%%%%%%%%%%%%%%%%%%%%%%%%%%%%%%%%%%%%%%%%%%%%%%%%%%%
\begin{figure}[t!]
\begin{center}
\setlength{\unitlength}{1cm}
\begin{picture}(14.9,11.3)
\put(.5,5.9){\includegraphics[height=5.4cm]{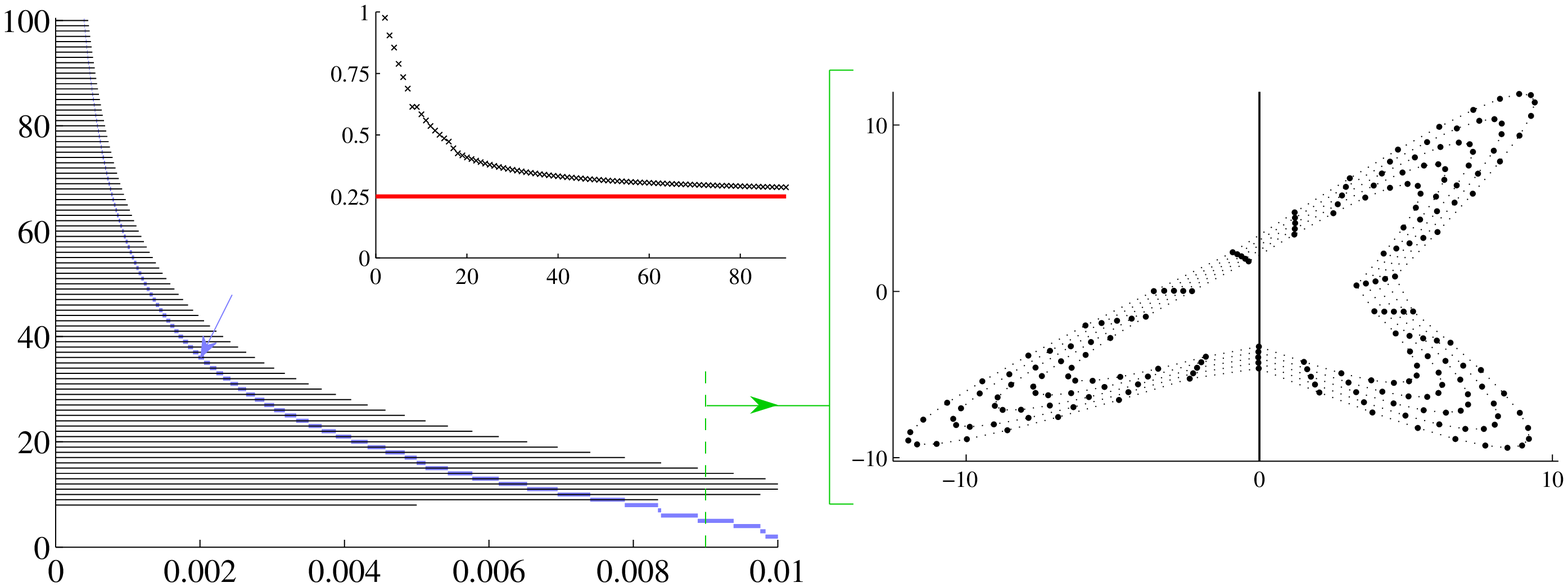}} % width = 13.9
\put(0,0){\includegraphics[height=5.4cm]{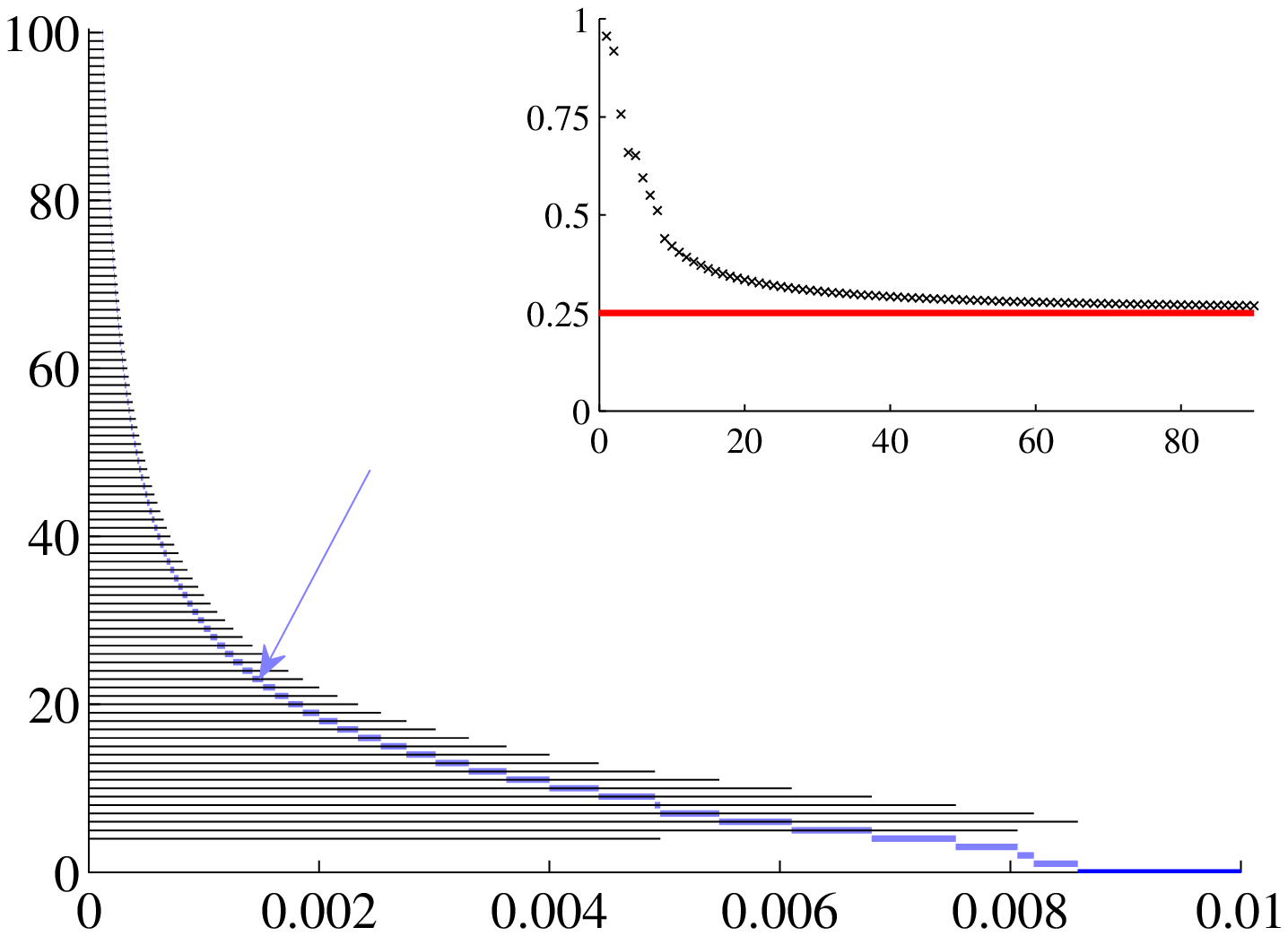}}	
\put(7.7,0){\includegraphics[height=5.4cm]{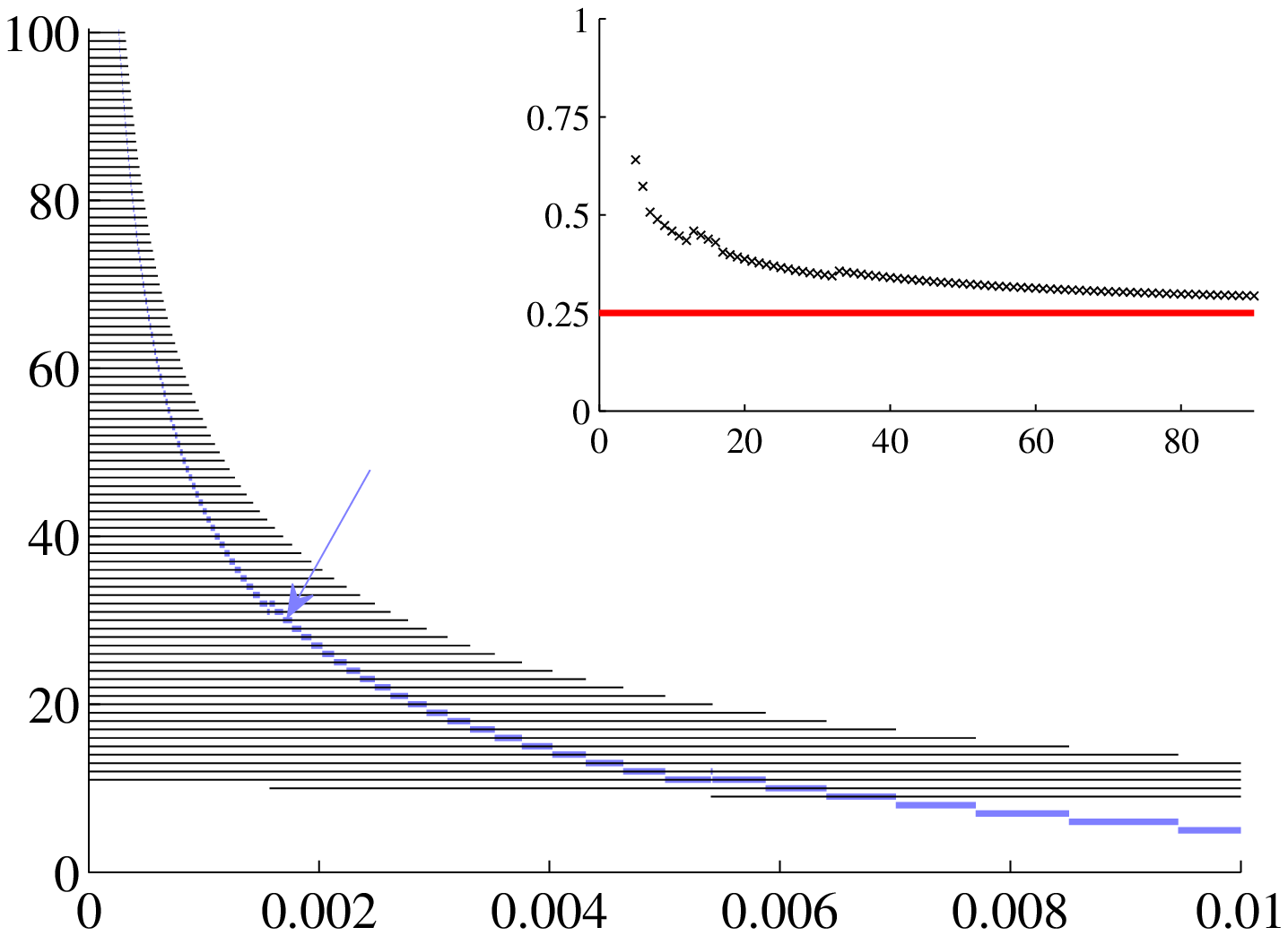}}
\put(11.47,6.8){\scriptsize $x$}
\put(8,8.7){\scriptsize $y$}
\put(10.9,10.6){\scriptsize $\ee = 0.009$}
\put(4.2,5.9){\small $\ee$}
\put(.5,9.55){\small $k$}
\put(3.7,0){\small $\ee$}
\put(0,3.65){\small $k$}
\put(11.4,0){\small $\ee$}
\put(7.7,3.65){\small $k$}
\put(2.45,8.8){\footnotesize \color{blue} $\kappa(\ee)$}
\put(1.95,2.9){\footnotesize \color{blue} $\kappa(\ee)$}
\put(9.65,2.9){\footnotesize \color{blue} $\kappa(\ee)$}
\put(5.55,8.5){\scriptsize $K$}
\put(2.95,10.2){\scriptsize $\frac{\ee_{2K}}{\ee_K}$}
\put(5.05,2.6){\scriptsize $K$}
\put(2.45,4.3){\scriptsize $\frac{\ee_{2K}}{\ee_K}$}
\put(12.75,2.6){\scriptsize $K$}
\put(10.15,4.3){\scriptsize $\frac{\ee_{2K}}{\ee_K}$}
\put(1.8,11.1){\large \sf \bfseries A}
\put(1.3,5.2){\large \sf \bfseries B}
\put(9,5.2){\large \sf \bfseries C}
\end{picture}
\caption{
Intervals of values of $\ee$ for which $\cS[k] = \cX^k \cY$-cycles are admissible, attracting periodic solutions
of (\ref{eq:f}) with $\mu = 1$, and remaining parameter values given by (\ref{eq:paramEe}).
Also plotted is $\kappa(\ee)$ -- the number of admissible, attracting $\cS[k]$-cycles.
The scaling law (\ref{eq:scalingLaw}) predicts that the points in the insets of the panels to converge to $\frac{1}{4}$.
To the right of panel A is a phase portrait corresponding to $\ee = 0.009$.
Here five $\cS[k]$-cycles are admissible and attracting ($k=10,\ldots,14$).
Panels A, B and C correspond to perturbations from examples with, respectively,
$n_{\cX} = 3$, $n_{\cX} = 4$ and $n_{\cX} = 1$, from \S\ref{sub:examples4}.
$\cX$ and $\cY$ are given by (\ref{eq:XYA}) in panel A,
(\ref{eq:XYB}) in panel B, and (\ref{eq:XYK}) in panel C.
The values of the parameters at $\ee = 0$ are given by (\ref{eq:paramA}) in panel A,
(\ref{eq:paramB}) in panel B, and (\ref{eq:paramK}) in panel C.
The values of $a$, $b$, $c$ and $d$ are given by (\ref{eq:abcdA}) in panel A,
(\ref{eq:abcdB}) in panel B, and (\ref{eq:abcdK}) in panel C.
\label{fig:manyEeAll}
}
\end{center}
\end{figure}
%%%%%%%%%%%%%%%%%%%%%%%%%%%%%%%%%%%%%%%%%%%%%%%%%%%%%%%%%%%%%

With the values of (\ref{eq:paramEe}) at $\ee = 0$ given by (\ref{eq:paramB})\removableFootnote{
\begin{align}
\alpha &= b + 3 d \;, \\
\beta &= 
\tau_{\sR}(0) \left( \tau_{\sR}(0)^2 - 2 \right) a
- \left( \tau_{\sR}(0)^2 - 1 \right) b
+ \frac{2 \tau_{\sR}(0) \left( \tau_{\sR}(0)^2 + 2 \right)}{\tau_{\sR}(0)^2 - 2} c
- \frac{\tau_{\sR}(0)^4 - \tau_{\sR}(0)^2 + 6}{\tau_{\sR}(0)^2 - 2} d \;.
\end{align}
}
we have $\beta < \alpha < 0$ when, for instance,
\begin{equation}
a = 0 \;, \qquad
b = -1 \;, \qquad
c = 0 \;, \qquad
d = -1 \;.
\label{eq:abcdB}
\end{equation}
As shown in Fig.~\ref{fig:manyEeAll}-B, the values of $\ee_K$ are consistent with (\ref{eq:scalingLaw}).

Lastly with (\ref{eq:paramK}) and
\begin{equation}
a = -2 \;, \qquad
b = -1 \;, \qquad
c = -4 \;, \qquad
d = 0 \;,
\label{eq:abcdK}
\end{equation}
we have $\beta < \alpha < 0$ (here simply $\alpha = b$ and $\beta = a$).
Again, this example is consistent with (\ref{eq:scalingLaw}), see Fig.~\ref{fig:manyEeAll}-C.

%=====================================================================
\section{Discussion}
\label{sec:conc}
\setcounter{equation}{0}

\subsubsection*{Summary and conclusions}

This paper investigates large numbers of periodic solutions in the 
two-dimensional border-collision normal form (\ref{eq:f}) when $\mu \ne 0$.
It was shown that, for appropriate choices of $\cX$ and $\cY$,
infinitely many stable $\cS[k]$-cycles (where $\cS[k] = \cX^k \cY$)
may coexist due to a repeated unit eigenvalue for the matrix $M_{\cX}$.
Necessary conditions for this phenomenon are given by Theorem \ref{th:codim4},
from which we see that the phenomenon is codimension-four.
Unlike at codimension-three points of
infinite coexistence due to a coincident homoclinic connection of an underlying saddle-type $\cX$-cycle (see Theorem \ref{th:codim3}),
at the codimension-four points the $\cS[k]$-cycles grow in size without bound.
In the context of border-collision bifurcations of piecewise-smooth maps,
the border-collision normal form is an approximation obtained by omitting nonlinear terms in the two half maps.
Consequently the dynamics of the border-collision normal form far from the origin
may diverge substantially from the intended application.
For this reason it is expected that nonlinear terms will have a greater effect
on $\cS[k]$-cycles for the codimension-four points than for the codimension-three points.

Perturbations from both types of high-codimension points were studied.
We let $\ee_K$ denote the supremum distance in a particular direction of parameter space from a
codimension-three or four point for which the number of admissible, attracting $\cS[k]$-cycles is $K$.
It was shown that near a codimension-three point, $\ee \sim \varphi(K) \lambda_2(0)^{-K}$,
whereas near a codimension-four point, $\ee \sim \varphi(K) K^{-2}$,
(where $\lambda_2(\ee) > 1$ is the unstable stability multiplier of the $\cX$-cycle
and in each case the function $\varphi(K)$ is bounded between positive constants).
Each $\ee_K$ corresponds to a bifurcation of an $\cS[k]$-cycle,
and the most natural scenario is for each $\ee_K$
to correspond to the same type of bifurcation, for large $K$.
This is the case for each example described above.
Each $\ee_K$ corresponds to the border-collision of a particular point of an $\cS[k]$-cycle.
The bifurcation values may then be matched to the root of an equation involving $\ee$ and $k$
from which, asymptotically, either $\ee_K \sim \varphi \lambda_2(0)^{-K}$
or $\ee_K \sim \varphi K^{-2}$, for some constant $\varphi$.
This suggests that for both scaling laws, $\varphi(K)$ is constant in general.
However, it appears that a demonstration of this claim requires
substantial additional analysis\removableFootnote{
For the codimension-four case the equations are nonlinear 
with coefficients given by complicated expressions
so it is difficult to extract values $\varphi$ or even
show that a root exists.
}
and a consideration of all points of $\cS[k]$-cycles,
as in the proof of Lemmas \ref{le:lowerBoundZ} and \ref{le:lowerBound},
and remains for future work.

In the four-dimensional parameter space of (\ref{eq:f}) (with fixed $\mu \ne 0$),
there are curves of codimension-three points associated with Theorem \ref{th:codim3}.
On such a curve the eigenvalues of $M_{\cX}$ are distinct, positive and multiply to $1$,
yet the end-points of the curves cannot correspond to a repeated unit eigenvalue and
a codimension-four point associated with Theorem \ref{th:codim4}.
To see why this is the case, we first note that at any codimension-three point of Theorem \ref{th:codim3}
we have ${\rm trace} \left( M_{\cX} \right) > 2$.
Thus if we attempt to find such a codimension-three point
by applying a small perturbation of a codimension-four point of Theorem \ref{th:codim4},
we must have $\beta > 0$, see (\ref{eq:traceMX}).
However, by Theorem \ref{th:alphaBeta} the $\cS[k]$-cycles are not stable in this case,
and therefore we will not find a codimension-three point associated with Theorem \ref{th:codim3}.

At first glance, the codimension-four points 
may appear to be less important than the codimension-three points
simply because they involve an additional codimension.
However, large numbers of attracting $\cS[k]$-cycles exist further from the codimension-four points
because $K^{-2}$ decays much slower than $\lambda_2(0)^{-K}$.
Indeed the notion that the border-collision normal form could exhibit arbitrarily many coexisting attractors
was in part motivated by the example (\ref{eq:paramSi10}),
at which there are six attracting $\cS[k]$-cycles.
We can now see that this point of parameter space is a relatively large distance away from the
corresponding codimension-four point, (\ref{eq:paramA}).
Given any $K$, and any codimension-four point associated with Theorem \ref{th:codim4},
there exists an open region of parameter space
at which (\ref{eq:f}) has $K$ attracting $\cS[k]$-cycles.
The codimension-four point lies on the boundary of this region.
Similarly, about any curve of codimension-three points associated with Theorem \ref{th:codim3},
there exists an open region of parameter space at which (\ref{eq:f}) has $K$ attracting $\cS[k]$-cycles.
If the $\cS[k]$-cycles are attracting on the curve, then the curve lies in the interior of this region.
In view of the scaling law, $\ee \sim \varphi(K) \lambda_2(0)^{-K}$, we should expect that this region is narrow and tubular in shape.

\subsubsection*{Newhouse regions}

It is instructive to compare the results of \S\ref{sec:codimperturb3} with Newhouse regions of smooth maps \cite{PaTa93,GoSh96}.
Let $f_{\ee}$ be a smooth map on $\mathbb{R}^2$ that varies continuously with a parameter $\ee$.
Suppose $f_0$ has a saddle-type periodic solution
for which the stable and unstable manifolds of the solution have a homoclinic tangency.
Then generically there exists a sequence of intervals $\left( \ee_k^-,\ee_k^+ \right)$,
with $\ee_k^{\pm} \to 0$ as $k \to \infty$,
within which $f_{\ee}$ has a periodic solution of period $n_k$,
where $n_k$ increases linearly with $k$,
and the periodic solutions are attracting if $|\lambda_1 \lambda_2| < 1$,
where $\lambda_1$ and $\lambda_2$ are the stability multipliers associated with the saddle-type periodic solution \cite{GaSi72,GaSi73}\removableFootnote{
Some particular cases:
(i) dyns near saddle-focus equilibrium \cite{GlSp84},
(ii) dyns near HC tangency of periodic orbit \cite{GaWa87},
(iii) dyns near Silnikov HC orbit \cite{HiKn93}.
}.
In comparison, at the codimension-three points of \S\ref{sec:codimperturb3}
there is a saddle-type periodic solution with a coincident homoclinic connection,
and by part (\ref{it:lambda12Z}) of Theorem \ref{th:codim3} we must have $\lambda_1 \lambda_2 = 1$.

The values $\ee_k^{\pm}$ exhibit the same scaling law as that described in \S\ref{sub:scaling3}.
Specifically $\ee_k^{\pm}$ scales with $\lambda_2^{-k}$, where $\lambda_2$ is the unstable
stability multiplier 	of the saddle-type periodic solution \cite{Ro83,CuJo82}.
However, the intervals $\left( \ee_k^-,\ee_k^+ \right)$ do not overlap.
The map $f_{\ee}$ has infinitely many attractors for a dense set of values of $\ee$ near zero
because there is a fractal structure of subsidiary bifurcation sequences \cite{Ro83,GoBa02}.
In our case, intervals of values of $\ee$ at which there exist attracting $\cS[k]$-cycles do overlap,
and (\ref{eq:f}) may have infinitely many attractors at $\ee = 0$.
The scenario of \S\ref{sec:codimperturb3} is reminiscent of an invertible,
continuous, piecewise-smooth map of a square given in \cite{GaTr83}.
This map consists of three pieces (two affine pieces and one nonlinear piece),
has a saddle-type fixed point (with stability multipliers that multiply to $1$)
whose stable and unstable manifolds have a homoclinic tangency.
The map has an infinite sequence of attracting periodic solutions
that converges to a homoclinic connection.

\subsubsection*{Outlook}

It remains to determine exactly what invariant sets are created
in the border-collision bifurcations $\ee_K$,
extend the results to the $N$-dimensional border-collision normal form,
and explore the phenomena in the context of grazing-sliding bifurcations \cite{GlJe12,GlKo12}.
Sequences of saddle-type solutions, which were considered in detail in \cite{Si14},
were not looked at here.
The stable manifolds of these periodic solutions can correspond to the
boundaries of the basins of attraction of the $\cS[k]$-cycles.
Given a base sequence $\cX$,
it remains to determine for which sequences $\cY$
the map (\ref{eq:f}) is able to exhibit infinitely many stable $\cS[k]$-cycles.
For each of the examples considered in \S\ref{sub:scaling3} and \S\ref{sub:examples4}, and for all $k$,
$\cS[k]$ is a ``rotational symbol sequence'' \cite{SiMe10}.
Such sequences relate to rigid rotation on a circle,
suggesting that the choices for $\cY$ may be relatively limited.

\appendix

%----------------------------------------------------------------------
\section{A derivation of equation (\ref{eq:xSkj})}
\label{app:eigenvector}

Here we suppose that $[0,1]^{\sf T}$ is an eigenvector of $M_{\cX}$ and derive (\ref{eq:xSkj}).
The approach here parallels that used in the proof of Theorem \ref{th:codim4},
and indeed the formulas that arise are different only in that $x$ and $y$ are switched, in a sense
(here $M_{\cX}$ is lower triangular, whereas in \S\ref{sub:codim4} $M_{\cX}$ is upper triangular).
In view of this switch it is necessary to apply different logical arguments
to the resulting formulas.

If $[0,1]^{\sf T}$ is an eigenvector of $M_{\cX}$
then $M_{\cX} = \left[ \begin{array}{cc} 1 & 0 \\ m_{21} & 1 \end{array} \right]$,
for some $m_{21} \ne 0$ ($m_{21}$ is nonzero because the geometric multiplicity is $1$).
We can therefore write
\begin{equation}
f^{\cX}(x,y) = \left[ \begin{array}{cc} 1 & 0 \\ m_{21} & 1 \end{array} \right]
\left[ \begin{array}{c} x \\ y \end{array} \right]
+ \left[ \begin{array}{c} \rho_1 \\ \rho_2 \end{array} \right] \;,
\label{eq:fX}
\end{equation}
for some $\rho_1, \rho_2 \in \mathbb{R}$.
It follows that powers of $f^{\cX}$ are given by
\begin{equation}
f^{\cX^k}(x,y) = \left[ \begin{array}{cc} 1 & 0 \\ m_{21} k & 1 \end{array} \right]
\left[ \begin{array}{c} x \\ y \end{array} \right]
+ \left[ \begin{array}{c} \rho_1 k \\
\frac{\rho_1 m_{21}}{2} k (k-1) + \rho_2 k \end{array} \right] \;.
\label{eq:fXk}
\end{equation}
We also write
\begin{equation}
f^{\cY}(x,y) = \left[ \begin{array}{cc}
\gamma_{11} & \gamma_{12} \\
\gamma_{21} & \gamma_{22} \end{array} \right]
\left[ \begin{array}{c} x \\ y \end{array} \right]
+ \left[ \begin{array}{c} \sigma_1 \\ \sigma_2 \end{array} \right] \;,
\label{eq:fY}
\end{equation}
for some $\gamma_{ij}, \sigma_1, \sigma_2 \in \mathbb{R}$.
By composing (\ref{eq:fXk}) and (\ref{eq:fY}) we obtain
\begin{equation}
f^{\cS[k]}(x,y) = \left[ \begin{array}{cc}
\gamma_{12} m_{21} k + \gamma_{11} & \gamma_{12} \\
\gamma_{22} m_{21} k + \gamma_{21} & \gamma_{22} 
\end{array} \right]
\left[ \begin{array}{c} x \\ y \end{array} \right]
+ \left[ \begin{array}{c}
\gamma_{11} \rho_1 k + \frac{\gamma_{12} \rho_1 m_{21}}{2} k (k-1) + \gamma_{12} \rho_2 k + \sigma_1 \\
\gamma_{21} \rho_1 k + \frac{\gamma_{22} \rho_1 m_{21}}{2} k (k-1) + \gamma_{22} \rho_2 k + \sigma_2
\end{array} \right] \;.
\label{eq:fSk}
\end{equation}
The matrix part of (\ref{eq:fSk}) is $M_{\cS[k]}$.
Hence ${\rm trace} \left( M_{\cS[k]} \right) = \gamma_{12} m_{21} k + \gamma_{11} + \gamma_{22}$.
But $\cS[k]$-cycles are assumed to be stable for arbitrarily large values of $k$,
hence ${\rm trace} \left( M_{\cS[k]} \right) \not\to \infty$ as $k \to \infty$,
see (\ref{eq:stabConditionSN})-(\ref{eq:stabConditionNS}).
Therefore we must have $\gamma_{12} = 0$.
Consequently $M_{\cS[k]}$ is a lower triangular matrix
with diagonal elements $\gamma_{11}$ and $\gamma_{22}$,
and these are eigenvalues of $M_{\cS[k]}$.
Since $\cS[k]$-cycles are assumed to be unique,
$M_{\cS[k]}$ cannot have a unit eigenvalue (see \S\ref{sec:backg}),
thus $\gamma_{11}, \gamma_{22} \ne 1$.

The unique fixed point of $f^{\cS[k]}$ is denoted $\left( x^{\cS[k]}_0, y^{\cS[k]}_0 \right)$.
By using (\ref{eq:fSk}) with $\gamma_{12} = 0$, we obtain
$x^{\cS[k]}_0 = \frac{\rho_1 \gamma_{11}}{1 - \gamma_{11}} k + \frac{\sigma_1}{1 - \gamma_{11}}$.
By (\ref{eq:fX}), $x^{\cS[k]}_{(j+1) n_{\cX}} = x^{\cS[k]}_{j n_{\cX}} + \rho_1$,
for all $j = 0,\ldots,k-1$.
This produces
\begin{equation}
x^{\cS[k]}_{j n_{\cX}} = \rho_1 j 
+ \frac{\rho_1 \gamma_{11}}{1 - \gamma_{11}} k + \frac{\sigma_1}{1 - \gamma_{11}} \;,
\nonumber
%\label{eq:xSkj2}
\end{equation}
for all $j = 0,\ldots,k$, as required.

%----------------------------------------------------------------------
\section{Proof of infinite coexistence at a codimension-four point with $n_{\cX}=3$}
\label{app:verify3}

\begin{proof}[Proof of Proposition \ref{pr:verify3}]
We explicitly compute each point of an arbitrary $\cS[k]$-cycle in $(u,v)$-coordinates\removableFootnote{
see {\sc verifyA.m}
}.
For the parameter values (\ref{eq:paramA}), we have
$M_{\cX} = \left[ \begin{array}{cc}
3-\sqrt{2} & 1-\sqrt{2} \\
-2+2\sqrt{2} & -1+\sqrt{2}
\end{array} \right]$, and thus $[1,\nu]^{\sf T}$ is an eigenvector of $M_{\cX}$ with $\nu = \sqrt{2}$.
By applying the transformation (\ref{eq:uv}) with $\nu = \sqrt{2}$ we obtain
\begin{align}
g^{\sL}(w) &= \left[ \begin{array}{cc}
0 & 1 \\
-1 & -\sqrt{2}
\end{array} \right] w +
\left[ \begin{array}{c} 1 \\ -\sqrt{2} \end{array} \right] \;, \label{eq:gLverify} \\
g^{\sR}(w) &= \left[ \begin{array}{cc}
1 & 1 \\
-\left(1+\sqrt{2}\right) & -\sqrt{2}
\end{array} \right] w +
\left[ \begin{array}{c} 1 \\ -\sqrt{2} \end{array} \right] \;, \label{eq:gRverify} \\
g^{\cX}(w) &= \left[ \begin{array}{cc}
1 & -\left(-1+\sqrt{2}\right) \\
0 & 1
\end{array} \right] w +
\left[ \begin{array}{c} -2\left(-1+\sqrt{2}\right) \\ 2-\sqrt{2} \end{array} \right] \;, \label{eq:gXverify} \\
g^{\cY}(w) &= \left[ \begin{array}{cc}
-1 & 1 \\
0 & -1
\end{array} \right] w +
\left[ \begin{array}{c} 3-\sqrt{2} \\ -\sqrt{2} \end{array} \right] \;. \label{eq:gYverify}
\end{align}
By appropriately composing (\ref{eq:gXverify}) and (\ref{eq:gYverify}) we arrive at
\begin{align}
& g^{\cX^{k-j} \cY \cX^j}(w) = \left[ \begin{array}{cc}
-1 & \left(-1+\sqrt{2}\right) k + 1 \\
0 & -1
\end{array} \right] w \nonumber \\
&+\left[ \begin{array}{c}
-\left(-4+3\sqrt{2}\right) j^2 + \frac{1}{2}\left(-4+3\sqrt{2}\right) k^2
-\sqrt{2} j + \frac{1}{2}\left(4-\sqrt{2}\right) k + 3-\sqrt{2} \\
2\left(2-\sqrt{2}\right) j - \left(2-\sqrt{2}\right) k - \sqrt{2}
\end{array} \right] \;.
\label{eq:gSkjverify}
\end{align}
The stability multipliers of the $\cS[k]$-cycle are the eigenvalues of the matrix part of (\ref{eq:gSkjverify}).
Therefore each $\cS[k]$-cycle has a repeated stability multiplier of $-1$, and is therefore stable
(assuming it is admissible).
The unique fixed point of (\ref{eq:gSkjverify}) is
\begin{equation}
w^{\cS[k]}_{3j} = \left[ \begin{array}{c}
\frac{1}{2}\left(-4+3\sqrt{2}\right) j(k-j)
- \left(-1+\sqrt{2}\right) j + \frac{\sqrt{2}}{4} k + \frac{3}{4}\left(2-\sqrt{2}\right) \\
\left(2-\sqrt{2}\right) j - \frac{1}{2}\left(2-\sqrt{2}\right) k - \frac{\sqrt{2}}{2}
\end{array} \right] \;,
\label{eq:wSk3jverify}
\end{equation}
valid for $j = 0,\ldots,k$.
(In $(u,v)$-coordinates we denote the points of an $\cS[k]$-cycle by $w^{\cS[k]}_i$, for $i = 0,\ldots,k n_{\cX} + n_{\cY}$,
where here $n_{\cX} = 3$ and $n_{\cY} = 4$.)
The image of (\ref{eq:wSk3jverify}) under $g^{\sR}$ (\ref{eq:gRverify}) is
\begin{equation}
w^{\cS[k]}_{3j+1} = \left[ \begin{array}{c}
\frac{1}{2}\left(-4+3\sqrt{2}\right) j(k-j)
+ \left(3-2\sqrt{2}\right) j + \frac{1}{4}\left(-4+3\sqrt{2}\right) k
+ \frac{5}{4}\left(2-\sqrt{2}\right) \\
-\frac{1}{2}\left(2-\sqrt{2}\right) j(k-j)
+ \left(3-2\sqrt{2}\right) j - \frac{3}{4}\left(2-\sqrt{2}\right) k
+ \frac{1}{4}\left(4-7\sqrt{2}\right)
\end{array} \right] \;,
\label{eq:wSk3jp1verify}
\end{equation}
valid for $j = 0,\ldots,k-1$,
and the image of (\ref{eq:wSk3jp1verify}) under $g^{\sR}$ is
\begin{equation}
w^{\cS[k]}_{3j+2} = \left[ \begin{array}{c}
-\left(3-2\sqrt{2}\right) j(k-j)
+ 2\left(3-2\sqrt{2}\right) j - \frac{1}{2}\left(5-3\sqrt{2}\right) k
+ \frac{3}{2}\left(3-2\sqrt{2}\right) \\
\frac{1}{2}\left(-4+3\sqrt{2}\right) j(k-j)
-\left(-5+4\sqrt{2}\right) j + \frac{1}{4}\left(-8+7\sqrt{2}\right) k
- \frac{1}{4}\left(-14+13\sqrt{2}\right)
\end{array} \right] \;,
\label{eq:wSk3jp2verify}
\end{equation}
valid for $j = 0,\ldots,k-1$.
By evaluating (\ref{eq:wSk3jverify}) at $j=k$ we obtain
\begin{equation}
w^{\cS[k]}_{3k} = \left[ \begin{array}{c}
-\frac{1}{4}\left(-4+3\sqrt{2}\right) k + \frac{3}{4}\left(2-\sqrt{2}\right) \\
\frac{1}{2}\left(2-\sqrt{2}\right) k - \frac{\sqrt{2}}{2}
\end{array} \right] \;.
\label{eq:wSk3kverify}
\end{equation}
Finally, by iterating (\ref{eq:wSk3kverify}) through the sequence
$g^{\sL}$, $g^{\sR}$, and $g^{\sL}$, we obtain
\begin{align}
w^{\cS[k]}_{3k+1} &= \left[ \begin{array}{c}
\frac{1}{2}\left(2-\sqrt{2}\right) k + \frac{1}{2}\left(2-\sqrt{2}\right) \\
-\frac{\sqrt{2}}{4} k - \frac{1}{4}\left(2+\sqrt{2}\right)
\end{array} \right] \;, \nonumber \\ %\label{eq:wSk3kp1verify} \\
w^{\cS[k]}_{3k+2} &= \left[ \begin{array}{c}
-\frac{1}{4}\left(-4+3\sqrt{2}\right) k + \frac{3}{4}\left(2-\sqrt{2}\right) \\
-\frac{1}{2}\left(-1+\sqrt{2}\right) k - \frac{1}{2}\left(-1+2\sqrt{2}\right)
\end{array} \right] \;, \nonumber \\ %\label{eq:wSk3kp2verify} \\
w^{\cS[k]}_{3k+3} &= \left[ \begin{array}{c}
-\frac{1}{2}\left(-1+\sqrt{2}\right) k + \frac{1}{2}\left(3-2\sqrt{2}\right) \\
\frac{\sqrt{2}}{4} k - \frac{1}{4}\left(-2+3\sqrt{2}\right)
\end{array} \right] \;. \nonumber %\label{eq:wSk3kp3verify}
\end{align}
Here $\cS[k] = \left( \sR^2 \sL \right)^k \sL \sR \sL^2$,
therefore the $\cS[k]$-cycle is admissible with no points on the switching manifold if
\begin{equation}
\begin{gathered}
u^{\cS[k]}_{3j} > 0 \;, \quad
u^{\cS[k]}_{3j+1} > 0 \;, \quad
u^{\cS[k]}_{3j+2} < 0 \;, \quad {\rm ~for~} j=0,\ldots,k-1 \;, \\
u^{\cS[k]}_{3k} < 0 \;, \qquad
u^{\cS[k]}_{3k+1} > 0 \;, \qquad
u^{\cS[k]}_{3k+2} < 0 \;, \qquad
u^{\cS[k]}_{3k+3} < 0 \;.
\end{gathered}
\label{eq:admissibilityverify}
\end{equation}
From (\ref{eq:wSk3kverify}) we have
$u^{\cS[k]}_{3k} = -\frac{1}{4}\left(-4+3\sqrt{2}\right) k + \frac{3}{4}\left(2-\sqrt{2}\right)$.
If $k \le 7$, $u^{\cS[k]}_{3k}$ is positive, thus by (\ref{eq:admissibilityverify}) the $\cS[k]$-cycle is virtual.
However, for all $k \ge 8$, $u^{\cS[k]}_{3k}$ is negative,
and elementary calculations suffice to show that the remaining inequalities in (\ref{eq:admissibilityverify}) also hold.
For instance, to verify $u^{\cS[k]}_{3j} > 0$, for all $j = 0,\ldots,k-1$,
we first note that by (\ref{eq:wSk3jverify}), $u^{\cS[k]}_{3j}$ is a concave down quadratic function of $j$.
Therefore its minimum value over $j = 0,\ldots,k-1$ occurs
at an endpoint of this range of values.
Simple calculations for (\ref{eq:wSk3jverify}) reveal that for any $k \ge 1$,
$u^{\cS[k]}_{3j} > 0$ for both $j=0$ and $j=k-1$,
hence $u^{\cS[k]}_{3j} > 0$ for all $j = 0,\ldots,k-1$.
\end{proof}

%=====================================================================
\section{Calculations in an alternate coordinate system relating to the codimension-four points}
\label{app:altalt}
\setcounter{equation}{0}

As in Theorem \ref{th:alphaBeta},
here we suppose that when $\ee = 0$, (\ref{eq:f}) satisfies the assumptions of Theorem \ref{th:codim4}
and $\delta_{\sL}(0) = \delta_{\sR}(0) = 1$.
Then with $\beta < \alpha < 0$, for small $\ee > 0$ the eigenvalues of
$M_{\cX}(\ee)$ are $\lambda(\ee) = r(\ee) {\rm e}^{\pm {\rm i} \theta(\ee)}$,
where $r(\ee)$ and $\theta(\ee)$ satisfy (\ref{eq:rtheta}).
Let $[1,p(\ee) \pm {\rm i} q(\ee)]^{\sf T}$ denote the eigenvectors of $M_{\cX}(\ee)$.
In this appendix we let
\begin{equation}
\hat{Q}(\ee) = \left[ \begin{array}{cc} 1 & 0 \\ p(\ee) & q(\ee) \end{array} \right] \;,
\label{eq:Qhat}
\end{equation}
and work in the alternate coordinate system
\begin{equation}
\left[ \begin{array}{c} \hat{u} \\ \hat{v} \end{array} \right] =
\hat{Q}^{-1}(\ee) 
\left[ \begin{array}{c} x \\ y \end{array} \right] \;.
\label{eq:uvHat}
\end{equation}
Note, $\hat{Q}(\ee)$ depends on $\ee$ and is well-defined for small $\ee > 0$,
whereas $Q$ (\ref{eq:Q}) is independent of $\ee$.

%----------------------------------------------------------------------
\subsection{The determinant and trace of $M_{\cS[k]}$}
\label{sub:detTraceMSk}

%Here we show that if the conditions of Theorem \ref{th:alphaBeta} are satisfied and $\beta < \alpha < 0$,
%then for all sufficiently small $\ee > 0$, and all $k \in \mathbb{Z}$, we have
%\begin{align}
%\det \left( M_{\cS[k]}(\ee) \right) &= r^{2k}(\ee) \left( 1 + \cO(\ee) \right) \;, \label{eq:detMSk4} \\
%{\rm trace} \left( M_{\cS[k]}(\ee) \right) &= -r^k(\ee) \left( 2 \cos(k\theta(\ee))
%+ \frac{\gamma_{12} \sqrt{\alpha-\beta}}{\omega_{12}} \sin(k\theta(\ee)) \ee^{\frac{1}{2}}
%+ \cO(\ee) \right) \;. \label{eq:traceMSk4}
%\end{align}
%These formulas are equations (\ref{eq:detMSk}) and (\ref{eq:traceMSk}),
%and are used in \S\ref{sub:stability4} to prove Theorem \ref{th:alphaBeta}.

Here we derive equations (\ref{eq:detMSk}) and (\ref{eq:traceMSk}).
First we calculate $p(\ee)$ and $q(\ee)$ to leading order.
By assumption there exists $\nu \in \mathbb{R}$ such that 
$Q^{-1} M_{\cX}(0) Q = \left[ \begin{array}{cc} 1 & \omega_{12} \\ 0 & 1 \end{array} \right]$,
where $\omega_{12} \ne 0$ and $Q$ is given by (\ref{eq:Q}).
From this identity it follows that we must have
$M_{\cX}(0) = \left[ \begin{array}{cc} 1 - \nu \omega_{12} & \omega_{12} \\ -\nu^2 \omega_{12} & 1 + \nu \omega_{12} \end{array} \right]$.
By using this expression to match both sides of the eigenvalue equation,
$M_{\cX}(\ee) \left[ \begin{array}{cc} 1 \\ p(\ee) \pm {\rm i} q(\ee) \end{array} \right]
= \lambda(\ee) \left[ \begin{array}{cc} 1 \\ p(\ee) \pm {\rm i} q(\ee) \end{array} \right]$,
accurate to $\cO \big( \ee^{\frac{1}{2}} \big)$,
we obtain the formulas
\begin{equation}
p(\ee) = \nu + \cO(\ee) \;, \qquad
q(\ee) = \frac{\sqrt{\alpha-\beta}}{\omega_{12}} \,\ee^{\frac{1}{2}} + \cO(\ee) \;.
\label{eq:pq}
\end{equation}

Let
\begin{equation}
\hat{\Omega}(\ee) = \hat{Q}^{-1}(\ee) M_{\cX}(\ee) \hat{Q}(\ee) \;, \qquad
\hat{\Gamma}(\ee) = \hat{Q}^{-1}(\ee) M_{\cY}(\ee) \hat{Q}(\ee) \;.
%\label{eq:OmegaGammahat}
\nonumber
\end{equation}
Then
\begin{equation}
\hat{\Omega}(\ee) = r(\ee) \left[ \begin{array}{cc}
\cos(\theta(\ee)) & \sin(\theta(\ee)) \\
-\sin(\theta(\ee)) & \cos(\theta(\ee))
\end{array} \right] \;,
\label{eq:Omegahat}
\end{equation}
and therefore
\begin{equation}
\hat{\Omega}^k(\ee) = r^k(\ee) \left[ \begin{array}{cc}
\cos(k\theta(\ee)) & \sin(k\theta(\ee)) \\
-\sin(k\theta(\ee)) & \cos(k\theta(\ee))
\end{array} \right] \;.
\label{eq:Omegahatk}
\end{equation}
By evaluating $\hat{\Gamma}(\ee)$ using (\ref{eq:Q}), (\ref{eq:Qhat}) and
$Q^{-1} M_{\cY}(0) Q = \left[ \begin{array}{cc} -1 & \gamma_{12} \\ 0 & -1 \end{array} \right]$
we obtain
\begin{equation}
\hat{\Gamma}(\ee) = \left[ \begin{array}{cc}
-1 + \cO(\ee) & \gamma_{12} q(\ee) + \cO \big( \ee^{\frac{3}{2}} \big) \\
\cO \big( \ee^{\frac{3}{2}} \big) & -1 + \cO(\ee)
\end{array} \right] \;.
\label{eq:Gammahat}
\end{equation}
Finally, since $M_{\cS[k]}(\ee) = M_{\cY}(\ee) M_{\cX}^k(\ee)$,
and eigenvalues are invariant under similarity transforms, we have
$\det \left( M_{\cS[k]}(\ee) \right) =
\det \big( \hat{\Gamma}(\ee) \big) \det \big( \hat{\Omega}(\ee) \big)^k$.
By substituting (\ref{eq:Omegahatk}) and (\ref{eq:Gammahat})
into this expression we obtain (\ref{eq:detMSk}).
Similarly, using ${\rm trace} \left( M_{\cS[k]}(\ee) \right) = {\rm trace} \big( \hat{\Gamma}(\ee) \hat{\Omega}^k(\ee) \big)$
we obtain (\ref{eq:traceMSk})
with $\phi = \frac{\gamma_{12} \sqrt{\alpha-\beta}}{\omega_{12}}$,
where we have also used (\ref{eq:pq}).

%----------------------------------------------------------------------
\subsection{A derivation of equations (\ref{eq:uv0}) and (\ref{eq:uvk1nX})}
\label{sub:Omegahat}

%First, notice $\hat{u} = x$.
For the coordinates (\ref{eq:uvHat}), let $\hat{w} = (\hat{u},\hat{v})$
and for any $\cS$, let $\hat{g}^{\cS}$ denote $f^{\cS}$ in $(\hat{u},\hat{v})$-coordinates.
Write
\begin{equation}
\hat{g}^{\cX}(\hat{w}) = \hat{\Omega}(\ee) \hat{w} + \hat{F}(\ee) \;, \qquad
\hat{g}^{\cY}(\hat{w}) = \hat{\Gamma}(\ee) \hat{w} + \hat{G}(\ee) \;,
\nonumber
\end{equation}
where $\hat{\Omega}(\ee)$ and $\hat{\Gamma}(\ee)$ are given by (\ref{eq:Omegahat})
and (\ref{eq:Gammahat}), and
\begin{equation}
\hat{F}(\ee) = \left[ \begin{array}{c}
\rho_1 + \cO(\ee) \\
\ee^{-\frac{1}{2}} \left( \frac{\omega_{12} \rho_2}{\sqrt{\alpha-\beta}} +
\cO \big( \ee^{\frac{1}{2}} \big) \right)
\end{array} \right] \;, \qquad
\hat{G}(\ee) = \left[ \begin{array}{c}
\sigma_1 + \cO(\ee) \\
\ee^{-\frac{1}{2}} \left( \frac{\omega_{12} \sigma_2}{\sqrt{\alpha-\beta}} +
\cO \big( \ee^{\frac{1}{2}} \big) \right)
\end{array} \right] \;.
%\label{eq:FGhat}
\nonumber
\end{equation}
Then
\begin{equation}
\hat{w}^{\cS[k]}_0(\ee) = \left( I - \hat{\Gamma}(\ee) \hat{\Omega}^k(\ee) \right)^{-1}
\left( \hat{\Gamma}(\ee) \sum_{m=0}^{k-1} \hat{\Omega}^m(\ee) \hat{F}(\ee) + \hat{G}(\ee) \right) \;.
\label{eq:hatwSk0}
\end{equation}
Into (\ref{eq:hatwSk0}) we substitute (\ref{eq:Omegahatk}) and (\ref{eq:Gammahat}) to obtain,
\begin{equation}
\hat{w}^{\cS[k]}_0(\ee) = \left[ \begin{array}{l}
-\rho_1 \Big\{ \left( 1 + r^k(\ee) \cos(k\theta(\ee)) \right)
\sum_{m=0}^{k-1} r^m(\ee) \cos(m\theta(\ee)) \\
\qquad+\,r^k(\ee) \sin(k\theta(\ee)) \sum_{m=0}^{k-1} r^m(\ee) \sin(m\theta(\ee)) \Big\} \\
\qquad+\,\frac{\omega_{12} \rho_2}{\sqrt{\alpha-\beta}} \Big\{
r^k(\ee) \sin(k\theta(\ee)) \sum_{m=0}^{k-1} r^m(\ee) \cos(m\theta(\ee)) \\
\qquad-\,\left( 1 + r^k(\ee) \cos(k\theta(\ee)) \right) \sum_{m=0}^{k-1} r^m(\ee) \sin(m\theta(\ee)) \\
\qquad+\,\gamma_{12} q(\ee) \sum_{m=0}^{k-1} r^m(\ee) \cos(m\theta(\ee)) \Big\} \ee^{-\frac{1}{2}} \\
\qquad-\,\frac{\omega_{12} \sigma_2}{\sqrt{\alpha-\beta}}
\,r^k(\ee) \sin(k\theta(\ee)) \ee^{-\frac{1}{2}} + {\rm H.O.T.} \\
- \frac{\omega_{12} \rho_2}{\sqrt{\alpha-\beta}}
\sum_{m=0}^{k-1} r^m(\ee) \cos(m\theta(\ee)) \ee^{-\frac{1}{2}} + {\rm H.O.T.}
\end{array} \right]
\label{eq:hatwSk02}
\end{equation}
where, for brevity, we have omitted terms that do not provide a leading order contribution.
To simplify (\ref{eq:hatwSk02}) we use the identities\removableFootnote{
\begin{align}
\sum_{m=0}^{k-1} r^m \cos(m\theta) + {\rm i} r^m \sin(m\theta)
&= \sum_{m=0}^{k-1} z^m \quad {\rm ~where~} z = r {\rm e}^{{\rm i} \theta} \nonumber \\
&= \frac{1-z^k}{1-z} \nonumber \\
&= \frac{1 - r^k \cos(k\theta) - {\rm i} r^k \sin(k\theta)}
{1 - r \cos(\theta) - {\rm i} r \sin(\theta)} \nonumber \\
&= \frac{\big( 1-r\cos(\theta) \big) \left( 1 - r^k \cos(k\theta) \right)
+ r^{k+1} \sin(\theta) \sin(k\theta)}{1 - 2 r \cos(\theta) + r^2} \nonumber \\
&\quad+ {\rm i} \frac{r \sin(\theta) \left( 1 - r^k \cos(k\theta) \right)
- \big( 1 - r\cos(\theta) \big) r^k \sin(k\theta)}{1 - 2 r \cos(\theta) + r^2} \;.
\end{align}
}
\begin{align}
\sum_{m=0}^{k-1} r^m \cos(m\theta) &=
\frac{\big( 1-r \cos(\theta) \big) \left( 1 - r^k \cos(k\theta) \right)
+ r^{k+1} \sin(\theta) \sin(k\theta)}
{1 - 2 r \cos(\theta) + r^2(\ee)} \;, \nonumber \\ %\label{eq:rcosIdentity} \\
\sum_{m=0}^{k-1} r^m \sin(m\theta) &=
\frac{r \sin(\theta) \left( 1 - r^k \cos(k\theta) \right)
- \big( 1 - r \cos(\theta) \big) r^k \sin(k\theta)}
{1 - 2 r \cos(\theta) + r^2(\ee)} \;, \nonumber %\label{eq:rsinIdentity}
\end{align}
which may be derived by evaluating $\sum_{m=0}^{k-1} \left( r {\rm e}^{{\rm i} \theta} \right)^m$
with the classical formula for a truncated geometric series.
Using the formulas for $r(\ee)$ and $\theta(\ee)$ (\ref{eq:rtheta}),
we expand to obtain the following expressions
for the various pieces of (\ref{eq:hatwSk02})\removableFootnote{
Also,
\begin{align}
r^k(\ee) \cos(k\theta(\ee)) &= 1 - \frac{\alpha-\beta}{2} k^2 \ee  + \frac{\alpha}{2} k \ee
+ \cO \left( k^4 \ee^2 \right) + \cO \left( k^2 \ee^{\frac{3}{2}} \right) \;, \label{eq:cExp} \\
\sum_{m=0}^{k-1} r^m(\ee) \sin(m\theta(\ee)) &= 
\frac{\sqrt{\alpha-\beta}}{2} k(k-1) \ee^{\frac{1}{2}}
+ \cO \left( k^4 \ee^{\frac{3}{2}} \right) + \cO \left( k^2 \ee \right) \;. \label{eq:SExp}
\end{align}
}\removableFootnote{
Naively substituting (\ref{eq:SExp}) and the second equation of (\ref{eq:allExp})
gives the third and fourth equations of (\ref{eq:allExp}),
but with larger error terms that we cannot deal with.
It is necessary to use the additional formulas
\begin{align}
\left( 1 + r^k \cos(k\theta) \right) \sum_{m=0}^{k-1} r^m \cos(m\theta) 
+ r^k \sin(k\theta) \sum_{m=0}^{k-1} r^m \sin(m\theta) &=
\frac{(1-r\cos(\theta)) \left(1-r^{2k}\right) + 2 r^{k+1} \sin(\theta) \sin(k\theta)}
{1-2 r\cos(\theta)+r^2} \;, \\
r^k \sin(k\theta) \sum_{m=0}^{k-1} r^m \cos(m\theta)
- \left( 1 + r^k \cos(k\theta) \right) \sum_{m=0}^{k-1} r^m \sin(m\theta) &=
\frac{2 (1-r\cos(\theta)) r^k \sin(k\theta) - r\sin(\theta) \left(1-r^{2k}\right)}
{1-2 r\cos(\theta)+r^2} \;.
\end{align}
}
\begin{equation}
\begin{split}
r^k(\ee) \sin(k\theta(\ee)) &= \sqrt{\alpha-\beta} k \ee^{\frac{1}{2}}
+ \cO \big( k^3 \ee^{\frac{3}{2}} \big) + \cO \left( k \ee \right) \;, \\ %label{eq:sExp} \\
\sum_{m=0}^{k-1} r^m(\ee) \cos(m\theta(\ee)) &= k
+ \cO \left( k^3 \ee \right) + \cO \big( k \ee^{\frac{1}{2}} \big) \;, \\ %\label{eq:CExp} \\
\left( 1 + r^k(\ee) \cos(k\theta(\ee)) \right)
\sum_{m=0}^{k-1} r^m(\ee) \cos(m\theta(\ee)) \\
+\,r^k(\ee) \sin(k\theta(\ee)) \sum_{m=0}^{k-1} r^m(\ee) \sin(m\theta(\ee)) &= 2 k
+ \cO \left( k^3 \ee \right) + \cO \big( k \ee^{\frac{1}{2}} \big) \;, \\ %\label{eq:cCsSExp} \\
r^k(\ee) \sin(k\theta(\ee)) \sum_{m=0}^{k-1} r^m(\ee) \cos(m\theta(\ee)) \\
-\,\left( 1 + r^k(\ee) \cos(k\theta(\ee)) \right) \sum_{m=0}^{k-1} r^m(\ee) \sin(m\theta(\ee)) &=
\sqrt{\alpha-\beta} k \ee^{\frac{1}{2}}
+ \cO \big( k^3 \ee^{\frac{3}{2}} \big) + \cO \left( k \ee \right) \;. %\label{eq:sCcSExp}
\end{split}
\label{eq:allExp}
\end{equation}
Into (\ref{eq:hatwSk02}) we substitute (\ref{eq:pq}), (\ref{eq:allExp}) and
$\det \left( I - M_{\cS[k]}(\ee) \right) = 4 + \cO \left( k^2 \ee \right)$ to arrive at
\begin{equation}
\hat{w}^{\cS[k]}_0(\ee) = \left[ \begin{array}{c}
k \left( -\frac{\rho_1}{2} + \frac{ \left( \gamma_{12} + \omega_{12} \right) \rho_2}{4}
- \frac{\omega_{12} \sigma_2}{4} + \cO \left( k^2 \ee \right) + \cO \big( \ee^{\frac{1}{2}} \big)
+ \cO \left( k^{-1} \right) \right) \\
k \ee^{-\frac{1}{2}} \left( -\frac{\omega_{12} \rho_2}{2 \sqrt{\alpha-\beta}}
+ \cO \left( k^2 \ee \right) + \cO \big( \ee^{\frac{1}{2}} \big)
+ \cO \left( k^{-1} \right) \right)
\end{array} \right] \;.
\label{eq:hatwSk03}
\end{equation}
To (\ref{eq:hatwSk03})
we apply the inverse of $\hat{g}^{\cY}$ and then the inverse of $\hat{g}^{\cX}$ to obtain
%In a similar fashion, by evaluating
%\begin{equation}
%\hat{w}^{\cS[k]}_{(k-1) n_{\cX}}(\ee) =
%\left( I - \hat{\Omega}^{k-1}(\ee) \hat{\Gamma}(\ee) \hat{\Omega}(\ee) \right)^{-1}
%\left( \hat{\Omega}^{k-1}(\ee) \hat{\Gamma}(\ee) \hat{F}(\ee)
%+ \sum_{m=0}^{k-2} \hat{\Omega}^m(\ee) \hat{F}(\ee)
%+ \hat{\Omega}^{k-1}(\ee) \hat{G}(\ee) \right) \;,
%\nonumber
%\end{equation}
%we obtain
\begin{equation}
\hat{w}^{\cS[k]}_{(k-1) n_{\cX}}(\ee) =
\left[ \begin{array}{c}
k \left( \frac{\rho_1}{2} + \frac{\left( \gamma_{12} - 3 \omega_{12} \right) \rho_2}{4}
+ \frac{\omega_{12} \sigma_2}{4}
+ \cO \left( k^2 \ee \right) + \cO \big( \ee^{\frac{1}{2}} \big)
+ \cO \left( k^{-1} \right) \right) \\
k \ee^{-\frac{1}{2}} \left( \frac{\omega_{12} \rho_2}{2 \sqrt{\alpha-\beta}}
+ \cO \left( k^2 \ee \right) + \cO \big( \ee^{\frac{1}{2}} \big)
+ \cO \left( k^{-1} \right) \right)
\end{array} \right] \;.
\label{eq:hatwSk1nX}
\end{equation}
Finally, multiplying the right-hand sides of (\ref{eq:hatwSk03}) and (\ref{eq:hatwSk1nX})
by $Q^{-1} \hat{Q}(\ee)$ (to change from $(\hat{u},\hat{v})$-coordinates to
$(u,v)$-coordinates) we obtain (\ref{eq:uv0}) and (\ref{eq:uvk1nX}) as required.

%=====================================================================
\section{Calculations performed in $m$-dependent coordinates}
\label{app:mDependent}
\setcounter{equation}{0}

%----------------------------------------------------------------------
\subsection{A calculation of the $\cX$-cycle}
\label{sub:Xcycle}

For small $\ee > 0$, $w^{\cX}_m(\ee)$ is the unique fixed point of $g^{\cX^{(m)}}$,
where $\cX^{(m)}$ denotes the $m^{\rm th}$ left shift permutation of $\cX$.

When $\ee = 0$, $g^{\cX}$ is given by (\ref{eq:gX}) and
$\left( g^{\cX_{m-1}} \circ \cdots \circ g^{\cX_0} \right)(w)$ is given by (\ref{eq:gXtruncated}).
By combining these appropriately we obtain
\begin{align}
g^{\cX^{(m)}}(w) &= \left[ \begin{array}{cc}
1 - \frac{\psi_{11m} \psi_{21m} \omega_{12}}{\det(\Psi_m)} & \frac{\psi_{11m}^2 \omega_{12}}{\det(\Psi_m)} \\
-\frac{\psi_{21m}^2 \omega_{12}}{\det(\Psi_m)} & 1 + \frac{\psi_{11m} \psi_{21m} \omega_{12}}{\det(\Psi_m)}
\end{array} \right] w \nonumber \\
&~+\left[ \begin{array}{c}
\psi_{11m} \rho_1 + \psi_{12m} \rho_2 + \frac{\psi_{11m} \psi_{21m} \omega_{12} \chi_{1m}}{\det(\Psi_m)}
- \frac{\psi_{11m}^2 \omega_{12} \chi_{2m}}{\det(\Psi_m)} \\
\psi_{21m} \rho_1 + \psi_{22m} \rho_2 + \frac{\psi_{21m}^2 \omega_{12} \chi_{1m}}{\det(\Psi_m)}
- \frac{\psi_{11m} \psi_{21m} \omega_{12} \chi_{2m}}{\det(\Psi_m)}
\end{array} \right] \;,
\label{eq:gXm}
\end{align}
where $\Psi_m$ denotes the matrix part of (\ref{eq:gXtruncated}).
Note, $\det(\Psi_m) \ne 0$ because $M_{\cX}$ is nonsingular
and $\Psi_m$ is a truncation of $M_{\cX}$ in $(u,v)$-coordinates.

Since $[\psi_{11m},\psi_{21m}]$ is an eigenvector for the matrix part of (\ref{eq:gXm}),
it is useful to consider the $m$-dependent coordinate change\removableFootnote{
I could use $s$, $t$, $z$ and $h$ instead of $\tilde{u}$, $\tilde{v}$, $\tilde{w}$ and $\tilde{g}$?
}
\begin{equation}
\left[ \begin{array}{c} \tilde{u} \\ \tilde{v} \end{array} \right] = 
\left[ \begin{array}{cc} 1 & 0 \\ -\frac{\psi_{21m}}{\psi_{11m}} & 1 \end{array} \right]
\left[ \begin{array}{c} u \\ v \end{array} \right] \;.
\label{eq:uvTilde}
\end{equation}
We also let $\tilde{w} = (\tilde{u},\tilde{v})$
and $\tilde{g}^{\cS}$ denote $g^{\cS}$ in $(\tilde{u},\tilde{v})$-coordinates.
From (\ref{eq:gXm}) and (\ref{eq:uvTilde})
\begin{equation}
\tilde{g}^{\cX^{(m)}}(\tilde{w}) =
\left[ \begin{array}{cc} 1 & \frac{\psi_{11m}^2 \omega_{12}}{\det(\Psi_m)} \\ 0 & 1 \end{array} \right] \tilde{w}
+ \left[ \begin{array}{c} \tilde{\rho}_{1m} \\ \frac{\det(\Psi_m) \rho_2}{\psi_{11m}} \end{array} \right] \;,
\label{eq:gTildeXm}
\end{equation}
where $\tilde{\rho}_{1m} = \psi_{11m} \rho_1 + \psi_{12m} \rho_2
+ \frac{\psi_{11m} \psi_{21m} \omega_{12} \chi_{1m}}{\det(\Psi_m)}
- \frac{\psi_{11m}^2 \omega_{12} \chi_{2m}}{\det(\Psi_m)}$.

For $\ee > 0$, the coefficients of (\ref{eq:gTildeXm}) vary smoothly with $\ee$
because $A_L(\ee)$ and $A_R(\ee)$ vary smoothly with $\ee$.
We need to investigate how the coefficients of the matrix part of $\tilde{g}^{\cX^{(m)}}$ vary with $\ee$,
but we are not concerned with the constant part of $\tilde{g}^{\cX^{(m)}}$, so we write
\begin{equation}
\tilde{g}^{\cX^{(m)}}(\tilde{w}) =
\left( \left[ \begin{array}{cc}
1 + \xi_{11m} \ee & \frac{\psi_{11m}^2 \omega_{12}}{\det(\Psi_m)} + \xi_{12m} \ee \\
\xi_{21m} \ee & 1 + \xi_{22m} \ee \end{array} \right]
+ \cO \left( \ee^2 \right) \right) \tilde{w}
+ \left[ \begin{array}{c} \tilde{\rho}_{1m} \\ \frac{\det(\Psi_m) \rho_2}{\psi_{11m}} \end{array} \right] + \cO(\ee) \;,
\label{eq:gTildeXm2}
\end{equation}
for some $m$-dependent coefficients $\xi_{ijm}$.
%where the matrix and constant parts of $\tilde{g}^{\cX^{(m)}}$ have been given
%to different orders in $\ee$.

The spectrum of $M_{\cX^{(m)}}$ is independent of $m$
(because changing $m$ changes only the cyclical order in which $A_L$ and $A_R$ are multiplied),
and the matrix part of (\ref{eq:gTildeXm2}) is similar to $M_{\cX^{(m)}}$,
therefore the spectrum of the matrix part of (\ref{eq:gTildeXm2}) is the same as that of $M_{\cX}(\ee)$.
Hence,
\begin{align}
\det \left( M_{\cX}(\ee) \right) &= 
1 + \left( \xi_{11m} + \xi_{22m}
- \frac{\psi_{11m}^2 \omega_{12} \xi_{21m}}{\det(\Psi_m)} \right) \ee + \cO \left( \ee^2 \right) \;, \\
{\rm trace} \left( M_{\cX}(\ee) \right) &= 
2 + \left( \xi_{11m} + \xi_{22m} \right) \ee + \cO \left( \ee^2 \right) \;.
\end{align}
Thus by (\ref{eq:detMX}) and (\ref{eq:traceMX}),
\begin{equation}
\alpha = \xi_{11m} + \xi_{22m} - \frac{\psi_{11m}^2 \omega_{12} \xi_{21m}}{\det(\Psi_m)} \;, \qquad
\beta = \xi_{11m} + \xi_{22m} \;.
\label{eq:alphabeta}
\end{equation}
If $\alpha \ne \beta$, then for small $\ee \ne 0$,
(\ref{eq:gTildeXm2}) has the unique fixed point
\begin{equation}
\tilde{w}^{\cX}_m(\ee) = \frac{1}{\ee} \left( \frac{1}{\alpha - \beta} \left[ \begin{array}{c}
\psi_{11m} \omega_{12} \rho_2 \\ 0 \end{array} \right] + \cO(\ee) \right) \;,
%\label{eq:wTildeXm}
\nonumber
\end{equation}
where we have used (\ref{eq:alphabeta}) to substitute
$\alpha - \beta = -\frac{\psi_{11m}^2 \omega_{12} \xi_{21m}}{\det(\Psi_m)}$.
Since $\tilde{u} = u$, we have
$u^{\cX}_m(\ee) = \frac{1}{\ee} \left( \frac{\psi_{11m}
\omega_{12} \rho_2}{\alpha - \beta} + \cO(\ee) \right)$,
which is (\ref{eq:uXm}).

%----------------------------------------------------------------------
\subsection{Additional calculations for the proof of Lemma \ref{le:lowerBound}}
\label{sub:lowerBound}

From (\ref{eq:gTildeXm}) and (\ref{eq:alphabeta}) we obtain
\begin{align}
\tilde{u}^{\cS[k]}_{(j+2) n_{\cX} + m}
- 2 \tilde{u}^{\cS[k]}_{(j+1) n_{\cX} + m}
+ \tilde{u}^{\cS[k]}_{j n_{\cX} + m} &=
\left( -(\alpha-\beta) \ee + \cO \left( \ee^2 \right) \right) \tilde{u}^{\cS[k]}_{j n_{\cX} + m} \nonumber \\
&\quad+ \left( \frac{\psi_{11m}^2 \omega_{12} \beta}{\det(\Psi_m)} \ee + \cO \left( \ee^2 \right) \right) \tilde{v}^{\cS[k]}_{j n_{\cX} + m}
+ \left( \psi_{11m} \omega_{12} \rho_2 + \cO(\ee) \right) \;,
\label{eq:concavityu}
\end{align}
for $j = 0,\ldots,k-3$,
where the right-hand side of (\ref{eq:concavityu}) is an affine function of
$u^{\cS[k]}_{j n_{\cX} + m}$ and $v^{\cS[k]}_{j n_{\cX} + m}$
and the coefficients of the linear terms are given to a different order in $\ee$ than the constant term.
Similarly we have
\begin{equation}
\tilde{v}^{\cS[k]}_{(j+1) n_{\cX} + m} - \tilde{v}^{\cS[k]}_{j n_{\cX} + m} = 
\left( \xi_{21m} \ee + \cO \left( \ee^2 \right) \right) \tilde{u}^{\cS[k]}_{j n_{\cX} + m}
+ \left( \xi_{22m} \ee + \cO \left( \ee^2 \right) \right) \tilde{v}^{\cS[k]}_{j n_{\cX} + m}
+ \frac{\det(\Psi_m) \rho_2}{\psi_{11m}} + \cO(\ee) \;.
\label{eq:monotonicityv}
\end{equation}
Without loss of generality, and for simplicity, suppose $\cX_m = \sR$.
Then $\psi_{11m} \omega_{12} \rho_2 < 0$, by (\ref{eq:signkappa11omega12rho2}).
Since $\tilde{w}^{\cS[k]}_m(\ee)$ and $\tilde{w}^{\cS[k]}_{(k-1) n_{\cX} + m}(\ee)$
are admissible for all $0 \le \ee \le \Delta k^{-2}$,
and $\tilde{v}^{\cS[k]}_m(\ee)$ and $\tilde{v}^{\cS[k]}_{(k-1) n_{\cX} + m}(\ee)$
are affine functions of $k$,
there exists $C_m \in \mathbb{R}$ such that 
for all $0 \le \ee \le \Delta k^{-2}$,
$\tilde{w}^{\cS[k]}_m(\ee),\tilde{w}^{\cS[k]}_{(k-1) n_{\cX} + m}(\ee) \in \Sigma_m$,
where
\begin{equation}
\Sigma_m = \left\{ (\tilde{u},\tilde{v}) ~\big|~
\tilde{u} > 0 ,\, |\tilde{v}| < C_m k \right\} \;.
%\label{eq:Sigmam}
\nonumber
\end{equation}
Within $\Sigma_m$,
the $\tilde{u}^{\cS[k]}_{j n_{\cX} + m}$-term of (\ref{eq:concavityu}) is negative
for sufficiently small values of $\ee > 0$,
because $\alpha > \beta$,
The $\tilde{v}^{\cS[k]}_{j n_{\cX} + m}$-term of (\ref{eq:monotonicityv}) is $\cO \left( k \ee \right)$,
and thus dominated by the constant term, which is negative for small $\ee > 0$.
Hence, for sufficiently small $\ee > 0$,
in $\Sigma_m$ the values $\tilde{u}^{\cS[k]}_{j n_{\cX} + m}$ are concave down.
Similarly, in $\Sigma_m$
the values $\tilde{v}^{\cS[k]}_{j n_{\cX} + m}$ are increasing if $\rho_2 > 0$
and decreasing if $\rho_2 < 0$.
To see this we recall, $\alpha - \beta = -\frac{\psi_{11m}^2 \omega_{12} \xi_{21m}}{\det(\Psi_m)}$ (\ref{eq:alphabeta}),
%which implies ${\rm sgn} \left( \xi_{21m} \right) = {\rm sgn} \left( \tilde{\rho}_2 \right)$,
hence the $\tilde{u}^{\cS[k]}_{j n_{\cX} + m}$-term and constant terms of (\ref{eq:monotonicityv})
have the same sign in $\Sigma_m$.

Let $U_m(\ee) = \min_{j=0,\ldots,k-1} \left[ \tilde{u}^{\cS[k]}_{j n_{\cX} + m} \right]$.
Since $U_m$ is a continuous function of $\ee$,
$U_m(\ee)$ must be positive over the interval
$0 \le \ee \le \Delta k^{-2}$
(because otherwise there exists $0 \le \ee \le \Delta k^{-2}$
for which $U_m(\ee) = 0$ and $\tilde{w}_{j n_{\cX} + m} \in \Sigma_m$, for each $j$,
which is not possible in view of the concavity condition (\ref{eq:concavityu})).
Therefore the sign of $\tilde{u}^{\cS[k]}_{j n_{\cX} + m}(\ee)$ is constant, for all $j = 0,\ldots,k-1$,
as required.

\end{document}